\documentclass[11pt,reqno]{amsart}
\usepackage{amsmath,amssymb,graphicx,mathrsfs,amsopn,stmaryrd,amsthm,amscd,bm,extarrows}
\usepackage{newtxtext}
\usepackage[dvipsnames]{xcolor}
\usepackage{ytableau}
\usepackage[colorlinks=true,citecolor=blue,linkcolor=blue]{hyperref}
\usepackage[alphabetic,nobysame]{amsrefs}
\renewcommand{\MR}[1]{} \renewcommand{\PrintDOI}[1]{}
\newcommand{\arxiv}[1]{\href{http://arxiv.org/abs/#1}{\sf arXiv:\nolinkurl{#1}}}
\usepackage{youngtab}
\usepackage[english]{babel}
\usepackage[mathscr]{euscript}
\usepackage[margin=1in]{geometry}
\DeclareFontFamily{OT1}{pzc}{}
\DeclareFontShape{OT1}{pzc}{m}{it}{ <-> s*[1.0] pzcmi7t }{}
\DeclareMathAlphabet{\mathpzc}{OT1}{pzc}{m}{it}

\usepackage{collectbox}


\usepackage{dsfont}
\usepackage{tikz}    
\usetikzlibrary{arrows,matrix,decorations.markings,shapes.geometric}   
\usetikzlibrary{decorations.pathreplacing}

\usepackage{xy}
\allowdisplaybreaks
\xyoption{all}
\numberwithin{equation}{section}
\newtheorem{thm}{Theorem}[section]
\newtheorem{prop}[thm]{Proposition}
\newtheorem{lem}[thm]{Lemma}
\newtheorem{cor}[thm]{Corollary}

\newtheorem{alphatheorem}{Theorem}

\theoremstyle{definition} 
\newtheorem{eg}[thm]{Example}
\newtheorem{dfn}[thm]{Definition}
\theoremstyle{remark}
\newtheorem{rem}[thm]{Remark}

\newcommand{\beq}{\begin{equation}}
\newcommand{\eeq}{\end{equation}}
\newcommand{\be}{\begin{equation*}}
\newcommand{\ee}{\end{equation*}}

\newcommand{\bC}{\mathbb{C}}

\newcommand{\bZ}{\mathbb{Z}}

\newcommand{\bN}{\mathbb{N}}

\newcommand{\mc}{\mathcal}

\newcommand{\cM}{\mathcal{M}}

\newcommand{\sfD}{\mathsf{D}}
\newcommand{\sfE}{\mathsf{E}}
\newcommand{\sfF}{\mathsf{F}}

\newcommand{\sfT}{\mathsf{T}}

\newcommand{\bA}{\mathbf{A}}
\newcommand{\bB}{\mathbf{B}}

\newcommand{\mfk}{\mathfrak}
\newcommand{\g}{\mathfrak{g}}
\newcommand{\h}{\mathfrak{h}}

\newcommand{\gl}{\mathfrak{gl}}
\newcommand{\fkn}{\mathfrak{n}}

\newcommand{\fkm}{\mathfrak{m}}
\newcommand{\fkM}{\mathfrak{M}}
\newcommand{\fkb}{\mathfrak{b}}

\newcommand{\fkS}{\mathfrak{S}}

\newcommand{\fkc}{\mathfrak{c}}

\newcommand{\fks}{\mathfrak{s}}
\newcommand{\fkh}{\mathfrak{h}}
\newcommand{\fkp}{\mathfrak{p}}

\newcommand{\rY}{\mathrm{Y}}

\newcommand{\rU}{\mathrm{U}}

\newcommand{\gr}{{\mathrm{gr}}}

\newcommand{\tl}{\tilde}

\newcommand{\gge}{\geqslant}
\newcommand{\lle}{\leqslant}
\newcommand{\la}{\lambda}

\newcommand{\bla}{\bm\lambda}

\newcommand{\ka}{\kappa}

\newcommand{\I}{{\mathbb I}}

\newcommand{\ve}{\varepsilon}

\newcommand{\Y}{{\mathrm{Y}_{m|n}(\sigma)}}
\newcommand{\Yz}{{\mathrm{Y}_{m|n}^0}}
\newcommand{\Yp}{{\mathrm{Y}_{m|n}^+(\sigma)}}
\newcommand{\Ym}{{\mathrm{Y}_{m|n}^-(\sigma)}}

\newcommand{\ovl}[1]{\overline{#1}}
\newcommand{\paa}[1]{|{#1}|}
\newcommand{\tp}{\operatorname{pa}}
\newcommand{\str}{\operatorname{str}}
\newcommand{\col}{\operatorname{col}}
\newcommand{\row}{\operatorname{row}}
\newcommand{\pr}{\operatorname{pr}}
\newcommand{\ev}{\operatorname{ev}}
\newcommand{\HC}{\operatorname{HC}}
\newcommand{\hc}{\operatorname{hc}}
\newcommand{\dprime}{{\prime\prime}}
\newcommand{\trprime}{{\prime\prime\prime}}

\begin{document}
\pagestyle{myheadings}
\setcounter{page}{1}

\title[Shifted super Yangians and finite $W$-superalgebras]{Representations of shifted super Yangians and finite $W$-superalgebras of type A}

\author{Kang Lu}
\address{
Shenzhen International Center for Mathematics and Department of Mathematics, Southern University
of Science and Technology, Shenzhen, China
}\email{(Lu) luk@sustech.edu.cn}
\author{Yung-Ning Peng}
\address{Department of Mathematics, National Cheng Kung University, Tainan City, 70101, Taiwan, and National Center for Theoretical Sciences, Mathematics Division, Taipei City, 10617, Taiwan}\email{(Peng) ynp@gs.ncku.edu.tw}


\begin{abstract}
In this article, we study the representation theory of shifted super Yangians and finite $W$-superalgebras of type A. A criterion for the finite dimensionality of irreducible modules is obtained in the standard parity case. Furthermore, we provide an explicit Gelfand-Tsetlin character formula for Verma modules of finite $W$-superalgebras. As an application, we show that the centers of the finite $W$-superalgebras associated to any even nilpotent elements belonging to the same general linear Lie superalgebra are all isomorphic to the center of the universal enveloping superalgebra.
\end{abstract}
	
\maketitle

\setcounter{tocdepth}{1}
\tableofcontents

\thispagestyle{empty}
\section{Introduction} 
Finite $W$-algebras are certain associative algebras determined by nilpotent elements in complex semisimple Lie algebras. Geometrically, they can be realized as quantizations of the Slodowy slices to the corresponding nilpotent orbits.
The study of finite $W$-algebras originated from Kostant's celebrated work \cite{Ko78}, in which the finite $W$-algebras associated to principal (also called regular) nilpotent orbits are proved to be isomorphic to the center of the universal enveloping algebras. The modern terminology and the name itself were later introduced by Premet \cites{Pr95, Pr02}. For a comprehensive history and overview of applications, we refer to the survey articles by Losev \cite{Losev10} and Wang \cite{Wa11}.

Under mild conditions with specific technical requirements depending on the type, one can define finite $W$-superalgebras associated to nilpotent elements in both basic classical and strange Lie superalgebras.
In recent years, finite $W$-superalgebras and their representations have been extensively studied by mathematicians and physicists from different points of view and using various approaches \cites{CC24,CW26,PS16,PS17,WZ09,Zha14,ZS16,ZS25}.

The algebraic structure of finite $W$-algebras is significantly more intricate than that of universal enveloping algebras. 
A cornerstone is the realization of a finite $W$-algebra associated to an arbitrary nilpotent element in a general linear Lie algebra in terms of the {\em shifted Yangian}, as established by Brundan-Kleshchev \cites{BK05, BK06}. Their work generalized earlier results \cites{Ko78, RS99} from principal or rectangular nilpotent elements to the arbitrary case.
By studying shifted Yangians, the representation theory of finite $W$-algebras is then systematically developed in \cite{BK08}, leading to fundamental results such as highest weight theory, the classification of finite-dimensional simple modules, and a Gelfand-Tsetlin character formula for Verma modules. 

The aforementioned isomorphism in \cite{BK06} for other types of Lie algebras was still unavailable in the literature, except for some very special cases, until a recent work \cite{LPTTW25}, in which the isomorphism is established for the case of type BC (and conjecturally for D) Lie algebras.
On the other hand, the corresponding results are relatively complete for type A Lie superalgebras.
In \cite{Pe21}, an explicit presentation of a finite $W$-superalgebra associated to an arbitrary even nilpotent element in a type A Lie superalgebra was obtained, superizing the main results of \cite{BK06} and generalizing special cases in \cites{BR03,BBG13}.
Hence, it is natural to seek further results on the representation theory of finite $W$-superalgebras. 
The goal of this article is to study the representation theory of the shifted super Yangian of type A in the spirit of \cite{BK08}. As a consequence, several corresponding results of the representation theory of finite $W$-superalgebras of type A are obtained, generalizing some earlier results in \cites{BBG13, BG19} to non-principal even nilpotent case. 

To quickly summarize our results in this introduction, we introduce some terminology; precise definitions can be found in the context.
A pyramid is, roughly speaking, a skew Young diagram $\pi$ in French style, but not allowing any box hanging in the air.
In the super-setting, each row in $\pi$, and hence each box in $\pi$, is assigned a $\bZ_2$-parity. Let $m$ and $n$ denote the number of even and odd rows, respectively, so that the height of $\pi$ is $m+n$.
A pyramid  $\pi$ records an even nilpotent element $e$ and a good $\bZ$-grading, and hence corresponds to a finite $W$-superalgebra which we denote by $W(\pi)$.
On the other hand, a pyramid $\pi$ corresponds uniquely to a triple $(\sigma, \ell, \fks)$, where $\sigma=(s_{i,j})_{1\lle i,j\lle m+n}$ is a shift matrix of size $m+n$ with entries in $\bZ_{\gge 0}$, $\ell$ is the size of the largest Jordan block of $e$, and $\fks$ is the parity sequence recording the parity of each row in $\pi$ from top to bottom. We call $\ell$ the  level of $\pi$. The matrix $\sigma$ (and the sequence $\fks$) allows one to define the shifted super Yangian $\Y$, which is a subalgebra of the super Yangian $\rY_{m|n}:=\rY(\gl_{m|n})$.
In terms of a Drinfeld type presentation, $\Y$ is an associative superalgebra generated by the following symbols
\begin{align*}
\big\{D_i^{(r)} \mid 1\lle i&\lle m+n,r>0\big\}\\ \cup\,
\big\{E_j^{(r)}&\mid 1\lle j< m+n,r>s_{j,j+1}\big\}\\
&\cup\big\{F_j^{(r)}\mid 1\lle j< m+n,r>s_{j+1,j}\big\},
\end{align*}
subject to certain relations in Section \ref{sec:ssy}. When $\sigma=0$, or when the pyramid is of rectangular shape, $\Y$ is exactly the whole super Yangian $\rY_{m|n}$. 

Let $I_\ell$ denote the two-sided ideal of $\Y$ generated by $\big\{D_1^{(r)} \mid r> \ell - s_{1,m+n} - s_{m+n,1}\big\}$. It is proved in \cite{Pe21} that the quotient $\rY_{m|n}^\ell(\sigma):=\Y /I_\ell$ is isomorphic to $W(\pi)$ as superalgebras. One should note that $\Y$ is not closed under the restriction of the coproduct of $\rY_{m|n}$. 
Our first result is a description of the behavior of $\Y$ under the coproduct.
\begin{alphatheorem}[Theorem~\ref{thm:grownup}, Corollary~\ref{SWcop}]\label{thm:A}
The restriction of the coproduct
$\Delta:\rY_{m|n}\rightarrow \rY_{m|n} \otimes \rY_{m|n}$ 
induces a homomorphism
\be
\Delta:\rY_{m|n}(\sigma)\rightarrow \rY_{m|n}(\sigma^\prime) \otimes \rY_{m|n}(\sigma^\dprime).
\ee
Moreover, this homomorphism factors through the corresponding quotients to yield a homomorphism of finite $W$-superalgebras 
\be
\Delta_{\ell^\prime,\ell^\dprime} : W(\pi) \rightarrow W(\pi^\prime)\otimes W(\pi^\dprime).
\ee
\end{alphatheorem}
We roughly explain the notation in the theorem above.
Starting with the pyramid $\pi$, we cut it vertically along its highest column to obtain two smaller pyramids $\pi^\prime$ and $\pi^\dprime$, where $\pi^\prime$ is the left pyramid and $\pi^\dprime$ is the right pyramid. Let $\ell, \ell^\prime$, and $\ell^\dprime$ denote the levels of $\pi$, $\pi^\prime$, and $\pi^\dprime$, respectively. 
Identify $W(\pi)\cong \rY_{m|n}^\ell(\sigma)$, $W(\pi^\prime)\cong \rY_{m|n}^{\ell^\prime}(\sigma^\prime)$, and $W(\pi^\dprime)\cong \rY_{m|n}^{\ell^\dprime}(\sigma^\dprime)$, where $\sigma^\prime$ and $\sigma^\dprime$ are the corresponding shift matrices of the same height $m+n$ by allowing empty rows. Since we cut $\pi$ at its highest column, it turns out that $\sigma^\prime$ is lower-triangular and $\sigma^\dprime$ is upper-triangular. As a result, we have $\sigma=\sigma^\prime + \sigma^\dprime$ and we may identify $\Y$ as a subalgebra of both $\rY_{m|n}(\sigma^\prime)$ and $\rY_{m|n}(\sigma^\dprime)$. This homomorphism allows one to construct the tensor product of modules for shifted super Yangians and finite $W$-superalgebras, under a certain compatibility assumption related to parities.

Now we turn our attention to the representations. Using a PBW-type basis and triangular decomposition, one can define the highest weight theory for $\Y$ and hence the notion of Verma modules. Here, the role of the Cartan subalgebra is replaced by the Gelfand-Tsetlin subalgebra generated by $\big\{ D_i^{(r)}\mid 1\lle i\lle m+n, \, r\gge 1\,\big\}$.
This subalgebra is a maximal commutative (purely even) subalgebra that is contained in any shifted super Yangian $\Y$. 
It is convenient to use generating series $D_i(u)=1+\sum_{r\gge 1} D_i^{(r)}u^{-r} \in \rY_{m|n}(\sigma)[\![u^{-1}]\!]$ to describe the results. 
As usual, a highest weight vector in a $\Y$-module means a vector which is annihilated by all $E_j^{(r)}$ while each $D_i^{(r)}$ acts  on it as a scalar, say $\la_i^{(r)}$.
Thus, by collecting these scalars, one obtains an $(m+n)$-tuple of power series in $u^{-1}$ which we call an $\bm\ell$-weight:
\[
\bla=\big(\la_i(u)\big)_{1\lle i\lle m+n},
\]
where each $\la_i(u)=1+\sum_{r\gge 1} \la_i^{(r)}u^{-r}$ is a formal series in $\bC[\![u^{-1}]\!]$ with constant term 1.
Let $\mathscr P_{m|n}$ denote the set of $\bm\ell$-weights. Given an $\bm\ell$-weight $\bla\in\mathscr P_{m|n}$, let $\mc M(\sigma,\bla)$ denote the corresponding Verma module for $\Y$ and let $\mc L(\sigma,\bla)$ denote its irreducible quotient.
Our next result is a finite-dimensional criterion of $\mc L(\sigma,\bla)$ in the standard parity case.

\begin{alphatheorem} [Theorem~\ref{thm:fd}]
\label{thm:B}
Let $\fks$ be the standard parity sequence. For $\bla\in\mathscr P_{m|n}$, a finite dimensionality criterion for $\mc L(\sigma,\bla)$ in terms of Drinfeld polynomials is obtained.
\end{alphatheorem}
We should mention that such a criterion for the whole super Yangian $\rY_{m|n}$ was only established for the standard parity sequence \cites{Zh96}. For an arbitrary parity sequence, an explicit criterion is still unavailable. However, using the notion of odd reflections for super Yangians \cites{Mo22, Lu22}, one can still check the finite dimensionality of irreducible modules by recursively reducing them to the standard parity case.
Deducing an explicit criterion or a procedure checking the finite-dimensionality of $\mc L(\sigma,\bla)$ for an arbitrary parity sequence is an interesting yet non-trivial problem, which will be considered elsewhere.

Let $p_i$ denote the number of boxes in the $i$-{th} row of $\pi$. If $v$ is a highest weight vector in a $W(\pi)$-module, by a highly non-trivial technical result Theorem~\ref{thm:vanish} one deduces that there exist complex numbers $\{ a_{i,j} \mid 1\lle i\lle m+n, 1\lle j\lle p_i\}$ such that 
\be
(u-\ka_i)^{p_i}D_i(u-\ka_i)v=\prod_{j=1}^{p_i}(u+a_{i,j})v, \qquad \forall\, 1\lle i\lle m+n,
\ee
where $\kappa_i$'s are certain integers defined in \eqref{eq:ka}.
Putting the scalars $a_{i,1}, \ldots, a_{i,p_i}$ into the boxes of the $i$-{th} row of $\pi$, we call the result a $\pi$-tableau. It turns out that the highest $\bm\ell$-weights for $W(\pi)$-modules are then parametrized by the set of row-symmetrized $\pi$-tableaux, which we denote by $\mathrm{Row}(\pi)$. For $\bA\in \mathrm{Row}(\pi)$, let $\mc M(\bA)$ denote the corresponding Verma module for $W(\pi)$ defined in \eqref{eq:Verma-W} and denote by $\mc L(\bA)$ its irreducible quotient. Our next result is the classification of finite-dimensional irreducible $W(\pi)$-modules and a character formula for Verma modules.

\begin{alphatheorem} [Theorem~\ref{thm:fd-W}, Theorem~\ref{thm:character}]
\label{thm:C}
We have the following.
\begin{enumerate} 
    \item When $\pi$ is a standard pyramid (i.e. the parity sequence $\fks$ corresponding to $\pi$ is standard), a criterion for the irreducible $W(\pi)$-modules $\mc L(\bA)$ to be finite dimensional is obtained.
    \item For $\bA\in \mathrm{Row}(\pi)$, an explicit character formula \eqref{eq:ch-Verma} for $\mc M(\bA)$ is obtained.
\end{enumerate}
\end{alphatheorem}

Since a $W(\pi)$-module is a $\Y$-module by pulling back of the map $\Y\twoheadrightarrow \rY_{m|n}^\ell(\sigma) \cong W(\pi)$, Part (1) of Theorem \ref{thm:C} can be thought as a consequence of Theorem \ref{thm:B}. Alternatively, it also follows from the corresponding results \cite[Thm.~7.9]{BK08} for non-super case and the classification result for the principal case \cite[Thm.~7.2]{BBG13}.

For Part (2), the primary differences from the non-super setting arise from the relations in Part (1) of Lemma \ref{lem:relations} as follows:
\begin{itemize} 
    \item Certain vectors constructed in constructed in \cite[Lem. 6.17]{BK08} are sensitive to the ordering of the product and may vanish in the super case, see Corollary \ref{cor:string} and also \cite[\S4]{Mo22};
    \item the second relation in Lemma \ref{lem:relations}(1) after taking higher order derivatives contains many unwanted residue terms, which makes the detail more complicated.
\end{itemize} 
To bypass these obstacles, we first consider the {\em generic $\pi$-tableau case}, which not only prevents such a vanishing issue but also simplifies some complicated calculations. 
Then the general case follows from that since the set of generic $\pi$-tableaux is Zariski dense in the set of all $\pi$-tableaux.

Finally, as an application of Part (2) of Theorem \ref{thm:C}, we prove that the centers of finite $W$-superalgebras for any even nilpotent elements in the same general linear Lie superalgebra are all isomorphic, confirming a conjecture \cite[Conj.~4.12]{ZS25} for the case of type A.
\begin{alphatheorem} [Theorem~\ref{isocent}]
\label{thm:D}
For any pyramid $\pi$ corresponding to an even nilpotent element in $\gl_{M|N}$, the center $Z(W(\pi))$ of the corresponding $W$-superalgebra is isomorphic to $Z(\rU(\gl_{M|N}))$. In particular, it depends only on $M$ and $N$, and is independent of the shape of $\pi$.
\end{alphatheorem}

This article is organized as follows.
In Section \ref{sec:ssy}, we set up the notation and recall some preliminary results about shifted super Yangians.
In Section \ref{sec:W-alg}, we recall the definition of finite $W$-superalgebras and some known results, particularly their relationship to shifted super Yangians. After that, we construct parabolic inductions, as replacements for comultiplications, for finite $W$-superalgebras, and then lift them to the corresponding shifted super Yangians. This allows one to construct the tensor product of modules, under certain compatibility assumptions on parities.
In Section \ref{sec:rep-shift}, we define highest weight modules, Verma modules, category $\mathcal O$, and $q$-characters (or Gelfand-Tsetlin characters) for shifted super Yangians.
These notions are extended to finite $W$-superalgebras, and the corresponding results are developed in Section \ref{sec:rep-W}. In particular, we give a $q$-character formula for $W$-superalgebras, which leads to an isomorphism between the center of a finite $W$-superalgebra and the center of its corresponding universal enveloping superalgebra.

\vspace{2mm}

\noindent {\bf Acknowledgements.} 
YP is partially supported by the NSTC grants 111-2628-M-006-006-MY3 and 114-2115-M-006-006-MY2, and by the National Center for Theoretical Sciences, Taipei, Taiwan. YP would also like to thank Ngau Lam for informative discussions.

\section{Shifted super Yangians}\label{sec:ssy}
Throughout the paper we work over $\bC$. Let $\bN$ be the set of natural numbers (non-negative integers). Denote  $\gge$ the partial order on $\bC$ defined by $x\gge y$ if $x-y\in\bN$.

\subsection{Shifted super Yangians}
We recall the shifted super Yangians following \cite{Pe21}.

We first fix some notations that will be frequently used. A \textit{parity sequence of type $(m|n)$} is a sequence $\fks=(\fks_i)_{1\lle i\lle m+n}$ such that $\fks_i\in \{\pm 1\}$ and $1$ appears exactly $m$ times. We shall call $\fks$ a parity sequence when the type is clear from the context. The parity sequence $(1,\dots,1,-1,\dots,-1)$ is called \textit{standard}. For each $1\lle i\lle m+n$, define $|i|\in \mathbb Z_2$ by the rule: $\fks_i=(-1)^{|i|}$. 

A shift matrix is a matrix $\sigma=(s_{i,j})_{1\lle i,j\lle m+n}$ of natural  integers such that
\beq\label{sijk}
s_{i,j}+s_{j,k}=s_{i,k}
\eeq
whenever $|i-j|+|j-k|=|i-k|$.  In particular, $s_{i,i}=0$ for all $1\lle i\lle m+n$. A shift matrix $\sigma$ is completely determined by the upper and lower diagonal entries $s_{i,i+1}$ and $s_{i+1,i}$ for $1\lle i<m+n$.

\begin{dfn}\label{def:sY}
The \textit{shifted super Yangian} $\rY_{m|n}(\sigma)$ is the unital $\bZ_2$-graded algebra generated by the elements 
\begin{align}
    &\big\{D_i^{(r)},D_i^{\prime(r)} \mid 1\lle i\lle m+n,r>0\big\},\label{eq:D-gen}\\
    &\big\{E_j^{(r)}\mid 1\lle j< m+n,r>s_{j,j+1}\big\},\\
    &\big\{F_j^{(r)}\mid 1\lle j< m+n,r>s_{j+1,j}\big\},\label{eq:F-gen}
\end{align}
where $D_i^{(r)}$, $D_i^{\prime(r)}$ are even elements while 
\beq\label{z2par}
|E_j^{(r)}|=|F_j^{(r)}|=\paa{j}+\paa{j+1},
\eeq
and subject to the following relations:
\begin{align}
\sum_{t=0}^{r}D_{i}^{(r-t)}D_{j}^{\prime (t)}&=\sum_{t=0}^{r}D_{j}^{\prime (r-t)}D_{i}^{(t)}=\delta_{r0}\delta_{ij},\,\label{dr1}\\
\big[D_{i}^{(r)},D_{j}^{(s)}\big]&=0, \label{dr2} \\
 [D_{i}^{(r)}, E_{j}^{(s)}]
        &=\fks_i(\delta_{i,j}-\delta_{i,j+1})\sum_{t=0}^{r-1} D_{i}^{(t)} E_{j}^{(r+s-1-t)}
        ,\label{dr3}\\
 [D_{i}^{(r)}, F_{j}^{(s)}]
        &=\fks_i(\delta_{i,j+1}-\delta_{i,j})\sum_{t=0}^{r-1} F_{j}^{(r+s-1-t)}D_{i}^{(t)},  \label{dr4}  \\
 [E_{i}^{(r)} , F_{j}^{(s)}]
        &=-\delta_{i,j}\fks_{i+1}
        \sum_{t=0}^{r+s-1} D_{i}^{\prime (r+s-1-t)} D_{i+1}^{(t)},    \label{dr5} \\
 [E_{i}^{(r)} , E_{i}^{(s)}]
        &=\fks_{i+1}
          \Big( \sum_{t=s_{i,i+1}+1}^{s-1} E_{i}^{(t)} E_{i}^{(r+s-1-t)} 
          -\sum_{t=s_{i,i+1}+1}^{r-1} E_{i}^{(t)} E_{i}^{(r+s-1-t)}  \Big),    \label{dr6} \\
 [F_{i}^{(r)} , F_{i}^{(s)}]
         & =\fks_{i+1} 
          \Big( \sum_{t=s_{i+1,i}+1}^{r-1} F_{i}^{(r+s-1-t)} F_{i}^{(t)} 
          -\sum_{t=s_{i+1,i}+1}^{s-1} F_{i}^{(r+s-1-t)} F_{i}^{(t)}  \Big),   \label{dr7} \\
 [E_{i}^{(r+1)}, E_{i+1}^{(s)}]&-[E_{i}^{(r)}, E_{i+1}^{(s+1)}]
=\fks_{i+1} E_{i}^{(r)}E_{i+1}^{(s)}\,, \label{dr8}      \\[2mm]
[F_{i+1}^{(s)}, F_{i}^{(r+1)}]&-[F_{i+1}^{(s+1)}, F_{i}^{(r)}]
=\fks_{i+1} F_{i+1}^{(s)}F_{i}^{(r)}\,, \label{dr9}\\[2mm]
[E_{i}^{(r)}, E_{j}^{(s)}] &= [F_{i}^{(r)}, F_{j}^{(s)}]=0
\qquad\qquad\text{\;\;if\;\; $|i-j|>1$}, \label{dr10} \\[2mm]
\big[E_{i}^{(r)},[E_{i}^{(s)}&,E_{j}^{(t)}]\big]+
\big[E_{i}^{(s)},[E_{i}^{(r)},E_{j}^{(t)}]\big]=0 \quad \text{if}\,\,\, |i-j|=1, \label{dr11}\\[2mm]
\big[F_{i}^{(r)},[F_{i}^{(s)}&,F_{j}^{(t)}]\big]+
\big[F_{i}^{(s)},[F_{i}^{(r)},F_{j}^{(t)}]\big]=0 \quad \text{if}\,\,\, |i-j|=1, \label{dr12} \\[2mm]
\big[\,[E_{i}^{(r)},E_{i+1}^{(t)}]\,&,\,[E_{i+1}^{(t)},E_{i+2}^{(s)}]\,\big]=0 \;\;\text{if\;\;}  m+n\gge 4, \, \paa{i+1}\neq\paa{i+2}, \label{sdr1}\\[2mm]
\big[\,[F_{i}^{(r)},F_{i+1}^{(t)}]\,&,\,[F_{i+1}^{(t)},F_{i+2}^{(s)}]\,\big]=0 \;\;\text{if\;\;}  m+n\gge 4, \, \paa{i+1}\neq\paa{i+2}, \label{sdr2}
\end{align}
for all indices $i,j,r,s,t$ that make sense.
\end{dfn}

\begin{rem}
If $\fks_i\neq\fks_{i+1}$, then we may interchange $r$ and $s$ in the left-hand side of \eqref{dr6} since $E_i^{(r)}$ is odd. As a consequence, the right-hand side is zero. Similarly for \eqref{dr7}.
\end{rem}

\begin{rem}
The shifted super Yangians here are dominantly shifted and they are closely related to finite $W$-superalgebras \cite{Pe21}. It is possible to define shifted super Yangians for arbitrary shifts, cf. \cite{BFN18,FPT22}. We shall only focus on the dominant shifted ones as one of our main goals is to study representations of finite $W$-superalgebras. 
\end{rem}

When $\sigma$ is the zero matrix, $\rY_{m|n}:=\rY_{m|n}(0)$ is the \textit{super Yangian} introduced in \cite{Na91} and the above presentation for $\rY_{m|n}$ was obtained by \cite{Go07} in the case of standard $\fks$, where the general $\fks$ case was obtained in \cite{Pe16} and \cite{Ts20} independently. In most cases we work with a fixed $\fks$ and omit it in our notation. In case if different choices of $\fks$'s are involved we write $\rY_{m|n}(\sigma,\fks)$ to emphasize which parity sequence we are using.

The super Yangian $\rY_{m|n}$ was originally defined via an RTT type presentation in \cite{Na91} as follows. The super Yangian $\rY_{m|n}$ is a unital $\bZ_2$-graded algebra generated by $\sfT_{i,j}^{(r)}$ of parity $|i|+|j|$ for $1\lle i,j\lle m+n$ and $r>0$, subject to the relation 
\beq\label{RTT}
(u-v)[\sfT_{i,j}(u),\sfT_{k,l}(v)]=(-1)^{|i||j|+|j||k|+|i||k|}(\sfT_{k,j}(u)\sfT_{i,l}(v)-\sfT_{k,j}(v)\sfT_{i,l}(u)).\eeq
Here
\[
\sfT_{i,j}(u)=\delta_{i,j}+\sum_{r>0}\sfT_{i,j}^{(r)}u^{-r}\in \rY_{m|n}[\![u^{-1}]\!].
\]
There exists an anti-automorphism $\tau:\rY_{m|n}\rightarrow \rY_{m|n}$ defined by the following assignment:
\beq\label{tautran}
\tau(\sfT_{ij}(u))=(-1)^{|j|(|i|+1)}\sfT_{ji}(u).
\eeq

The following surjective homomorphism $\ev: \rY_{m|n} \rightarrow \rU(\gl_{m|n})$ is called the {\em evaluation homomorphism}, which is defined by:
\beq\label{evahom}
\ev(\sfT_{ij}(u)) = \delta_{ij} + \fks_i e_{ij} u^{-1}.
\eeq

One can introduce the Gaussian generators $\sfD_{i}^{(r)},\sfE_i^{(r)},\sfF_{i}^{(r)}$ for $1\lle i\lle m+n$ and $r\gge 1$ by applying the Gauss decomposition to $\sfT(u):=(\sfT_{ij}(u))_{1\lle i,j\lle m+n}$. Specifically, let
\beq\label{eq:D-def}
\sfD_i(u)=1+\sum_{r>0}\sfD_i^{(r)}u^{-r},\qquad 1\lle i\lle m+n,
\eeq
\[
\sfE_{i,j}(u)=\sum_{r>0}\sfE^{(r)}_{i,j}u^{-r},\quad \sfF_{j,i}(u)=\sum_{r>0}\sfF^{(r)}_{j,i}u^{-r},\quad 1\lle i<j\lle m+n,
\]
which satisfy the relations
\beq\label{eq:GD}
\sfT_{i,j}(u)=\sum_{k=1}^{\min(i,j)}\sfF_{i,k}(u)\sfD_k(u)\sfE_{k,j}(u)
\eeq
with the convention that $\sfE_{i,i}(u)=\sfF_{i,i}(u)=1$ for all $1\lle i\lle m+n$. In particular, we set
\[
\sfE_i(u)=\sum_{r>0}\sfE_i^{(r)}u^{-r}:=\sfE_{i,i+1}(u),\quad 
\sfF_i(u)=\sum_{r>0}\sfF_i^{(r)}u^{-r}:=\sfF_{i+1,i}(u)
\]
for $1\lle i<m+n$.

It is known from \cite{Pe21} that $\Y$ can be identified as a subalgebra of $\rY_{m|n}$ via the homomorphism induced by
\[
D_i^{(r)}\mapsto \sfD_i^{(r)},\quad E_i^{(r)}\mapsto \sfE_i^{(r)},\quad F_i^{(r)}\mapsto \sfF_i^{(r)}
\]
for admissible indices. However, in general the elements $\sfT_{i,j}^{(r)}$, $\sfE_{i,j}^{(r)}$, and $\sfF_{i,j}^{(r)}$ in $\rY_{m|n}$ differ from the elements $T_{i,j}^{(r)}$, $E_{i,j}^{(r)}$, and $F_{i,j}^{(r)}$ in $\Y$, respectively, defined below in \S\ref{ssec:PBW}. From now on, we shall not distinguish $D_i^{(r)},E_i^{(r)},F_i^{(r)}$ and $\sfD_i^{(r)},\sfE_i^{(r)},\sfF_i^{(r)}$, respectively.

\subsection{PBW basis}\label{ssec:PBW}
For $1\lle i<j\lle m+n$ and $r>s_{i,j}$ (resp. $r>s_{j,i}$), define the elements $E_{i,j}^{(r)}$ (resp. $F_{j,i}^{(r)}$) inductively by the rule:
\begin{align}\label{gedef}
E_{i,i+1}^{(r)}&:=E_i^{(r)},\qquad E_{i,j}^{(r)}:=\fks_{j-1}[E_{i,j-1}^{(r-s_{j-1,j})},E_{j-1}^{(s_{j-1,j}+1)}],\\
\label{gfdef} F_{i+1,i}^{(r)}&:=F_i^{(r)},\qquad F_{j,i}^{(r)}:=\fks_{j-1}[F_{j-1}^{(s_{j,j-1}+1)},F_{i,j-1}^{(r-s_{j,j-1})}].
\end{align}

Denote by $\Yz:=\Yz(\sigma)$ (it is known that $\Yz(\sigma)$ is independent of $\sigma$) the subalgebra of $\Y$ generated by all the $D_i^{(r)}$ with $1\lle i\lle m+n$ and $r>0$, $\Yp$ the subalgebra of $\Y$ generated by all the $E_i^{(r)}$ with $1\lle i< m+n$ and $r>s_{i,i+1}$, and  $\Ym$ the subalgebra of $\Y$ generated by all the $F_i^{(r)}$ with $1\lle i< m+n$ and $r>s_{i+1,i}$.

\begin{prop}[{\cite[Coro.~5.9, 5.10]{Pe21}}]\label{prop:PBW}
The supermonomials in the elements
\begin{align*}
&\{D_i^{(r)}\mid 1\lle i\lle m+n,r>0\},\\
&\{E_{i,j}^{(r)}\mid1\lle i<j\lle m+n,r>s_{i,j}\},\\
&\{F_{j,i}^{(r)}\mid1\lle i<j\lle m+n,r>s_{j,i}\},
\end{align*}
respectively, taken in any fixed order form a basis for $\Yz$, $\Yp$, $\Ym$, respectively. Moreover, the multiplicative map
\beq\label{eq:triangle}
\Ym\otimes \Yz\otimes \Yp \stackrel{\sim}{\longrightarrow} \Y 
\eeq
is an isomorphism of superspaces.
\end{prop}

Define the canonical filtration of $\mathcal F_0\Y\subset \mathcal F_1\Y\subset \cdots$ of $\Y$ by declaring that all $D_i^{(r)}$, $E_{i,j}^{(r)}$ and $F_{j,i}^{(r)}$ are of degree $r$. In other words, $\mathcal F_d\Y$ is the span of all supermonomials in these elements of total degree at most $d$. Then the associated graded superalgebra $\gr\,\Y$ is free supercommutative on generators
\begin{align}
&\{\gr_r\,D_i^{(r)}\mid 1\lle i\lle m+n,r>0\},\label{eq:gr-img-d}\\
&\{\gr_r\,E_{i,j}^{(r)}\mid1\lle i<j\lle m+n,r>s_{i,j}\},\\
&\{\gr_r\,F_{j,i}^{(r)}\mid1\lle i<j\lle m+n,r>s_{j,i}\}.\label{eq:gr-img-f}
\end{align}

Also define the \textit{positive} and \textit{negative Borel subalgebras}
\beq\label{}
\rY_{m|n}^{\gge 0}(\sigma):=\Yz\Yp,\qquad \rY_{m|n}^{\lle 0}(\sigma):=\Ym\Yz.
\eeq
Clearly, by the defining relations of $\Y$, they are indeed subalgebras of $\Y$.

It is also convenient to introduce a different set of generators $T_{i,j}^{(r)}$ of $\rY_{m|n}(\sigma)$ as follows. Let
\[
E_{i,j}(u):=\sum_{r>s_{i,j}}E_{i,j}^{(r)}u^{-r},\quad F_{j,i}(u):=\sum_{r>s_{j,i}}F_{j,i}^{(r)}u^{-r}
\]
for $1\lle i<j\lle m+n$ and set $E_{i,i}(u)=F_{i,i}(u)=1$, $D_i(u)=\sfD_i(u)$ from \eqref{eq:D-def} by convention. Then define
\beq\label{eq:GD-shift}
T_{i,j}(u)=\sum_{r\gge 0}T_{i,j}^{(r)}u^{-r}:=\sum_{k=1}^{\min(i,j)}F_{i,k}(u)D_k(u)E_{k,j}(u).
\eeq
Clearly, $T_{i,j}^{(0)}=\delta_{ij}$ and $T_{i,j}^{(r)}=0$ for $0< r\lle s_{i,j}$. 

Note that if $\sigma$ is the zero matrix, then $T_{i,j}^{(r)}$ corresponds to $\sfT_{i,j}^{(r)}$ in the (non-shifted) super Yangian $\rY_{m|n}$. 
\begin{lem}\label{lem:PBW-T}
The associated graded superalgebra $\gr\,\Y$ is free supercommutative on generators 
\[\{\gr_r\,T_{i,j}^{(r)}\mid 1\lle i,j\lle m+n, r>s_{i,j}\}.\] 
Thus, the supermonomials in the elements 
\[\{T_{i,j}^{(r)}\mid 1\lle i,j\lle m+n, r>s_{i,j}\}\] 
taken in some fixed order form a basis for $\rY_{m|n}(\sigma)$.
\end{lem}
\begin{proof}
Note that $T_{i,j}^{(r)}=0$ if $0<r\lle s_{i,j}$. Clearly, the elements $D_{i}^{(r)}$ (resp. $E_{i,j}^{(r)}$ and $F_{j,i}^{(r)}$) can be written as $T_{i,i}^{(r)}$ (resp. $T_{i,j}^{(r)}$ and $T_{j,i}^{(r)}$) plus a linear combination of monomials of total degree $r$ in the elements $\{T_{a,b}^{(s)}\mid 1\lle a,b\lle m+n,0<s<r\}$. Since $\gr\,\Y$ is free supercommutative on generators \eqref{eq:gr-img-d}--\eqref{eq:gr-img-f}, the elements $\{\gr_r\,T_{i,j}^{(r)}\mid 1\lle i,j\lle m+n, r>s_{i,j}\}$ also freely generate $\gr\,\Y$.
\end{proof}

\subsection{Properties of shifted Yangians}
In this subsection, we summarize a few important properties of shifted super Yangians that will be used later.
\subsubsection{Some isomorphisms}
Let $\sigma$ be a shift matrix and fix $\fks$.
Note that the transpose $\sigma^t$ is also a shift matrix.
On the other hand, suppose that
$\vec\sigma = (\vec{s}_{i,j})_{1 \lle i,j \lle m+n}$ is another shift matrix satisfying  the condition 
$$\vec{s}_{i,i+1}+\vec s_{i+1,i}= s_{i,i+1}+s_{i+1,i}$$ for all $1\lle i< m+n$. 
Denote by $\vec{D}_{i}^{(r)}, \vec{E}_{j}^{(r)}$ and $\vec{F}_{j}^{(r)}$ the generators of $\rY_{m|n}(\vec\sigma)$ to avoid confusion. 

\begin{prop}\cite[Props. 5.13, 5.14]{Pe21}
\begin{enumerate}
\item 
The restriction of \eqref{tautran} gives rise to an anti-isomorphism $\tau: \rY_{m|n}(\sigma)\rightarrow \rY_{m|n}(\sigma^t)$ such that
\beq\label{taudef}
\tau(D_{i}^{(r)}) =
D_{i}^{(r)},\,\,\,
\tau(E_{j}^{(r)}) =
(-1)^{|j+1|(|j|+1)}F_{j}^{(r)},\,\,\,
\tau(F_{j}^{(r)}) =
(-1)^{|j|(|j+1|+1)}E_{j}^{(r)}.
\eeq
Moreover, for $1\lle i<j\lle m+n$, we have
\beq\label{taudef2}
\tau(E_{i,j}^{(r)})=(-1)^{|j|(|i|+1)}F_{j,i}^{(r)}, \quad \tau(F_{j,i}^{(r)})=(-1)^{|i|(|j|+1)}E_{i,j}^{(r)}.
\eeq
\item 
The map $\iota:\rY_{m|n}(\sigma) \rightarrow \rY_{m|n}({\vec{\sigma}})$ defined by
\beq\label{iotadef}
\iota(D_{i}^{(r)}) = \vec{D}_{i}^{(r)},\quad
\iota(E_{j}^{(r)}) = \vec{E}_{j}^{(r-s_{j,j+1}+\vec s_{j,j+1})},\quad
\iota(F_{j}^{(r)}) = \vec{F}_{j}^{(r-s_{j+1,j}+\vec s_{j+1,j})},
\eeq
is an isomorphism.
\end{enumerate}
\end{prop}

For any power series $f(u)\in 1+u^{-1}\bC[\![u^{-1}]\!]$, there exists an automorphism of $\rY_{m|n}$ given by
\beq\label{eq:mu-f}
\mu_f:\rY_{m|n}\to \rY_{m|n}, \qquad \sfT_{i,j}(u)\mapsto f(u)\sfT_{i,j}(u).
\eeq
It is well known that $\mu_f$ fixes each $E_i^{(r)}$ and $F_i^{(r)}$ and sends $D_i(u)$ to $f(u)D_i(u)$. Hence $\mu_f$ restricts to an automorphism to $\Y$.

\subsubsection{$\mathbf Q_{m|n}$-grading}
Let $\fkc$ be the (commutative) Lie subalgebra of $\Y$ spanned by the elements $D_i^{(1)}$ for $1\lle i\lle m+n$. Let $\ve_1,\dots,\ve_{m+n}$ be the basis for $\fkc^*$ dual to the elements $\fks_1 D_1^{(1)},\dots,\fks_{m+n} D_{m+n}^{(1)}$. We call elements of $\fkc^*$ \textit{weights} and elements of 
\beq\label{eq:Pdef}
\mathbf P_{m|n}:=\bigoplus_{i=1}^{m+n} \bZ \ve_i \subset \fkc^*
\eeq
\textit{integral weights}. The \textit{root lattice} associated to the Lie superalgebra $\gl_{m|n}$ is the $\bZ$-submodule $\mathbf Q_{m|n}$ of $\mathbf P_{m|n}$ spanned by the simple roots $\alpha_i:=\ve_i-\ve_{i+1}$ for $1\lle i<m+n$. Define the \textit{dominance order} of $\fkc^*$ by the rule: $\alpha\gge \beta$ if $\alpha-\beta$ is a sum of simple roots.

It is often helpful to consider $\Y$ as a superalgebra graded by the root lattice $\mathbf Q_{m|n}$. Explicitly, by the defining relations of $\Y$ and the notation, we can define a $\mathbf Q_{m|n}$-grading
\[
\Y=\bigoplus_{\alpha\in\mathbf Q_{m|n}}\big(\Y\big)_{\alpha}
\]
of the superalgebra $\Y$ by declaring that the generators $D_{i}^{(r)}$, $E_{i}^{(r)}$, and $F_{i}^{(r)}$ are of degree $0$, $\alpha_i$, $-\alpha_i$, respectively. In particular, it follows from \eqref{eq:GD-shift} that $T_{i,j}^{(r)}$ is of degree $\ve_i-\ve_j$.

\subsubsection{Coproduct}
It is well known that the super Yangian $\rY_{m|n}$ is a Hopf superalgebra with counit $\ve:\rY_{m|n}\to \bC$ and  coproduct $\Delta:\rY_{m|n}\to \rY_{m|n}\otimes \rY_{m|n}$ respectively defined by
\begin{align}\label{counit}
\ve(\sfT_{i,j}(u))&=\delta_{i,j},\\
\label{coproduct}\Delta(\sfT_{i,j}(u))&=\sum_{k=1}^{m+n}\sfT_{i,k}(u)\otimes \sfT_{k,j}(u).
\end{align}

Define
\beq\label{eq:def-H}
H_i(u)=\sum_{r\gge 0}H_i^{(r)}u^{-r}:=\big(  D_i(u)\big)^{-1}D_{i+1}(u)
\eeq
for $1\lle i<m+n$. Clearly, $H_i^{(0)}=1$. Denote $\mathrm{K}_{m|n}^{>0}(\sigma)$ (resp. $\mathrm{K}_{m|n}^{<0}(\sigma)$) the two-sided ideal of the Borel subalgebra $\rY_{m|n}^{\gge 0}(\sigma)$ (resp. $\rY_{m|n}^{\lle 0}(\sigma)$) generated by $E_i^{(r)}$ (resp. $F_i^{(r)}$) for $1\lle i<m+n$ and $r>s_{i,i+1}$ (resp. $r>s_{i+1,i}$). Again, we simply write them as $\mathrm{K}_{m|n}^{>0},\mathrm{K}_{m|n}^{<0}$ when the shift matrix $\sigma$ is zero.

The following proposition for the coproduct is a refinement of \cite[Prop.~2.7]{LM21} for the standard parity sequence, see also \cite[Prop.~2.6]{Lu22} and \cite[Lem.~2.5]{HZ24}. 

\begin{prop}\label{prop:copro}
The coproduct $\Delta:\rY_{m|n}\to \rY_{m|n}\otimes \rY_{m|n}$ has the following properties:
\begin{enumerate}
    \item $\Delta (D_i^{(r)})\equiv \sum_{s=0}^{r}D_i^{(s)}\otimes D_i^{(r-s)}$ modulo $\mathrm{K}_{m|n}^{>0}\otimes \mathrm{K}_{m|n}^{<0}$;
    \item $\Delta (E_i^{(r)})\equiv 1\otimes E_i^{(r)}+\sum_{s=1}^rE_i^{(s)}\otimes H_i^{(r-s)}$ modulo $\big(\mathrm{K}_{m|n}^{>0}\big)^2\otimes \mathrm{K}_{m|n}^{<0}$;
    \item $\Delta(F_i^{(r)})\equiv F_i^{(r)}\otimes 1+\sum_{s=1}^rH_i^{(r-s)}\otimes F_i^{(s)}$ modulo $\mathrm{K}_{m|n}^{>0}\otimes \big(\mathrm{K}_{m|n}^{<0}\big)^2$.
\end{enumerate}
\end{prop}
\begin{proof}
We prove it using the same argument as in \cite[Prop.~2.7]{LM21}. Let
\[
\mathbf D_i^{(2)}:=D_i^{(2)}-\frac12(D_i^{(1)})^2-\frac12D_i^{(1)},
\]
then by \eqref{dr3}, we have $[\mathbf D_i^{(2)},E_i^{(r)}]=\fks_iE_i^{(r+1)}$. Note that $D_i^{(2)}=\sfT_{i,i}^{(2)}-\sum_{j<i}\sfT_{i,j}^{(1)}\sfT_{j,i}^{(1)}$ and $D_i^{(1)}=\sfT_{i,i}^{(1)}$. A direct computation implies that
\begin{align*}
\Delta(\mathbf D_i^{(2)})\in\mathbf D_i^{(2)}\otimes 1+1\otimes \mathbf D_i^{(2)}+E_i^{(1)}\otimes F_i^{(1)}+\bC^{\times} E_{i-1}^{(1)}\otimes F_{i-1}^{(1)}+\big(\rY_{m|n}^{+}\big)^2\otimes \big(\rY_{m|n}^{-}\big)^2.
\end{align*}
Note that $\Delta(E_i^{(1)})=E_i^{(1)}\otimes 1+1\otimes E_i^{(1)}$. Using $[\mathbf D_i^{(2)},E_i^{(r)}]=\fks_iE_i^{(r+1)}$, $[F_i^{(1)},E_i^{(s)}]=\fks_i H_i^{(s)}$, and $[E_i^{(r)},\rY_{m|n}^{-}]\subset \mathrm{Y}_{m|n}^{\lle0}$, one shows inductively Part (2). Similarly, one proves Part (3).

We then prove Part (1). It is clearly for the case $i=1$. By \eqref{eq:def-H}, it suffices to prove the statement with $D$ replaced by $H$. Applying $[\Delta(F_i^{(1)}),?]$ to Part (2) and using $[F_i^{(1)},E_i^{(s)}]=\fks_i H_i^{(s)}$, we find that
\[
\Delta(H_i^{(r)})-\sum_{s=0}^r H_i^{(s)}\otimes H_i^{(r)}\in \Big[F_i^{(1)},\big(\mathrm{K}_{m|n}^{>0}\big)^2\Big]\otimes \mathrm{K}_{m|n}^{<0}.
\]
Similarly, applying $[\Delta(E_i^{(1)}),?]$ to Part (3), we have
\[
\Delta(H_i^{(r)})-\sum_{s=0}^r H_i^{(s)}\otimes H_i^{(r)}\in \mathrm{K}_{m|n}^{>0}\otimes \Big[E_i^{(1)},\big(\mathrm{K}_{m|n}^{<0}\big)^2\Big].
\]
Combining the two equations above completes the proof of Part (1).
\end{proof}

Recall that we can consider $\Y$ as a subalgebra of $\rY_{m|n}$ via the canonical embedding $\Y\hookrightarrow \rY_{m|n}$ such that the generators $D_i^{(r)}$, $E_{i}^{(r)}$, and $F_i^{(r)}$ of $\Y$ map to the elements of $\rY_{m|n}$ with the same name. Write $\sigma=\sigma'+\sigma''$ where $\sigma'$ (resp. $\sigma''$) is an lower (resp. upper) triangular shift matrix. Embedding the shifted super Yangians $\Y$, $\rY_{m|n}(\sigma')$, and $\rY_{m|n}(\sigma'')$ into $\rY_{m|n}$ via the canonical way.
\begin{lem}\label{lem:copro-res}
The coproduct $\Delta:\rY_{m|n}\to \rY_{m|n}\otimes \rY_{m|n}$ restricts to a superalgebra homomorphism
\beq\label{eq:coproduct}
\Delta:\Y\to \rY_{m|n}(\sigma')\otimes \rY_{m|n}(\sigma'').
\eeq
\end{lem}
\begin{proof}
This is established in Theorem \ref{thm:grownup}.
\end{proof}

In particular, when $\sigma$ is upper triangular, then the coproduct $\Delta:\rY_{m|n}\to \rY_{m|n}\otimes \rY_{m|n}$ restricts to
\beq\label{eq:copro-spec}
\Delta:\Y\to \rY_{m|n}\otimes \Y.
\eeq
More generally, if $\sigma=\sigma'+\sigma''$, then there exists a homomorphism 
\[
\Delta^\iota:\Y\to \rY_{m|n}(\sigma')\otimes \rY_{m|n}(\sigma'')
\]
obtained by twisting $\Delta$ using isomorphism $\iota$, see Remark \ref{rem:row-shift}.

\subsection{Quantum Berezinian and center of $\Y$}
Recall from \cite{GGRW05} that the $kk$-\textit{quasideterminant} of a $k\times k$ matrix $X$, denoted by $|X|_{k,k}$, is given by
\beq\label{eq:quasi}
|X|_{k,k}:= d-ca^{-1}b,
\eeq
where 
$X= 
\left(
\begin{array}{c|c}
 a & b   \\
\hline
 c & d 
\end{array}
\right)
$ 
so that $a$ is an invertible square matrix of size $k-1$ and $d$ is a scalar. Then it is well known that
\[
D_k(u)=\big|\sfT_{[k]}(u)\big|_{k,k}\in \Y[\![u^{-1}]\!],
\]
where $\sfT_{[k]}(u)=(\sfT_{i,j}(u))_{1\lle i,j\lle k}$ is the $k\times k$-submatrix of $\sfT(u)$ on the upper left corner.

For a given parity sequence $\fks$,  define $\ka_i$ for $1\lle i\lle m+n$ recursively by
\beq\label{eq:ka}
\ka_1=\begin{cases}
    0, & \text{ if }\fks_1=1,\\
    -1, & \text{ otherwise,}
\end{cases}
\qquad
\ka_i=\ka_{i-1}+\tfrac{1}{2}(\fks_i+\fks_{i-1}), \qquad 1<i\lle m+n.
\eeq
It follows from \cite[Theorem~1]{Gow05} and  \cite[Theorem~4.5]{CH23} (see also \cite{Ts20}) that the \textit{quantum Berezinian} $\mathcal C_{m|n}(u)$ of the matrix $\sfT(u)$ introduced in \cite{Na91} can be expressed as
\beq\label{eq:Ber}
\mathcal C_{m|n}(u)=\sum_{r\gge 0}\mathcal C_{m|n}^{(r)}u^{-r}=\prod_{i=1}^{m+n}D_i(u-\ka_i)^{\fks_i}\in \rY_{m|n}(\sigma)[\![u^{-1}]\!].
\eeq
It is a group like element, by which we mean that
\beq\label{coBer}
\Delta(\mathcal C_{m|n}(u)) =\mathcal C_{m|n}(u) \otimes \mathcal C_{m|n}(u).
\eeq
In the case of the standard parity sequence, \eqref{coBer} first appeared in \cite[(15)]{Na91} where a detail proof is given in \cite[\textsection 4]{Na20}. The general case follows from the fact that quantum Berezinian of $\sfT^\fks(u)$ for various $\fks$ are identified under suitable Hopf superalgebra isomorphisms between super Yangians associated with different parity sequences, see e.g. \cite[(4.14)]{CH23}.

Denote by $\mathcal Z(\rY_{m|n}(\sigma))$ the center of $\rY_{m|n}(\sigma)$.
\begin{prop}\label{prop:center}
The elements $\mathcal C_{m|n}^{(r)}$ for $r\gge 1$ are algebraically free generators of the center $\mathcal Z(\rY_{m|n}(\sigma))$.
\end{prop}
\begin{proof}
When $\sigma=0$ and $\fks$ is the standard one, it is well known \cite{Go07,Na20} that the set $\{\mathcal C_{m|n}^{(r)}\,|\,r\gge1\}$ is algebraically independent and it generates $\mathcal Z(\rY_{m|n})$. For a general $\fks$, the proof in the standard case works analogously as in \cite[Thm.~2.43]{Ts20}.

We claim that $\mathcal Z(\rY_{m|n})= \mathcal Z(\rY_{m|n}(\sigma))$ for any $\sigma$.
Note that $D_i^{(r)}$ belongs to $\rY_{m|n}^0 \subseteq \rY_{m|n}$ for all $1\lle i\lle m+n, r>0$. Then \eqref{eq:Ber} implies that $\mathcal Z(\rY_{m|n})\subseteq\mathcal Z(\rY_{m|n}(\sigma))$.
The other inclusion can be established by passing to the associated graded superalgebra $\gr^L\, \rY_{m|n}(\sigma)$ with respect to the loop filtration, with the help of \cite[Thm.~4.4]{Pe21} and \cite[Prop.~3.8]{Na20}, and modifying the argument in \cite[Thm. 1.7.5]{Mo07}.
\end{proof}

Let $\mathcal E(u):=(\delta_{i,j}+\fks_i e_{i,j}u^{-1})_{1\lle i,j\lle m+n}$. For $1\lle k\lle m+n$, define the following formal series 
\beq\label{eq:zetakdef}
\zeta_k(u):=\big|\mathcal E_{[k]}(u)\big|_{kk}\in \mathrm U(\gl_{m|n})[\![u^{-1}]\!], 
\eeq
where $\mathcal E_{[k]}(u)$ is the $k\times k$-submatrix of $\mathcal E(u)$ on the upper left corner. 
Set 
\beq\label{eq:zdef}
z_{m|n,\fks}(u)=\sum_{r\gge 0} z_{m|n,\fks}^{(r)} u^{-r} := \prod_{i=1}^{m+n} (u-\kappa_i)^{\fks_i}\zeta_i(u-\kappa_i)^{\fks_i} \in \mathrm U(\gl_{m|n})[\![u^{-1}]\!].
\eeq
We simply write $z_{m|n}(u)$ or $z(u)$ when there is no possible confusion about which $\fks$ we are using.
Clearly, $z_{m|n}(u)$ is obtained by applying the evaluation map \eqref{evahom} to $\prod_{i=1}^{m+n}(u-\kappa_i)^{\fks_i}\mathcal C_{m|n}(u)$. 

The series appeared in \cite{Na91,MR04,BG19} for the standard parity sequence. It was also observed in \cite[Rmk.~2.5]{BG19} that $z_{M|N,\fks}(u)$ can be defined for a general $\fks$. It is well known that the coefficients of $z_{m|n}(u)$ generate the center of $\mathrm U(\gl_{m|n})$; see Theorem \ref{thmHC} below.

\section{Finite $W$-superalgebras}\label{sec:W-alg}
In this section, we first recall the definition of finite $W$-superalgebras and their relations to shifted super Yangians following \cite{Pe21}. 
Then we develop some properties about finite $W$-superalgebras that will be used to study their representations.
In particular, a homomorphism called {\em parabolic induction} is constructed in \eqref{Wdel}. We lift its domain to the shifted super Yangian in Theorem \ref{thm:grownup}.
These homomorphisms allow one to construct the tensor product of modules.

\subsection{Pyramid}
We recall the notion of  pyramid \cite{EK05, Ho12} following \cite{BK08,Pe21}. 
A \textit{pyramid} $\pi$ of \textit{level} $\ell$ is a sequence $\pi=(q_1,\ldots,q_\ell)$ of integers such that
\beq\label{eq:k-def}
0<q_1\lle \cdots\lle q_k,\quad q_{k+1}\gge \cdots q_\ell>0.
\eeq
for some fixed integer $0\lle k\lle \ell$. One can visualize the pyramid $\pi$ as a diagram consisting of $q_1$ boxes
stacked in the first column, $q_2$ boxes stacked in the second column, $\ldots$, $q_\ell$ boxes
stacked in the $\ell$-th column, where columns are numbered $1, 2, \ldots, \ell$ from left to right. We say that
the pyramid is \textit{left-justified} if $q_1\gge q_2\gge \cdots\gge  q_\ell$ and \textit{right-justified} if $q_1\lle q_2\lle \cdots\lle  q_\ell$.

For our purpose, we shall consider the pyramid of at most $m+n$ rows of boxes. Let $p_i$ be the number of boxes on the $i$-th row. Then it defines the tuple $(p_1,\dots,p_{m+n})$ of row lengths with 
\[
0\lle p_1\lle p_2\lle \ldots\lle p_{m+n}=\ell.
\]
Note that we allow empty rows.
Fix a parity sequence $\fks=(\fks_1,\cdots ,\fks_{m+n})$ of type $(m|n)$.
We say the $i$-th row of $\pi$ is an even (resp. odd) row if $\fks_i=1$ (resp. $-1$). We call a box in $\pi$ even (resp. odd) if it belongs to an even (resp. odd) row.
Let $M$ and $N$ be defined by
\be
M:= \sum_{\substack{1\lle i\lle m+n\\ \fks_i=1}} p_i, \qquad N:= \sum_{\substack{1\lle i\lle m+n\\ \fks_i=-1}} p_i.
\ee
Namely, $M$ is the number of even boxes in $\pi$, while $N$ is the number of odd boxes in $\pi$.
We put integers from $\I=\{1,2,\ldots, M+N\}$ into boxes of $\pi$ by the following rule: put $1,2,\ldots, M$ into the even boxes of $\pi$ down columns, from left to right, and put $M+1,\ldots,M+N$ into the odd boxes of $\pi$ in the same way. For any $i\in \I$, let col($i$) and row($i$) denote the column and row in which $i$ appears, respectively.

For $1\lle i\lle \ell$, denote $m_i$ (resp. $n_i$) the number of even (resp. odd) boxes in the $i$-th column of $\pi$. Pick some $1\lle k\lle \ell$ as in \eqref{eq:k-def} such that $q_k=m_k+n_k$ is maximal among all $q_i$ for $1\lle i\lle \ell$. Define a shift matrix $\sigma=(s_{i,j})_{1\lle i,j\lle m+n}$ to the pyramid $\pi$ by setting
\beq\label{eq:sigma-def}
s_{i,j}:=\begin{cases}
    \#\{c=1,\dots,k\mid i>m+n-q_c\gge j\}, &\text{ if }i\gge j,\\
    \#\{c=k+1,\dots,\ell\mid i\lle m+n-q_c< j\}, &\text{ if }i\lle j.
\end{cases}
\eeq
Note that the pyramid $\pi$ can be recovered if the shift matrix $\sigma$ and the level $\ell$ are known.
For each $1\lle i\lle m+n$, one simply puts $\ell-s_{m+n,i}-s_{i,m+n}$ boxes on the $i$-th row of $\pi$ with $s_{m+n,i}$ columns indented from the left and $s_{i,m+n}$ columns indented from the right. 

\subsection{Finite $W$-superalgebras}
We recall the definition of $W$-superalgebra in terms of a fixed pyramid $\pi$ as above.

Let $\g=\gl_{M|N}$ be the general linear Lie superalgebra consisting of $(M+N)\times(M+N)$ matrices equipped with the standard $\bZ_2$-grading $\g=\g_{\ovl 0}\oplus \g_{\ovl 1}$. Denote by $(\, \cdot \, ,\,\cdot\,)$ the non-degenerate even supersymmetric $\g$-invariant bilinear form defined by 
$$(x,y):=\str(x y)$$
for all $x,y \in \g$, where $x y$ stands for the usual matrix product and $\str$ means the supertrace.
Endow the $\bZ$-grading $\g=\oplus_{j\in\bZ}\g_j$ defined by declaring that $e_{ij}$ is of degree $(\col(j)-\col(i))$ for all $1\lle i,j\lle M+N$.
Define the following subalgebras according to the $\bZ$-grading:
\[\h:=\g_0,\qquad  \fkp:=\oplus_{j\gge 0}\g_j, \qquad \fkm:=\oplus_{j<0}\g_j.\]
Let $e\in\fkp$ denote the nilpotent element in $\g_{\ovl{0}}$ defined by
\beq\label{e_pi}
e:=\sum_{i,j} e_{ij},
\eeq
summing over all pairs $\ytableausetup{boxsize=1em}
\ytableausetup{centertableaux}\begin{ytableau}
i & j
\end{ytableau}$ appearing in the pyramid $\pi$.
It is a fact \cite{Ho12} that the $\bZ$-grading $\g=\oplus_{j\in\bZ}\g_j$ defined in this fashion is a good $\bZ$-grading for $e\in\g_1$. Note that in the setup of \cite{Pe21} this grading is doubled.

Define $\chi\in\g^*$ by 
\beq\label{eq:chi}
\chi(y):=(y,e), \qquad \forall \,y\in\g.
\eeq
The restriction of $\chi$ on $\mathfrak{m}$ induces a 1-dimensional $\rU(\mfk{m})$-module. Let $I_\chi$ be the left ideal of $\rU(\g)$ generated by 
$$\{a-\chi(a)  \,  |  \,  a\in\mfk{m}\}.$$ 
As a consequence of the PBW theorem, we have $\rU(\g)=I_\chi\oplus \rU(\mathfrak{p})$ together with the following identification 
$$\rU(\g)/I_\chi \cong \rU(\mfk{p})$$ 
induced by the natural projection $\pr_\chi:\rU(\g)\twoheadrightarrow \rU(\mfk{p})$. One then defines the following $\chi$-twisted action of $\mathfrak{m}$ on $\rU(\mathfrak{p})$ by  
$$a\cdot y := \pr_\chi([a,y]) 
\qquad \forall\, a\in\mfk{m}, y\in \rU(\mfk{p}).$$ 
The associated {\em finite W-superalgebra}, which we often omit the prefix ``finite" from now on, is defined to be the space of $\mfk{m}$-invariants in $\rU(\mfk{p})$ under the $\chi$-twisted action. That is,
\beq\label{def:Walg}
W(\pi):=\rU(\mfk{p})^\mfk{m}=\{y\in \rU(\mfk{p})\,\,|\,\, [a,y]\in I_\chi, \text{for all } a\in\mfk{m}\}.
\eeq
We adapt the multiplication of $\rU(\mfk{p})$ as the multiplication of $W(\pi)$.

In the special case when $\pi$ consists of a single column, the corresponding nilpotent element \eqref{e_pi} is zero and therefore $\chi=0$, $\fkp=\g$, $\fkm=0$. Thus, the corresponding $W$-superalgebra is precisely the universal enveloping superalgebra $\rU(\g)$.

\subsection{Identification} 
Recall that, for each $1\lle j\lle \ell$, $m_j$ and $n_j$ denote the number of even and odd boxes in the $j$-th column of $\pi$, respectively.
Pick some $1\lle k\lle \ell$ as described in \eqref{eq:k-def} and set $\hbar=m_k-n_k$.
Let $(\check{q}_1,\dots,\check{q}_{\ell})$ denote the {\em super column heights} of $\pi$, where $\check{q}_j$ is defined by
\beq\label{sheight}
\check{q}_j:= m_j-n_j.
\eeq

Consider the ordered index set $\I:=\lbrace 1<\ldots<M < \ovl{1}<\ldots<\ovl{N}\rbrace$, where $\ovl{i}=i+M$ for $1\lle i\lle N$ while $M$ and $N$ are the numbers of even and odd boxes of $\pi$, respectively. The {\em Kazhdan filtration} of $\rU(\g)$
\[
\cdots\subseteq \mathrm{F}_d\rU(\g) \subseteq \mathrm{F}_{d+1}\rU(\g)\subseteq \cdots
\] 
is defined by declaring
\beq\label{kzdegdef}
\deg (e_{i,j}):= \text{col}(j)-\text{col}(i)+1
\eeq
for each $i,j \in \I$, where $\mathrm{F}_d\rU(\g)$ denotes the span of all supermonomials $e_{i_1,j_1}\cdots e_{i_s,j_s}$ for $s\gge 0$ with $\sum_{k=1}^s \deg (e_{i_k,j_k})\lle d$. 
The Kazhdan filtration descends to filtrations of $\rU(\fkp)$ and $W(\pi)$ by setting $\mathrm{F}_d\rU(\fkp):=\pr_\chi(\mathrm{F}_d\rU(\g))$ and $\mathrm{F}_d W(\pi):=W(\pi)\cap \mathrm{F}_d\rU(\fkp)$, respectively.
Let $\gr\,\rU(\g)$ and $\gr\, W(\pi)$ denote the associated graded superalgebras. 

Define $\rho = (\rho_1,\dots,\rho_\ell)$, where $\rho_r$ is given by
\beq\label{rhodef}
\rho_r := \hbar-\check{q}_{r} - \check{q}_{r+1} -\cdots-\check{q}_\ell
\eeq
for each $r=1,\dots,\ell$.
For all $i,j\in \I$, define 
\beq\label{etilde}
\tilde e_{i,j} :=  
e_{i,j} + \delta_{i,j} (-1)^{\tp{(i)}}\rho_{\col(i)},
\eeq
where $\tp{(i)}:= 0$ if $i\in\{ 1 ,\ldots, M \}$ and $\tp{(i)}:= 1$ otherwise. 

By calculation one easily shows that 
\beq\label{etilrel}
\begin{split}
[\tilde e_{i,j}, \tilde e_{h,k}]
=&\,(\tilde{e}_{i,k} - \delta_{i,k} (-1)^{\tp(i)} \rho_{\col(i)})\delta_{h,j}\\
&\,- (-1)^{(\tp(i)+\tp(j))(\tp(h)+\tp(k))}
\delta_{i,k} (\tilde e_{h,j} - \delta_{h,j} (-1)^{\tp(j)}\rho_{\col(j)}).
\end{split}
\eeq
The homomorphism $\rU(\mathfrak{m}) \rightarrow \mathbb{C}$ induced by the character $\chi$ can be explicitly described as follows.
For any $i,j\in \I$, we have
\beq\label{chidef}
\chi (\tilde e_{i,j}) = \left\{
\begin{array}{ll}
(-1)^{\tp(i)}&\hbox{if $\row(i)=\row(j)$ and $\col(i) = \col(j)+1$;}\\[3mm]
0&\hbox{otherwise.}
\end{array}\right.
\eeq

Now we recall certain elements in the universal enveloping superalgebra $\rU(\gl_{M|N})$ defined in \cite{Pe21}. Although we slightly modify the expressions from \cite{Pe21}, the same series \eqref{newtseries} will be obtained.
For $1\lle i,j\lle m+n$ and integer $0\lle x\lle m+n$,
we set
$$\sfT_{ i,j;x}^{(0)} :=  \left\{
\begin{array}{ll}
- \delta_{i,j}&\hbox{if $i\lle x$;}\\[3mm]
 \delta_{i,j}&\hbox{otherwise.}
\end{array}\right.$$
and for $r \gge 1$ we define
\beq\label{mystery}
\sfT_{i,j;x}^{(r)}
:=
\sum_{s = 1}^r(-1)^{r-s}
\sum_{\substack{i_1,\dots,i_s\\j_1,\dots,j_s}}
(-1)^{\#\{1\lle t\lle s-1\,|\, \text{row}(j_t)\lle x\}+\tp(i_1)+\cdots+\tp(i_s)}
 \tilde e_{i_1,j_1} \cdots \tilde e_{i_s,j_s}
\eeq
where the second sum is taken over all $i_1,\dots,i_s,j_1,\dots,j_s\in \I$
such that
\begin{itemize}
\item[(1)] $\sum_{1\lle t\lle s}(\col(j_t)-\col(i_t))= r-s$;
\item[(2)] $\col(i_t) \lle \col(j_t)$ for each $t=1,\dots,s$;
\item[(3)] if $\text{row}(j_t)>x$, then
$\col(j_t) < \col(i_{t+1})$ for each
$t=1,\dots,s-1$;
\item[(4)]
if $\text{row}(j_t)\lle x$, then $\col(j_t) \gge \col(i_{t+1})$
for each
$t=1,\dots,s-1$;
\item[(5)] $\row(i_1)=i$, $\row(j_s) = j$;
\item[(6)]
$\row(j_t)=\row(i_{t+1})$ for each $t=1,\dots,s-1$.
\end{itemize}
Due to conditions (1) and (2), $\sfT_{i,j;x}^{(r)}$ belongs to $\mathrm{F}_r \rU(\fkp)$. 
We further define the following series for all $1\lle i,j\lle m+n$:
\beq\label{newtseries}
\sfT_{i,j;x}(u) := \sum_{r \gge 0} \sfT_{i,j;x}^{(r)} u^{-r}
\in \rU(\fkp) [\![u^{-1}]\!].
\eeq
In particular, if either $p_i=0$ or $p_j=0$, we set $\sfT_{i,j;x}(u)=\sfT_{i,j;x}^{(0)}$.

We are now able to recall the following result relating the shifted super Yangians and finite $W$-superalgebras. 
\begin{thm}[{\cite[Coro. 9.3 \& Thm. 10.1]{Pe21}}]\label{thm:pe21}
The elements
\begin{align}\label{Wdi}
&\{\sfT_{i,i;i-1}^{(r)} \, | \, 1\lle i\lle m+n, \, r> s_{i,i}\}\\ \label{Wei}
&\{\sfT_{i,i+1;i}^{(r)} \, | \, 1\lle i\lle m+n-1, \, r> s_{i,i+1}\}\\ \label{Wfi}
&\{\sfT_{i+1,i;i}^{(r)} \, | \, 1\lle i\lle m+n-1, \, r> s_{i+1,i}\}
\end{align}
in $\rU(\fkp)$ defined by \eqref{mystery} are $\fkm$-invariants and they generate the subalgebra $W(\pi)$. 
Moreover, there exists a unique surjective homomorphism 
\beq \label{omegadef}
\Omega: \rY_{m|n}(\sigma) \twoheadrightarrow W(\pi)
\eeq
such that 
\beq\label{omega2}
\Omega(D_i^{(r)})=\sfT_{i,i;i-1}^{(r)}, \quad \Omega(E_i^{(r)})=\sfT_{i,i+1;i}^{(r)}, \quad \Omega(F_i^{(r)})=\sfT_{i+1,i;i}^{(r)}, 
\eeq
where the kernel of $\Omega$ is the two-sided ideal generated by $\{D_{1}^{(r)}\,|\, r> p_1\}$. 
\end{thm}

In view of Theorem \ref{thm:pe21}, we define the \textit{truncated shifted super Yangian} $\rY_{m|n}^\ell(\sigma)$ to be the quotient of $\rY_{m|n}(\sigma)$ by the two-sided ideal generated by the elements $\{D_1^{(r)}\mid r>p_1\}$. As a consequence of Theorem \ref{thm:pe21}, we identify $\rY_{m|n}^\ell(\sigma)\cong W(\pi)$ from now on.
In particular, under this identification, the elements $D_i^{(r)}$, $E_i^{(r)}$ and $F_i^{(r)}$ in the quotient  $\rY_{m|n}^\ell(\sigma)$ are identified with corresponding elements in $W(\pi)$ by \eqref{omega2}.
Although omitted in the notation, the parity sequence $\fks$ in the definition of $\rY_{m|n}(\sigma)$ is obtained by reading the parity of each row of $\pi$, from top to bottom.

The following PBW theorem for truncated shifted super Yangian $\rY_{m|n}^\ell(\sigma)$ is known.
\begin{prop}[{\cite[Coro.~8.2]{Pe21}}]\label{prop:PBW-truncated}
The supermonomials in the elements
\begin{align*}
&\{D_i^{(r)}\mid 1\lle i\lle m+n,0<r\lle p_i\},\\
&\{E_{i,j}^{(r)}\mid1\lle i<j\lle m+n,s_{i,j}<r\lle s_{i,j}+p_i\},\\
&\{F_{j,i}^{(r)}\mid1\lle i<j\lle m+n,s_{j,i}<r\lle s_{j,i}+p_i\}
\end{align*}
taken in any fixed order form a basis for $\rY_{m|n}^\ell(\sigma)$.
\end{prop}

The following theorem is a super generalization of \cite[Thm.~3.5]{BK08}, which turns out to be crucial later. 
\begin{thm}\label{thm:vanish}
The generators $T_{i,j}^{(r)}$ of $\rY_{m|n}^\ell(\sigma)$ are zero for all $1\lle i,j\lle m+n$ and $r>s_{i,j}+p_{\min(i,j)}$.
\end{thm}
The proof of Theorem \ref{thm:vanish} is lengthy and technical, thus we postpone it to the Appendix \ref{sec:app}.

\begin{prop}\label{prop:PBW-T-truncated}
The supermonomials in the elements 
\be
\{T_{i,j}^{(r)}\mid 1\lle i,j\lle m+n,s_{i,j}<r\lle s_{i,j}+p_{\min(i,j)}\}
\ee
taken in some fixed order form a basis for $\rY_{m|n}^\ell(\sigma)$.
\end{prop}
\begin{proof}
Follows from the same argument as in the proof of Lemma \ref{lem:PBW-T} using Theorem \ref{thm:vanish}.
\end{proof}

\subsection{Parabolic induction}
The goal of this subsection is to construct certain homomorphisms, generalizing the notion of baby comultiplications defined in \cite[Thm 6.1]{Pe21}. These homomorphisms serve as comultiplications for $W$-superalgebras and shifted super Yangians, allowing one to construct the tensor product of modules under mild assumptions on parity sequence. In particular, most results in this subsection are generalizations of corresponding results of \cite[\textsection 11]{BK06} to general linear Lie superalgebras.

Choose non-negative integers $\ell^\prime$ and $\ell^{\dprime}$ such that $\ell=\ell^\prime + \ell^{\dprime}$. 
The leftmost $\ell^\prime$ columns of $\pi$ form a pyramid, which we denote by $\pi^\prime$.
Similarly, let $\pi^\dprime$ be the pyramid consisting of the rightmost $\ell^\dprime$ columns of $\pi$.
In this situation, we write the decomposition as  $\pi=\pi^\prime\otimes \pi^\dprime$. For example, 
\beq\label{ex:pi1}
\ytableausetup{centerboxes,boxsize=1.1em}
\begin{ytableau}
\none & \ovl{2} & \ovl{4} \\
\none & 1 & 3 & 5 \\
\none & 2 & 4 & 6 & 7 \\
\ovl{1} & \ovl{3} & \ovl{5} & \ovl{6} & \ovl{7} & \ovl{8}
\end{ytableau}
\quad = \quad
\begin{ytableau}
\none & \ovl{2} & \ovl{4} \\
\none & 1 & 3 \\
\none & 2 & 4 \\
\ovl{1} & \ovl{3} & \ovl{5}
\end{ytableau}
\quad \otimes \quad
\begin{ytableau}
\none \\
1 \\
2 & 3 \\
\ovl{1} & \ovl{2} & \ovl{3}
\end{ytableau}
\eeq

Suppose that there are $M^\prime$ even and $N^\prime$ odd boxes in $\pi^\prime$, respectively.
Let $\fkp^\prime$, $\fkh^\prime$, $\fkm^\prime$ be the subalgebras of $\g^\prime=\gl_{M^\prime | N^\prime}$ 
defined with respect to $\pi^\prime$, and let $W(\pi^\prime)=\rU(\fkp^\prime)^{{\fkm}^\prime}$ be the $W$-superalgebra associated with $\pi^\prime$.
Similarly, suppose that there are $M^\dprime$ even and $N^\dprime$ odd boxes in $\pi^\dprime$, respectively.
One defines the subalgebras $\fkp^\dprime$, $\fkh^\dprime$, $\fkm^\dprime$ of $\g^\dprime=\gl_{M^\dprime | N^\dprime}$ 
and $W(\pi^\dprime)=\rU(\fkp^\dprime)^{\fkm^\dprime}$, where $M=M^\prime+M^\dprime$ and $N=N^\prime+N^\dprime$.

For all $1\lle j\lle N'$, let $\ovl{j}=M'+j$ be a shorthand notation.
Following exactly the same rule (\ref{etilde}), one defines elements $\tilde{e}_{ij}$ in $\rU(\fkp^\prime)$ for all $i,j\in \I^\prime=\{1,\ldots, M^\prime,\ovl{1}, \ldots, \ovl{N^\prime}\}$ according to $\pi^\prime$.
Similarly, elements $\tilde{e}_{ij}$ in $\rU(\fkp^\dprime)$ for all $i,j\in \I^\dprime=\{1,\ldots, M^\dprime,\ovl{1}, \ldots, \ovl{N^\dprime}\}$ are defined according to $\pi^\dprime$ with similar shorthand notation.

Define a homomorphism $\Delta_{\ell^\prime,\ell^\dprime}:\rU(\fkp)\rightarrow \rU(\fkp^\prime)\otimes \rU(\fkp^\dprime)$ by 
\beq\label{wgcoprod}
\Delta_{\ell^\prime,\ell^\dprime}(\tilde e_{i,j}):= \left\{
\begin{array}{ll}
\tilde e_{i,j}\otimes 1 &\hbox{if $ \col (j)\lle \ell^\prime$,}\\
0 &\hbox{if $ \col (i)\lle \ell^\prime$ and $ \ell^\prime+1\lle \col (j)$,}\\
1\otimes \tilde e_{\eta(i),\eta(j)} &\hbox{if $\ell^\prime+1 \lle \col(i)$,}
\end{array}
\right.
\eeq
where 
\be
\eta(i):= \left\{
\begin{array}{ll}
i-M^\prime &\hbox{if $ \tp(i)=0$,}\\[2mm]
\overline{i-N^\prime} &\hbox{if $ \tp(i)=1$,}
\end{array}
\right.
\ee
for all $i,j\in \I$ satisfying $\col(i)\lle \col(j)$.

By definition, $\Delta_{\ell^\prime,\ell^\dprime}$ is a filtered map with respect to Kazhdan filtrations.
We spell the result of $\Delta_{\ell^\prime,\ell^\dprime}$ when applied to the distinguished elements $\sfT_{i,j;0}^{(r)}$ in the following lemma.
\begin{lem}
For $1\lle i,j \lle m+n$, $r>0$,
\beq\label{coproductW}
\Delta_{\ell^\prime,\ell^\dprime}(\sfT_{i,j;0}^{(r)})=\sum_{s=0}^{r} \sum_{k=1}^{m+n} \sfT_{i,k;0}^{(s)}\otimes \sfT_{k,j;0}^{(r-s)}.
\eeq
\end{lem}
\begin{proof}
It follows from the explicit formula of $\sfT_{i,j;0}^{(r)}$ and (\ref{wgcoprod}).
\end{proof}

We now provide another point of view about how $\Delta_{\ell^\prime,\ell^\dprime}$ is defined.
Consider the standard embedding of $\g^\prime\oplus\g^\dprime$ into $\g$. It embeds $\fkp^\prime\oplus\fkp^\dprime$ into $\fkp$ and $\fkm^\prime\oplus\fkm^\dprime$ into $\fkm$. We may think the map $\Delta_{\ell^\prime,\ell^\dprime}$ is induced by the natural projection $\fkp\twoheadrightarrow \fkp^\prime\oplus\fkp^\dprime$ composed by a constant shift.
Define $\chi^\prime\oplus\chi^\dprime:\fkm^\prime\oplus\fkm^\dprime\rightarrow \mathbb{C}$ by taking the supertrace form with $e^\prime+e^\dprime$, where $e^\prime\in\g^\prime$ and $e^\dprime\in\g^\dprime$ are nilpotent elements associated to $\pi^\prime$ and $\pi^\dprime$, respectively. 
Then $\chi^\prime\oplus\chi^\dprime$ coincides with the restriction of the character $\chi:\fkm\rightarrow \mathbb{C}$ defined by taking the supertrace form with the nilpotent element $e\in\g$ associated to $\pi$. As a result, $\Delta_{\ell^\prime,\ell^\dprime}$ sends $\chi$-twisted $\fkm$-invariants in $\rU(\fkp)$ to $(\chi^\prime\oplus\chi^\dprime)$-twisted $(\fkm^\prime\oplus\fkm^\dprime)$-invariants in $\rU(\fkp^\prime\oplus\fkp^\dprime)\cong \rU(\fkp^\prime)\otimes \rU(\fkp^\dprime)$. Therefore, the restriction of $\Delta_{\ell^\prime,\ell^\dprime}$ to $W(\pi)$ defines a homomorphism
\beq\label{Wdel}
\Delta_{\ell^\prime,\ell^\dprime} : W(\pi) \rightarrow W(\pi^\prime)\otimes W(\pi^\dprime),
\eeq
which is again a filtered map with respect to Kazhdan filtrations.

We call the map \eqref{Wdel} {\em parabolic induction}, naming after similar homomorphisms defined for affine $W$-algebras \cite{Ge20, KU22}. In fact, it can be traced back to a functor defined by Losev \cite{Lo11}, relating the module categories for different finite $W$-algebras.
As we shall see in Lemma \ref{nwbaby}, it is also a generalization of the baby comultiplications in \cite[(8.5), (8.6)]{Pe21}.
The parabolic induction is coassociative, as indicated in the following lemma.
\begin{lem}\label{wasso}
Suppose $\ell=\ell^\prime+\ell^\dprime+\ell^\trprime$ for some non-negative integers $\ell^\prime,\ell^\dprime,\ell^\trprime$ and let $\pi=\pi^\prime\otimes\pi^\dprime\otimes\pi^\trprime$, where $\pi^\prime$ means the pyramid consisting of the leftmost $\ell^\prime$ columns of $\pi$,
$\pi^\dprime$ means the pyramid consisting of the middle $\ell^\dprime$ columns of $\pi$,
and $\pi^\trprime$ means the pyramid consisting of the rightmost $\ell^\trprime$ columns of $\pi$.
Then the following diagram commutes.
\be
\begin{CD}
W(\pi^\prime \otimes \pi^\dprime \otimes \pi^\trprime) & @>\Delta_{\ell^\prime+\ell'',\ell'''}>> & W(\pi^\prime \otimes \pi'') \otimes W(\pi''')\\
@V\Delta_{\ell^\prime,\ell^\dprime+\ell^\trprime} VV&&@VV\Delta_{\ell^\prime,\ell^\dprime} \otimes 1V\\
W(\pi^\prime) \otimes W(\pi'' \otimes \pi''') &@>1 \otimes \Delta_{\ell^\dprime,\ell^\trprime}>> &W(\pi^\prime) \otimes W(\pi'')
\otimes W(\pi''')
\end{CD}
\ee
\end{lem}
\begin{proof}
It follows from \eqref{wgcoprod}.        
\end{proof}

Let $(\sigma, \ell, \fks)$ denote the unique triple corresponding to $\pi$ and keep in mind the isomorphism $\rY_{m|n}^\ell(\sigma)\cong W(\pi)$ in Theorem \ref{thm:pe21}. Let $t={\rm min} \{q_1,q_\ell\}$ denote the height of the shortest column of $\pi$.
Recall the baby comultiplications $\Delta_L$ and $\Delta_R$ defined in \cite[(8.5), (8.6)]{Pe21}. 

Suppose first that $\Delta_R:\rY_{m|n}^\ell(\sigma)\rightarrow \rY_{m|n}^{\ell-1}(\dot\sigma)\otimes \rU(\gl_{m_\ell|n_\ell})$ is defined, which means that $q_\ell\lle q_1$ and there are $m_\ell$ even and $n_\ell$ odd boxes in the rightmost column of $\pi$. In this case $\dot\sigma$ is given by \cite[(6.1)]{Pe21}. Consider the triple $(\dot\sigma,\ell-1,\fks)$. This triple corresponds to the pyramid $\pi^\prime$ obtained by removing the rightmost column of $\pi$. After identifying $W(\pi^\prime)$ with $\rY_{m|n}^{\ell-1}(\dot\sigma)$, we may describe the map as $\Delta_R:W(\pi)\rightarrow W(\pi^\prime)\otimes \rU(\gl_{m_\ell|n_\ell})$, just like $\Delta_{\ell-1,1}$ in \eqref{Wdel}.

Suppose instead that $\Delta_L:\rY_{m|n}^{\ell}(\sigma)\rightarrow \rU(\gl_{m_1|n_1})\otimes\rY_{m|n}^{\ell-1}(\dot\sigma)$ is defined, which means that $q_1\lle q_\ell$ and there are $m_1$ even and $n_1$ odd boxes in the leftmost column of $\pi$. In this case $\dot\sigma$ is given by \cite[(6.2)]{Pe21}. Then the triple $(\dot\sigma,\ell-1,\fks)$ corresponds to the pyramid $\pi^\dprime$ obtained by removing the leftmost column of $\pi$. By similar identifications, we have $\Delta_L:W(\pi)\rightarrow \rU(\gl_{m_1|n_1})\otimes W(\pi^\dprime)$, just like $\Delta_{1,\ell-1}$ in \eqref{Wdel}.
The next result shows that the baby comultiplications \cite[(8.5), (8.6)]{Pe21} are special cases of the parabolic induction \eqref{Wdel}.

\begin{lem}\label{nwbaby}
Whenever the baby comultiplications are defined, we have $\Delta_R=\Delta_{\ell-1,1}$ and $\Delta_L=\Delta_{1,\ell-1}$, respectively.
\end{lem}
\begin{proof}
The statement is trivial when $\ell=1$ so we assume $\ell\gge 2$. Suppose first that $\Delta_R$ is defined.   Examining \eqref{wgcoprod}, one finds that the map $\Delta_{\ell-1,1}$ coincides with the injective homomorphism $\psi_R:W(\pi)\rightarrow \rU(\dot{\fkp})\otimes \rU(\gl_{m_\ell|n_\ell})$ defined in the proof of \cite[Thm. 10.1]{Pe21}. Since the composition $\psi_R^{-1}\circ \Delta_R:\rY_{m|n}^{\ell}(\sigma)\rightarrow W(\pi)$ is exactly the identification $\rY_{m|n}^\ell(\sigma)\cong W(\pi)$, we deduce that $\Delta_{\ell-1,1}=\Delta_R$. The other case where $\Delta_L$ is defined is similar, using instead the injective homomorphism $\psi_L:W(\pi)\rightarrow \rU(\gl_{m_1|n_1})\otimes \rU(\dot{\fkp})$ defined in the proof of \cite[Thm. 10.1]{Pe21}, where the composition $\psi_L^{-1}\circ \Delta_L:\rY_{m|n}^{\ell}(\sigma)\rightarrow W(\pi)$ is our identification.
\end{proof}

Iterating the parabolic induction \eqref{Wdel} $(\ell-1)$ times, which can be done in any order with the help of Lemma \ref{wasso}, we split the pyramid $\pi$ into its individual columns and obtain the following homomorphism
\beq\label{miura}
\mu:W(\pi)\rightarrow \rU(\gl_{m_1|n_1})\otimes \cdots \otimes \rU(\gl_{m_{\ell}|n_\ell}).
\eeq
Recall that $\h\cong\gl_{m_1|n_1}\oplus \cdots\oplus \gl_{m_\ell|n_\ell}$.
For each $1\lle r\lle \ell$ and $1\lle i\lle q_r$, let $\pi(r,i)\in \I$ denote the entry of the $i^{th}$ box, counted from top to bottom, in the $r^{th}$ column of $\pi$. Define $e_{i,j}^{[r]}:=e_{\pi(r,i), \pi(r,j)}$, for all $1\lle r\lle \ell$, $1\lle i,j\lle q_r$. Then $\{e_{i,j}^{[r]} \, | \, 1\lle i,j\lle q_r, 1\lle r\lle \ell\}$ forms a basis for $\h$.
Under the identification $\rU(\h)\cong \rU(\gl_{m_1|n_1})\otimes \cdots \otimes \rU(\gl_{m_\ell|n_\ell})$, $e_{i,j}^{[r]}$ is identified with 
\be
1^{\otimes(r-1)}\otimes e_{\eta(\pi(r,i)),\eta(\pi(r,j))} \otimes 1^{(\ell-r)},
\ee
where 
\[
\eta(\pi(r,i))=\begin{cases}
    \pi(r,i)-(m_1+\cdots+m_{r-1}), & \text{ if }\fks_i=1,\\[2mm]
    \overline{\pi(r,i)-(n_1+\cdots+n_{r-1})}, &\text{ if }\fks_i=-1.
\end{cases}
\]
For example, take $\pi$ as the one in the left of \eqref{ex:pi1}. Then $e_{2,3}^{[4]}=e_{\pi(4,2),\pi(4,3)}=e_{6,\ovl{6}}\in \rU(\h)$, which is identified with 
\be
1^{\otimes 3}\otimes e_{2,\ovl{1}} \otimes 1^{\otimes 2}\in 
\rU(\gl_{0|1}) \otimes \rU(\gl_{2|2}) \otimes \rU(\gl_{2|2}) \otimes \rU(\gl_{2|1})\otimes \rU(\gl_{1|1}) \otimes \rU(\gl_{0|1}).
\ee
The map \eqref{miura} is then identified with a homomorphism
\beq\label{miura2}
\mu:W(\pi)\rightarrow \rU(\h),
\eeq
which is a filtered map where $W(\pi)$ is equipped with the Kazhdan filtration \eqref{kzdegdef} and $\rU(\h)$ is equipped with the standard filtration. 
Let $\gr\, \rU(\h)$ denote the corresponding associated graded superalgebra. 
Define following constant translation
\beq\label{zetatrans}
\zeta:\rU(\h)\rightarrow \rU(\h), \qquad e_{i,j}^{[r]}\mapsto e_{i,j}^{[r]}+\delta_{i,j}(-1)^{\tp{(i)}}(\check{q}_{r+1} +\cdots+\check{q}_\ell).
\eeq
Let $\xi:\rU(\fkp)\rightarrow \rU(\h)$ be the homomorphism induced by the natural projection $\fkp\twoheadrightarrow\h$. Comparing \eqref{rhodef}, \eqref{etilde}, and \eqref{miura2} as iterations of \eqref{wgcoprod}, one observes that the restriction of the composition $\zeta\circ\xi$ to $W(\pi)$ is exactly $\mu$. This homomorphism \eqref{miura2} is called the {\em Miura transform}. In fact, not only $\mu$ but also the parabolic induction $\Delta_{\ell^\prime,\ell^\dprime}$ is injective, as stated in the following results, generalizing \cite[Thm. 11.4, Coro. 11.5]{BK06}.
\begin{thm}\label{muinj}
The map $\gr\, \mu: \gr\,W(\pi)\rightarrow \gr\,\rU(\h)$ is injective and hence so is the Miura transform $\mu:W(\pi)\rightarrow \rU(\h)$. 
\end{thm}
\begin{proof}
As \eqref{miura}, $\mu$ can be understood as a composition of $(\ell-1)$ copies of $\Delta_L$ or $\Delta_R$ in any order which makes sense. Since both $\gr\,\Delta_L$ and $\gr\,\Delta_R$ are injective \cite[Thm.~8.2]{Pe21}, the statement follows. 
\end{proof}

\begin{cor}\label{popinj}
The map $\gr\,\Delta_{\ell^\prime,\ell^\dprime}: \gr\,W(\pi)\rightarrow \gr\,W(\pi^\prime)\otimes W(\pi^\dprime)$ is injective for any $\ell=\ell^\prime+\ell^\dprime$ and hence so is the parabolic induction $\Delta_{\ell^\prime,\ell^\dprime}: W(\pi)\rightarrow W(\pi^\prime)\otimes W(\pi^\dprime)$.
\end{cor}
\begin{proof}
$\mu$ can be understood as a composition of $\Delta_{\ell^\prime,\ell^\dprime}$ followed by $(\ell-2)$ more maps.       
\end{proof}

Recall that if $\vec{\pi}$ is a pyramid obtained by horizontally moving some rows of $\pi$, then there is a canonical isomorphism induced from \eqref{iotadef},
\beq\label{Wiota}
\iota:W(\pi)\rightarrow  W(\vec{\pi}).
\eeq
The parabolic induction is compatible with these isomorphisms, as stated in the following lemma, generalizing \cite[Lem. 11.6]{BK06}.
\begin{lem}\label{lem:paraiota}
Suppose that $\pi=\pi^\prime\otimes\pi^\dprime$ and $\vec{\pi}={\vec\pi}^\prime\otimes{\vec\pi}^\dprime$ are pyramids such that 
${\vec\pi}^\prime$ and ${\vec\pi}^\dprime$ have same row lengths as $\pi^\prime$ and $\pi^\dprime$, respectively. Then the following diagram commutes:
\be
\begin{CD}
W(\pi) & @>\iota>> & W(\vec\pi)\\
@V\Delta_{\ell^\prime,\ell^\dprime} VV&&@VV\Delta_{\ell^\prime,\ell^\dprime} V\\
W(\pi^\prime) \otimes W(\pi^\dprime) &@> \iota \otimes \iota >> &  W(\vec\pi^\prime) \otimes W(\vec\pi^\dprime),
\end{CD}
\ee
where $\ell^\prime,\ell^\dprime$ are levels of $\pi^\prime, \pi^\dprime$, respectively, with $\ell=\ell^\prime+\ell^\dprime$.
\end{lem}
\begin{proof}
We argue by induction on $\ell$. There is nothing to prove if either $\ell^\prime=0$ or $\ell^\dprime=0$, hence the statement holds trivially in the base case $\ell=1$. Assume from now on $\ell=\ell^\prime+\ell^\dprime$ with $\ell^\prime, \ell^\dprime>0$.
In any case, at least one of $\Delta_L$ and $\Delta_R$ is defined. We explain the inductive step in the case where $\Delta_R$ is defined, where in the other case the argument is almost identical.
Note that when $\Delta_R$ is defined on $W(\pi)$, by our assumptions on the lengths of rows, we ensure that $\Delta_R$ is defined on $W(\vec\pi)$ as well.

Consider first that $\ell^\dprime=1$. By Lemma \ref{nwbaby}, $\Delta_{\ell^\prime,\ell^\dprime}=\Delta_R$ for both $W(\pi)$ and $W(\vec\pi)$. Then the commutativity of the diagram can be checked by the explicit formula of $\Delta_R$ and $\iota$ in \cite[Thm.~6.1 (1)]{Pe21} and \cite[(5.23)]{Pe21} when they evaluate on parabolic generators.

Suppose now $\ell^\dprime>1$. Let $\rho, \hat\rho, \rho'$ and $\hat\rho^\prime$ denote the pyramids obtained by removing the rightmost column (which is of height $q_\ell$ in all cases) of $\pi,\vec\pi, \pi''$ and $\vec\pi''$, respectively. Consider the following cube
\be
\begin{picture}(290,130)
\put(10,10){\makebox(0,0){$W(\pi')\otimes W(\pi'')$}}
\put(200,10){\makebox(0,0){$W(\pi')\otimes W(\rho')\otimes \rU(\mathfrak{gl}_{m_\ell|n_\ell})$}}
\put(15,78){\makebox(0,0){$W(\pi)$}}
\put(180,78){\makebox(0,0){$W(\rho)\otimes \rU(\mathfrak{gl}_{m_\ell|n_\ell})$}}

\put(85,47){\makebox(0,0){$W(\vec\pi')\otimes W(\vec\pi'')$}}
\put(270,47){\makebox(0,0){$W(\vec\pi')\otimes W(\hat\rho')\otimes \rU(\mathfrak{gl}_{m_\ell|n_\ell})$}}
\put(72,115){\makebox(0,0){$W(\vec\pi)$}}
\put(260,115){\makebox(0,0){$W(\hat\rho)\otimes \rU(\mathfrak{gl}_{m_\ell|n_\ell})$}}
\put(43,33){\makebox(0,0){$\nearrow$}}
\put(43,33){\line(-1,-1){13}}
\put(43,103){\makebox(0,0){$\nearrow$}}
\put(43,103){\line(-1,-1){13}}
\put(223,33){\makebox(0,0){$\nearrow$}}
\put(223,33){\line(-1,-1){13}}
\put(223,103){\makebox(0,0){$\nearrow$}}
\put(223,103){\line(-1,-1){13}}
\put(10,27){\makebox(0,0){$\downarrow$}}
\put(9.9,27){\line(0,1){38}}
\put(170,27){\makebox(0,0){$\downarrow$}}
\put(169.9,27){\line(0,1){38}}
\put(110,79){\makebox(0,0){$\rightarrow$}}
\put(110,79.56){\line(-1,0){49}}
\put(110,9){\makebox(0,0){$\rightarrow$}}
\put(110,9.56){\line(-1,0){49}}
\put(70,64){\makebox(0,0){$\downarrow$}}
\put(69.9,64){\line(0,1){13}}
\put(69.9,82){\line(0,1){20}}
\put(245,64){\makebox(0,0){$\downarrow$}}
\put(244.9,64){\line(0,1){38}}
\put(178,114){\makebox(0,0){$\rightarrow$}}
\put(178,114.56){\line(-1,0){49}}
\put(178,46){\makebox(0,0){$\rightarrow$}}
\put(167,46.56){\line(-1,0){38}}
\end{picture}
\ee
where the maps on the front and back squares are defined from the parabolic inductions, and the remaining maps are induced by isomorphisms $\iota$.
The top and bottom squares are commutative by the special case $\ell^\dprime=1$ considered in the previous paragraph. 
The front and back squares are commutative by Lemma \ref{wasso}. 
The square in the right is commutative by the induction hypothesis. Since the horizontal maps are parabolic inductions, all of them are injective by Corollary \ref{popinj}. This implies that the square on the left is commutative, finishing the inductive step.  
\end{proof}

In the remaining part of this section, we will lift the parabolic induction from the truncated shifted super Yangian $\rY_{m|n}^\ell(\sigma)$ to $\rY_{m|n}(\sigma)$.
We first introduce the notion of column removal on pyramids.
Let $1\lle i_1 < \cdots < i_{\dot{\ell}}\lle \ell$ be a subset of the columns of $\pi$. 
Joining these columns together, we obtain a pyramid of column heights $(q_{i_1},\ldots,q_{i_{\dot\ell}})$, denoted by $\dot{\pi}$.
Let $\sigma=(s_{i,j})_{1\lle i,j\lle m+n}$ and $\dot\sigma=(\dot{s}_{i,j})_{1\lle i,j\lle m+n}$ denote the shift matrices corresponding to $\pi$ and $\dot\pi$, respectively.
By construction, we have $\dot{s}_{i,j}\lle s_{i,j}$ for all $i,j$. This implies that $\rY_{m|n}(\sigma) \subseteq \rY_{m|n}(\dot{\sigma})$ when identifying them as subalgebras of $\rY_{m|n}$. This canonical embedding $\rY_{m|n}(\sigma) \hookrightarrow \rY_{m|n}(\dot{\sigma})$ factors through the quotients since $p_1\gge \dot{p}_1$ by construction, inducing the following homomorphism
\beq\label{parsy}
\dag:\rY_{m|n}^{\ell}(\sigma) \rightarrow \rY_{m|n}^{\dot\ell}(\dot{\sigma}).
\eeq
Identifying $\rY_{m|n}^{\ell}(\sigma)$ with $W(\pi)$ and $\rY_{m|n}^{\dot\ell}(\dot\sigma)$ with $W(\dot\pi)$, it defines a homomorphism
\beq\label{parw}
\dag:W(\pi) \rightarrow W(\dot\pi)
\eeq
sending the generators $D_i^{(r)}$, $E_i^{(r)}$, $F_i^{(r)}$ of $W(\pi)$ to $\dot{D}_i^{(r)}$, $\dot{E}_i^{(r)}$, $\dot{F}_i^{(r)}$ of $W(\dot\pi)$, respectively. 
Let $\mu: W(\pi)\rightarrow \rU(\h)$ and $\dot\mu:W(\dot\pi)\rightarrow \rU(\dot{\h})$ denote the Miura transforms in \eqref{miura2}. There exists an obvious projection $\hat\dag:\h\twoheadrightarrow \dot\h$ defined by
\be
\hat\dag(e_{i,j}^{[r]}):= \left\{
\begin{array}{ll}
e_{i,j}^{[s]} &\hbox{if $r=i_s$ for some $s=1,\ldots,\dot\ell$,}\\[2mm]
0 &\hbox{otherwise.}
\end{array}
\right.
\ee
\begin{lem}\label{commu}
With the above notation, the following diagram commutes
\be
\begin{CD}
W(\pi) &@>\mu >> &\rU(\h) \\
@V\dag VV&&@VV\hat\dag V\\
W(\dot\pi)&@>\dot\mu >> & \rU(\dot{\h})
\end{CD}
\ee
\end{lem}
\begin{proof}
One checks explicitly from \eqref{Wdi}--\eqref{Wfi}, \eqref{omega2} and definition
\eqref{mystery} that
$\hat\dag$ maps the elements
$\mu(D_i^{(r)})$,
$\mu(E_i^{(r)})$ and $\mu(F_i^{(r)})$
to $\dot\mu(\dot D_i^{(r)})$,
$\dot\mu(\dot E_i^{(r)})$ and
$\dot\mu(\dot F_i^{(r)})$.
\end{proof}

\begin{cor}\label{comWdel}
Suppose that $\pi=\pi^\prime\otimes\pi^\dprime$ and $\dot\pi=\dot{\pi}^\prime \otimes \dot{\pi}^\dprime$, where $\dot\pi^\prime$ and $\dot{\pi}^\dprime$ are obtained by removing columns from $\pi^\prime$ and $\pi^\dprime$, respectively. Then we the following commutative diagram:
\be
\begin{CD}
W(\pi) & @>\Delta_{\ell^\prime,\ell^\dprime}>> & W(\pi^\prime)\otimes W(\pi^\dprime)\\
@V\dag VV&&@VV \dag \otimes \dag V\\
W(\dot\pi) &@> \Delta_{\dot{\ell}^\prime,\dot{\ell}^\dprime} >> &  W(\dot{\pi}^\prime)\otimes W(\dot{\pi}^\dprime),
\end{CD}
\ee
where $\ell^\prime,\ell^\dprime$ are levels of $\pi^\prime, \pi^\dprime$ 
and $\dot{\ell}^\prime,\dot{\ell}^\dprime$ are levels of $\dot{\pi}^\prime, \dot{\pi}^\dprime$, respectively.
\end{cor}
\begin{proof}
In view of Lemma~\ref{commu} and the injectivity of the Miura transforms, the statement follows from the following commutative diagram:
\be
\begin{CD}
\rU(\h) &@>\sim >> &\rU(\h^\prime) \otimes \rU(\h^\dprime)\\
@V\hat\dag VV&&@VV\hat\dag \otimes \hat\dag V\\
\rU(\dot{\h}) &@>\sim >> &\rU(\dot{\h}^\prime) \otimes
\rU(\dot{\h}^\dprime),
\end{CD}
\ee
where the horizontal maps are induced by the obvious isomorphisms $\h \cong \h^\prime \oplus \h^\dprime$ and $\dot{\h} \cong \dot{\h}^\prime \oplus \dot{\h}^\dprime$.
\end{proof}

The next is the main result of this section.
\begin{thm}\label{thm:grownup}
Let $\sigma$ be a shift matrix and write $\sigma = \sigma^\prime + \sigma^\dprime$ 
where $\sigma^\prime$ is strictly lower triangular and
$\sigma^\dprime$ is strictly upper triangular. 
Identify $\rY_{m|n}(\sigma), \rY_{m|n}(\sigma^\prime)$ and $\rY_{m|n}(\sigma^\dprime)$ as subalgebras of $\rY_{m|n}$. Then the restriction of the comultiplication
$\Delta:\rY_{m|n}\rightarrow \rY_{m|n} \otimes \rY_{m|n}$ 
gives a homomorphism
\be
\Delta:\rY_{m|n}(\sigma)\rightarrow \rY_{m|n}(\sigma^\prime) \otimes \rY_{m|n}(\sigma^\dprime).
\ee
Moreover, for $\ell^\prime \gge s_{m+n,1}, \ell^\dprime \gge s_{1,m+n}$ and $\ell = \ell^\prime + \ell^\dprime$,
this map $\Delta$ factors through the quotients to define a homomorphism
$\rY_{m|n}^\ell(\sigma)
\rightarrow \rY_{m|n}^{\ell^\prime}(\sigma^\prime) \otimes \rY_{m|n}^{\ell^\dprime}(\sigma^\dprime)$
which, after identifying
$\rY_{m|n}^{\ell}(\sigma)\cong W(\pi)$, 
$\rY_{m|n}^{\ell^\prime}(\sigma^\prime) \cong W(\pi^\prime)$ and
$\rY_{m|n}^{\ell^\dprime}(\sigma^\dprime)\cong W(\pi^\dprime)$,
is precisely the parabolic induction $\Delta_{\ell^\prime,\ell^\dprime}$ defined in \eqref{Wdel}.
\end{thm}
\begin{proof}
Recall from \cite[Rmk. 8.5]{Pe21} that $\rY_{m|n}(\sigma), \rY_{m|n}(\sigma^\prime)$, and $\rY_{m|n}(\sigma^\dprime)$ are identified with the inverse limits ${\displaystyle  \lim_{\leftarrow}\rY_{m|n}^\ell(\sigma),  \lim_{\leftarrow}\rY_{m|n}^{\ell^\prime}(\sigma^\prime)}$, and ${\displaystyle\lim_{\leftarrow}\rY_{m|n}^{\ell^\dprime}(\sigma^\dprime)}$, respectively. 
The commutative diagram in Corollary \ref{comWdel} ensures that the maps $\Delta_{\ell^\prime,\ell^\dprime}$ are stable as $\ell^\prime, \ell^\dprime\rightarrow \infty$, showing the existence of the induced homomorphism 
\be
\lim_{\leftarrow}\Delta_{\ell^\prime,\ell^\dprime}:\rY_{m|n}(\sigma)\rightarrow \rY_{m|n}(\sigma^\prime)\otimes \rY_{m|n}(\sigma^\dprime),
\ee
lifting from the maps $\Delta_{\ell^\prime,\ell^\dprime}$ for all $\ell^\prime\gge s_{m+n,1}, \ell^\dprime\gge s_{1,m+n}$.
It remains to show that ${\displaystyle \lim_{\leftarrow}\Delta_{\ell^\prime,\ell^\dprime}}$ agrees with the restriction of $\Delta$ on $\rY_{m|n}(\sigma)$.

Let $X$ denote some generator $D_i^{(r)}, E_{i}^{(r)}$ or $F_{i}^{(r)}$ of $\rY_{m|n}(\sigma)\subseteq \rY_{m|n}$. Its image $\Delta(X)$ can be calculated as follows: first express $X$ in terms of a linear combination of monomials in $\sfT_{i,j}^{(r)}$ by the explicit formulas in \cite[Prop. 3.1]{Pe16}. 
Then by \eqref{coproduct}, $\Delta(X)$ equals a linear combination of monomials in $\sfT_{h,k}^{(s)}$.
Finally we rewrite all $\sfT_{h,k}^{(s)}$ appearing in the expression back in terms of the generators $D_i^{(r)}, E_{i}^{(r)}$ and $F_{i}^{(r)}$.
On the other hand, exactly the same procedure can be applied to calculate $\Delta_{\ell^\prime, \ell^\dprime}(X)$ where $X$ runs over generators $D_i^{(r)}, E_{i}^{(r)}$ or $F_{i}^{(r)}$ in $W(\pi)$ for each $\ell$, using \eqref{omega2} and \eqref{coproductW}.
The crucial point is that the formulas expressing $\sfT_{i,j;0}^{(r)}$ in terms of $D_i^{(r)}, E_{i}^{(r)}$ and $F_{i}^{(r)}$ in $\rU(\fkp)$ is completely the same as the explicit formula in \cite[Prop. 3.1]{Pe16}, for both of them are deduced from applying Gauss decomposition to an invertible matrix.
This shows that ${\displaystyle \lim_{\leftarrow}\Delta_{\ell^\prime,\ell^\dprime} =\Delta}$. 
\end{proof}
\begin{rem}\label{rem:row-shift}
In fact, {\em all} parabolic inductions $\Delta_{\ell^\prime,\ell^\dprime}:W(\pi)\rightarrow W(\pi^\prime)\otimes W(\pi^\dprime)$ can be expressed in terms of the comultiplication $\Delta$ on $\rY_{m|n}$ as follows.
Let $\vec{\pi}^\prime$ denote the right-justified pyramid with the same row lengths as $\pi^\prime$, and let $\vec{\pi}^\dprime$ denote the left-justified pyramid with the same row lengths as $\pi^\dprime$. Let $\vec{\pi}=\vec{\pi}^\prime \otimes \vec{\pi}^\dprime$.
Note that the parabolic induction $\Delta_{\ell^\prime,\ell^\dprime}:W(\pi)\rightarrow W(\pi^\prime)\otimes W(\pi^\dprime)$ can be recovered from $\Delta_{\ell^\prime,\ell^\dprime}:W(\vec\pi)\rightarrow W(\vec{\pi}^\prime)\otimes W(\vec{\pi}^\dprime)$ by Lemma \ref{lem:paraiota}, using the isomorphisms $\iota$. Now $\Delta_{\ell^\prime,\ell^\dprime}:W(\vec\pi)\rightarrow W(\vec{\pi}^\prime)\otimes W(\vec{\pi}^\dprime)$ is one of the comultiplications described by Theorem~\ref{thm:grownup}.
\end{rem}

Let $\pi$ and $\dot\pi$ be pyramids with exactly the same height and row lengths, in particular both of level $\ell$.
Let $\sigma$ and $\dot\sigma$ be their corresponding shift matrices, respectively. Pick positive integers $\ell^\prime$ and $\ell^\dprime$ so that $\ell=\ell^\prime+\ell^\dprime$ and $\pi=\pi^\prime\otimes\pi^\dprime$.
By horizontally moving some rows of $\pi^\prime$ and $\pi^\dprime$, one obtains a right-justified pyramid $\dot\pi^\prime$ and a left-justified pyramid $\dot\pi^\dprime$ so that $\dot\pi=\dot\pi^\prime\otimes \dot\pi^\dprime$.
The pair $(\dot\pi^\prime,\ell^\prime)$ determines a shift matrix $\dot\sigma^\prime$, and similarly $(\dot\pi^\dprime,\ell^\dprime)$ determines $\dot\sigma^\dprime$.
Then we have $\dot\sigma=\dot\sigma^\prime+\dot\sigma^\dprime$, where $\dot\sigma^\prime$ and $\dot\sigma^\dprime$ are lower and upper triangular, respectively.
As a consequence of \eqref{Wiota}, we have $W(\dot\pi)\cong W(\pi)$ and $W(\dot\pi^\prime)\otimes W(\dot\pi^\dprime)\cong W(\pi^\prime)\otimes W(\pi^\dprime)$.
Therefore we obtain the following surjective homomorphisms
\[
Y_{m|n}(\dot\sigma)\twoheadrightarrow W(\pi), \quad 
Y_{m|n}(\dot\sigma^\prime)\otimes Y_{m|n}(\dot\sigma^\dprime) \twoheadrightarrow W(\pi^\prime)\otimes W(\pi^\dprime).
\]
The following result follows from Theorem~\ref{thm:grownup} and Remark~\ref{rem:row-shift}.
\begin{cor}\label{SWcop}
The following diagram commutes:
\be
\begin{CD}
\rY_{m|n}(\dot{\sigma}) & @>\Delta >> & \rY_{m|n}(\dot{\sigma}^\prime)\otimes \rY_{m|n}(\dot{\sigma}^\dprime)\\
@V VV&&@VV  V\\
W(\pi) &@> \Delta_{\ell^\prime,\ell^\dprime} >> &  W(\pi^\prime)\otimes W(\pi^\dprime).
\end{CD}
\ee
\end{cor}

The comultiplication of a Hopf algebra allows one to construct the tensor product of modules. In our setup, the parabolic induction \eqref{Wdel} plays a similar role with an assumption depending on their parity sequences, which is a new phenomenon in $W$-superalgebra. 
Let $\pi^\prime$, $\pi^\dprime$ be two pyramids with $W(\pi^\prime), W(\pi^\dprime)$ their associated $W$-superalgebras, respectively. With the isomorphism \eqref{Wiota}, we may assume that $\pi^\prime$ is right-justified and $\pi^\dprime$ is left-justified. 
We further assume that $\pi^\prime$ and $\pi^\dprime$ are {\em $\fks$-compatible}, by which we mean that every box in each row of the joint pyramid $\pi:=\pi^\prime\otimes \pi^\dprime$ has the same parity. 
For example, the following two pyramids on the left are not $\fks$-compatible, but the two in the middle are. 
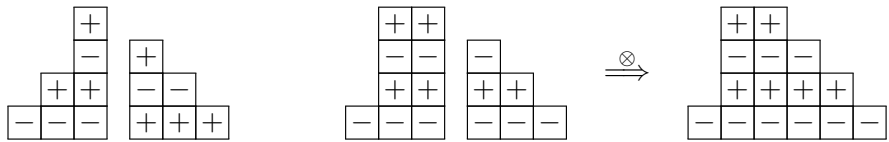
\[
\begin{ytableau}
\none & \none & + \\
\none & \none & - \\
\none & + & + \\
- & - & - 
\end{ytableau}
~~~
\begin{ytableau}
\none \\
+ \\
- & - \\
+ & + & +
\end{ytableau}
\qquad\qquad
\begin{ytableau}
\none & + & + \\
\none & - & - \\
\none & + & + \\
- & - & -
\end{ytableau}
~~~
\begin{ytableau}
\none \\
- \\
+ & + \\
- & - & -
\end{ytableau}
\quad \overset{\otimes}{\Longrightarrow} \quad
\begin{ytableau}
\none & + & + \\
\none & - & - & - \\
\none & + & + & + & + \\
- & - & - & - & - & - 
\end{ytableau}
\]

Let $M^\prime$ be a $W(\pi^\prime)$-module and let $M^\dprime$ be a $W(\pi^\dprime)$-module. With the assumption that $\pi^\prime$ and $\pi^\dprime$ are $\fks$-admissible, it makes sense to define $\pi:=\pi^\prime\otimes\pi^\dprime$. Then the outer tensor $M^\prime\boxtimes M^\dprime$ can be equipped with a $W(\pi)$-module structure by \eqref{Wdel}.

\begin{rem}\label{rem:parity}
Recall the observation in \cite{Pe14} (see also \cite{Ts20} for a proof) that the definition of the super Yangian $\rY_{m|n}$ is independent of the choice of the parity sequence $\fks$, up to isomorphism. A similar result holds for $\rY_{m|n}(\sigma)$ as follows. Let $\mu=(\mu_1,\ldots,\mu_z)$ be the minimal admissible shape of $\sigma$ and let $\fkS_\mu$ denote the corresponding Young subgroup of $\fkS_{m+n}$. Consider the obvious action of $\fkS_\mu$ permuting the parity sequence $\fks$. Then we have 
\beq\label{row-permute}
\rY_{m|n}(\sigma, \fks) \cong \rY_{m|n}(\sigma,\omega\fks), \quad \forall\, \omega\in \fkS_\mu,
\eeq
where the notation $\rY_{m|n}(\sigma, \fks)$ emphasizes the choice of the parity sequence $\fks$.
We first establish this fact for the truncation $\rY_{m|n}^\ell(\sigma, \fks)\cong \rY_{m|n}^\ell(\sigma,\omega\fks)$.
Indeed, this can be easily deduced from the fact that permuting the sequence $\fks$ is equivalent to permuting the rows of $\pi$, and any permutation in $\fkS_\mu$ does not change the corresponding nilpotent element \eqref{e_pi}. Since the definition of $W$-superalgebra depends only on the even nilpotent element $e$ \cite{GG02,Zha14}, we have $W(\pi)\cong W(\omega\pi)$, which further implies that 
\beq\label{row-permute-tr}
\rY_{m|n}^\ell(\sigma, \fks)\cong \rY_{m|n}^\ell(\sigma,\omega\fks),  \quad \forall\, \omega\in \fkS_\mu.
\eeq
Then a similar inverse limit argument as in Theorem \ref{thm:grownup} deduces \eqref{row-permute}. 
As a consequence, when necessary, one may first permute rows of $\pi^\prime$ and $\pi^\dprime$, making them $\fks$-compatible, and then construct the outer tensor of their modules. 
\end{rem}

\subsection{Harish-Chandra homomorphisms}\label{1col-case}
In this subsection, we consider the case when $\pi$ is a single column; that is, $\pi$ consists of $M$ even boxes and $N$ odd boxes that are stacked into a single column with respect to some order. Note that the reading of $\pi$ from top to bottom gives a parity sequence $(\fks_i)_{1\lle i\lle M+N}$ in which 1 appears exactly $M$ times. 
Since the nilpotent element corresponding to $\pi$ is zero, as a special case of Theorem \ref{thm:pe21}, the associated $W$-superalgebra $W(\pi)$ is identified as $\mathrm U(\gl_{M|N})$ by the evaluation map \eqref{evahom}.
To be explicit, $D_i^{(1)}\in W(\pi)$ is identified with $\fks_i e_{i,i}\in \mathrm U(\gl_{M|N})$ for all $1\lle i\lle M+N$.
Recall that $\fkc$ is the span of $\{D_i^{(1)} \, |\, 1\lle i\lle M+N\}$, where $\fkc^*$ is spanned by $\{ \ve_j \,|\, 1\lle j\lle M+N\}$ such that $\ve_j$ is dual to $\fks_jD_j^{(1)}$. 

It is well known that the center of $\mathrm U(\gl_{M|N})$ can be identified with $I(\fkc)$, which is a certain subalgebra of $S(\fkc)$ that we will recall immediately, by the Harish-Chandra map. 
For $1\lle i\lle M+N$, let $x_i:=e_{i,i}$ if $\fks_i=1$ and $y_i:=-e_{i,i}$ if $\fks_i=-1$ so that $S(\fkc)=\mathbb{C}[x_1,\ldots,x_M,y_1,\ldots, y_N]$.
The Weyl group of $\g=\gl_{M|N}$ is $\fkS_M\times \fkS_N$, where $\fkS_M$ permutes $\{x_i\}_{1\lle i\lle M}$ and $\fkS_N$ permutes $\{y_j\}_{1\lle j\lle N}$. Then we have 
\beq\label{def:Ioc}
I(\fkc):= \left\{f \in S(\fkc)^{\fkS_M \times \fkS_N}\:\bigg|\:
\begin{array}{l}
\frac{\partial f}{\partial x_i} + \frac{\partial f}{\partial y_j}
\equiv 0 \pmod{x_i-y_j}\\
\text{for any $1 \lle i \lle M$, $1 \lle j \lle N$}
\end{array}
\right\}.
\eeq
Moreover, a generating set of $I(\fkc)$ is given by the elementary supersymmetric polynomials
\beq\label{def:essp}
e_{r}(x_1,\ldots,x_M/y_1,\ldots,y_N) := \sum_{a+b=r} (-1)^be_{a}(x_1,\ldots,x_M) h_{b}(y_1,\ldots,y_N)
\eeq
for all $r\gge1$, where $e_{a}(x_1,\ldots,x_M)$ means the $a^{th}$ elementary symmetric polynomial and $h_{b}(y_1,\ldots,y_N)$ means the $b^{th}$ complete homogeneous symmetric polynomials.

By a normal order $\lhd$ on $\I=\{1,\ldots,M, \ovl{1},\ldots, \ovl{N}\}$ we mean a total order on $\I$ such that $1\lhd \cdots \lhd M$ and $\ovl{1} \lhd \cdots \lhd \ovl{N}$.
Given such an order, the associated Weyl vector $\rho^\lhd\in \fkc^*$ is uniquely defined by
\beq\label{def:rhos}
\rho^\lhd(\fks_jD_j^{(1)})=(\rho^\lhd, \ve_j)=-\ka_j
\eeq
for all $j\in \I$, where $\ka_j$ is defined in \eqref{eq:ka}.

The parity sequence $\fks$ uniquely determines a normal order $\lhd$ on $\I$ as follows.
We first insert $\lhd$ between each $\fks_i$ and $\fks_{i+1}$, then replace those $\{ \fks_i \, | \, \fks_i=1\}$ by $1,\ldots, M$ and those $\{ \fks_i \,|\, \fks_i=-1\}$ by $\ovl{1},\ldots, \ovl{N}$ from left to right. 
Conversely, if a normal ordering of $\I$ is given, it is easy to recover the corresponding parity sequence where any number without a bar corresponds to 1 and any number with a bar corresponds to $-1$.
For example, $\fks=(1,1,-1,1,-1,1)$ corresponds to the normal order $1\lhd 2\lhd \ovl{1}\lhd 3 \lhd \ovl{2} \lhd 4$. 

Let $\fkb^\lhd$ be the Borel subalgebra spanned by $\{e_{i,j} \,|\, i\unlhd  j\}$ with $\fkn_+^\lhd$ = span $\{e_{i,j} \,|\, i\lhd j\}$ and  $\fkn_-^\lhd$ = span $\{e_{j,i} \,|\, i\lhd j\}$.
It follows from the PBW theorem that the direct sum decomposition 
\beq\label{HC-decom}
\mathrm U(\g)=S(\fkc)\oplus (\fkn_-^\lhd \mathrm U(\g)+ \mathrm U(\g)\fkn_+^\lhd)
\eeq
induces an algebra homomorphism $\phi^\lhd:\mathrm U(\g)\rightarrow S(\fkc)$. The Harish-Chandra homomorphism is defined by the restriction of $\phi^\lhd$ on $Z(\g)$ followed by an automorphism
\beq\label{def:HC}
\HC:=S_{-\rho^\lhd}\circ \phi^\lhd: Z(\g)\rightarrow S(\fkc),
\eeq
where $S_{-\rho^\lhd}:S(\fkc)\rightarrow S(\fkc)$ is induced by 
\beq\label{rhose}
S_{-\rho^\lhd}(e_{i,i})= e_{i,i}+\fks_i\kappa_i, \qquad \forall\, 1\lle i\lle M+N.
\eeq
Similar to the notation $z_{M|N,\fks}(u)$ from \eqref{eq:zdef}, we shall write $\HC^\fks$ when we want to emphasize which parity sequence we are using in the definition of the Harish-Chandra homomorphism. 
We recall some well-known results as the following theorem. 

\begin{thm}[{\cite[Thms~2.2--2.4]{BG19}\cite[\S2.2.3]{CW12}}]\label{thmHC}
Let $\fks$ be an arbitrary parity sequence of type $(M|N)$. We have the following.
\begin{enumerate}
\item The homomorphism $\HC:Z(\g)\rightarrow I(\fkc)$ is an algebra isomorphism. 
\item The map $\HC:Z(\g)\rightarrow S(\fkc)$ is independent of the choice of the parity sequence $\fks$ (equivalently, independent of the choice of the normal order $\lhd$).
\item We have 
\[\HC^\fks(z_{M|N,\fks}(u))=\prod_{i\in \I} (u+\omega_i)^{\fks_i}=\prod_{k=1}^{M}(u+x_k) \bigg/ \prod_{t=1}^N(u+y_t),\]
where the product in the middle expression is written from left to right increasingly with respect to the normal ordering on $\I$ determined by $\fks$, and $\omega_j=e_{j,j}=x_{j'}$ if $\fks_j=1$ and $\omega_j=-e_{j,j}=y_{j'}$  if $\fks_j=-1$ with $$j':=\#\left\{1\lle i < j \: \big| \: \fks_i = \fks_j \right\} +1.$$
\item In particular, the coefficients $\{z_{M|N,\fks}^{(r)}\,|\,r\gge 1\}$ of $z_{M|N,\fks}(u)$ generate the center $Z(\mathrm U(\gl_{M|N}))$. 
Moreover, $\HC^\fks(z_{M|N,\fks}^{(r)})=e_r(x_1,\ldots,x_M/y_1,\ldots,y_N)$.
\end{enumerate}
\end{thm}
\begin{proof}
Part (1) is due to Kac and Sergeev, see \cite[\S13.2]{Mus12} for details. Part (2) is established in \cite[Thm.~2.2]{BG19} while Part (3) for the standard parity sequence is obtained in \cite[Thm.~2.3]{BG19}. Part (4) follows from Parts (1)-(3).

Part (3) for a general parity sequence $\fks$ follows from Part (2) and the statement for the standard parity case. Here we provide a more direct approach. By the explicit decomposition of the quantum Berezinian $\mc C_{M|N}(u)$ with respect to an arbitrary parity sequence given in \eqref{eq:Ber}, together with the fact that the evaluation map  $\ev:\mathrm Y_{M|N} \rightarrow \mathrm U(\gl_{M|N})$ from \eqref{evahom} is a homomorphism, one deduces that
\[
z_{M|N,\fks}(u)= \prod_{i=1}^{M+N} (u-\kappa_i)^{\fks_i}\ev( \mathcal C_{M|N}(u)).
\] 
By \eqref{def:HC}, the Harish-Chandra map $\HC^\fks$ with respect to $\fks$ carries the correct translations to each $e_{j,j}$, deducing the desired expression.
\end{proof}

Since the normal order $\lhd$ is determined by $\fks$, the theorem above implies that no matter how the boxes of the column $\pi$ are stacked, we always get the same isomorphism between $Z(\g)$ and $I(\fkc)$, i.e., their images under the Harish-Chandra maps coincide.

\begin{eg}
Consider $\gl_{2|2}$. Let $\fks=(1,1,-1,-1)$ be the standard one, where $(\kappa_1,\kappa_2,\kappa_3,\kappa_4)=(0,1,1,0)$. Thus $\I=\{1\lhd 2\lhd \ovl{1} \lhd \ovl{2}\}$, and 
\begin{multline*}
z_{2|2,\fks}(u)=\zeta_{1}(u)\zeta_2(u-1)\big(\zeta_{\ovl{1}}(u-1)\zeta_{\ovl{2}}(u)\big)^{-1}=\frac{(u+e_{1,1})(u-1+e_{2,2})}{(u-1-e_{3,3})(u-e_{4,4})}\\
=\frac{(u+(e_{1,1}-0))(u+(e_{2,2}-1))}{(u-(e_{3,3}+1))(u-(e_{4,4}+0))}\stackrel{\HC^\fks}{\Rightarrow}
\frac{(u+e_{1,1})(u+e_{2,2})}{(u-e_{3,3})(u-e_{4,4})}
=\frac{(u+x_1)(u+x_2)}{(u+y_1)(u+y_2)}.
\end{multline*}
On the other hand, choose $\tl\fks =(-1,1,-1,1)$, where $\kappa_i=-1$ for all $1\lle i\lle 4$.
Thus $\I=\{\ovl{1}\lhd 1\lhd \ovl{2} \lhd 2\}$, and 
\begin{multline*}
z_{2|2,\tl\fks}(u)=\big(\zeta_{\ovl{1}}(u+1)\big)^{-1}\zeta_1(u+1)\big(\zeta_{\ovl{2}}(u+1)\big)^{-1}\zeta_2(u+1)
=\frac{(u+1+e_{1,1})(u+1+e_{2,2})}{(u+1-e_{3,3})(u+1-e_{4,4})}\\
=\frac{(u+(e_{1,1}+1))(u+(e_{2,2}+1))}{(u-(e_{3,3}-1))(u-(e_{4,4}-1))}\stackrel{\HC^{\tl\fks}}{\Rightarrow}
\frac{(u+e_{1,1})(u+e_{2,2})}{(u-e_{3,3})(u-e_{4,4})}
=\frac{(u+x_1)(u+x_2)}{(u+y_1)(u+y_2)}.
\end{multline*}
\end{eg}

\section{Representations of shifted super Yangians}\label{sec:rep-shift}
In this section, we define category $\mathcal O$, highest weight modules, and $q$-characters for shifted super Yangians. Then we classify the finite dimensional irreducible modules for shifted super Yangians associated to the standard parity sequence.

\subsection{Category $\mc O$}
Recall the Lie subalgebra $\fkc$ of $\rY_{m|n}(\sigma)$ spanned by $D_i^{(1)}$ for $1\lle i\lle m+n$ and the root decomposition. Given a $\Y$-module $\mc M$, we consider it as a $\fkc$-module by restriction. For a weight $\alpha\in\fkc^*$, the \textit{weight space} of weight $\alpha$ is the subspace of $\mc M$ defined by
\[
\mc M_\alpha:=\big\{v\in \mc M\mid (D_i^{(1)}-\alpha(D_i^{(1)}))v=0, ~\forall~ 1\lle i\lle m+n\big\}.
\]

If $\mc M_\alpha$ is non-zero, we say that $\mc M_\alpha$ is a \textit{weight space} of weight $\alpha$ and $\alpha$ is a \textit{weight} of $\mc M$. A non-zero vector $v\in \mc M_\alpha$ is called a \textit{weight vector} of weight $\mathrm{wt}(v):=\alpha$.

Define $\mc O_\sigma$ to be the full subcategory of the category of $\Y$-modules, which consists of all $\Y$-modules $\mc M$ such that 
\begin{enumerate}
    \item $\mc M$ is a direct sum of finite dimensional weight subspaces;
    \item the set of all $\alpha\in\fkc^*$ such that $\mc M_{\alpha}$ is non-zero is contained in a finite union of sets of the form $\mathfrak D(\beta):=\{\alpha\in \fkc^*\mid \alpha\lle \beta\}$ for $\beta\in\fkc^*$.
\end{enumerate}

Let $\mathscr P_{m|n}:=(1+u^{-1}\bC[\![u^{-1}]\!])^{m+n}$ denote the subset of $(m+n)$-tuple of power series in $u^{-1}$ with constant term 1. We call an element in $\mathscr P_{m|n}$ an $\bm\ell$-weight. We write an $\bm\ell$-weight in the form $\bla=(\la_i(u))_{1\lle i\lle m+n}$. Clearly, $\mathscr P_{m|n}$ is an abelian group with respect to the point-wise multiplication of the tuples. For $\bla \in \mathscr P_{m|n}$, we write $\la_i^{(r)}$ for the coefficient of $u^{-r}$ in $\la_i(u)$. Then the associated weight of $\bla\in \mathscr P_{m|n}$ is defined by
\beq\label{eq:lwt-tw}
\varpi(\bla):=\fks_1 \la_1^{(1)}\ve_1+\fks_2\la_2^{(1)}\ve_2+\dots+\fks_{m+n}\la_{m+n}^{(1)}\ve_{m+n}\in\fkc^*.
\eeq

For a $\rY_{m|n}(\sigma)$-module $\mc M\in\mc O_\sigma$ and $\bla\in\mathscr P_{m|n}$, the \textit{$\bm\ell$-weight space} of $\bm\ell$-weight $\bla$ is a subspace of $\mc M$ defined by
\[
\mc M_{\bla}:=\big\{v\in \mc M\mid \forall~ 1\lle i\lle m+n \text{ and }r>0, \exists~ p>0 \text{ such that } (D_i^{(r)}-A_i^{(r)})^pv=0\big\}.
\]
If $\mc M_{\bla}$ is non-zero, we say that $\mc M_{\bla}$  is an $\bm\ell$-weight space of $\bm\ell$-weight $\bla$. A non-zero vector $v\in \mc M_{\bla}$ is called an $\bm\ell$-weight vector of $\bm\ell$-weight $\mathrm{wt}_\ell(v):=\bla$.

Note that the weight space of $\mc M$ are finite dimensional and the elements $D_i^{(r)}$ mutually commute, one has that for each $\alpha\in\fkc^*$,
\[
\mc M_\alpha=\bigoplus_{\substack{\bla\in\mathscr P_{m|n}\\ \varpi(\bla)=\alpha} }\mc M_{\bla}.
\]
Thus $\mc M$ is also a direct sum of its $\bm\ell$-weight spaces, $\mc M=\bigoplus_{\bla\in\mathscr P_{m|n}}\mc M_{\bla}$. 
\subsection{Highest $\bm\ell$-weight modules}
For $\bla\in \mathscr P_{m|n}$, a vector $v$ in a $\rY_{m|n}(\sigma)$-module $\mc M$ is called a \textit{highest $\bm\ell$-weight vector of $\bm\ell$-weight $\bla$} if
\begin{enumerate}
    \item $E_i^{(r)}v=0$ for all $1\lle i<m+n$ and $r>s_{i,i+1}$;
    \item $D_i^{(r)}v=\la_i^{(r)}v$ for all $1\lle i\lle m+n$ and $r> 0$.
\end{enumerate}
We call $\mc M$ a \textit{highest $\bm\ell$-weight module of $\bm\ell$-weight $\bla$} if it is generated by such a highest $\bm\ell$-weight vector.

Alternatively, it is not hard to see from the Gauss decomposition \eqref{eq:GD-shift} that a vector $v$ in a $\rY_{m|n}(\sigma)$-module $\mc M$ is a highest $\bm\ell$-weight module of $\bm\ell$-weight $\bla$ if and only if $T_{i,j}^{(r)}v=0$ for all $1\lle i<j\lle m+n$ and $r>s_{i,j}$, and $T_{i,i}^{(r)}v=\la_i^{(r)}v$ for all $1\lle i\lle m+n$ and $r>0$.

Write $\sigma=\sigma'+\sigma''$ where $\sigma'$ (resp. $\sigma''$) is a lower (resp. upper) shift matrix. 
\begin{lem}\label{lem:tensor}
Let $\mc M'$ and $\mc M''$ be modules over $\rY_{m|n}(\sigma')$ and $\rY_{m|n}(\sigma'')$, respectively. If $v\in \mc M'$ (resp. $w\in \mc M''$) is a highest $\bm\ell$-weight vector of $\bm\ell$-weight $\bla$ (resp. $\bm\xi$), then $v\otimes w$ is a highest $\bm\ell$-weight vector of $\bm\ell$-weight $\bla\bm\xi$.
\end{lem}
\begin{proof}
It follows from Proposition \ref{prop:copro} and Lemma \ref{lem:copro-res}.
\end{proof}

\begin{rem}
By using the isomorphism $\iota$ from \eqref{iotadef}, one can remove the restriction that $\sigma'$ and $\sigma''$ are triangular.
\end{rem}

For $\bla\in\mathscr P_{m|n}$, define the corresponding Verma module $\mc M(\sigma,\bla)$ as the quotient of $\rY_{m|n}(\sigma)$ by the left ideal generated by the elements
\[
\big\{D_i^{(r)}-\la_i^{(r)}\mid 1\lle i\lle m+n,r>0\big\},\qquad \big\{E_i^{(r)} \mid 1\lle i< m+n, r>s_{i,i+1}\big\}.
\]
The Verma module $\mc M(\sigma,\bla)$ is a highest $\bm\ell$-weight module of $\bm\ell$-weight $\bla$, generated by the highest weight vector $w_{\bla}$ (which is the quotient image of the identity). Moreover, it is universal among such modules, i.e. all other highest $\bm\ell$-weight modules of $\bm\ell$-weight $\bla$ are quotients of $\mc M(\sigma,\bla)$.

\begin{prop}\label{prop:PBW-Verma-Y}
For any $\bla\in \mathscr P_{m|n}$, the following sets of vectors form bases for the Verma module $\mc M(\sigma,\bla)$:
\begin{enumerate}
    \item $\{xw_{\bla}\mid x\in \mathcal X\}$, where $\mc X$ denotes the collection of all supermonomials in the elements $\{F_{j,i}^{(r)}\mid 1\lle i<j\lle m+n,r>s_{j,i}\}$ taken in some fixed order;
    \item $\{yw_{\bla}\mid y\in \mathcal Y\}$, where $\mc Y$ denotes the collection of all supermonomials in the elements $\{T_{j,i}^{(r)}\mid 1\lle i<j\lle m+n,r>s_{j,i}\}$ taken in some fixed order.
\end{enumerate}
\end{prop}
\begin{proof}
The first statement follows immediately from Proposition \ref{prop:PBW}. The second statement follows from similar filtration argument as in \cite[Thm.~5.5]{BK08} with the help of Proposition \ref{prop:PBW} and Lemma \ref{lem:PBW-T}.
\end{proof}

Therefore, any other weight vectors other than $w_{\bla}$ have weights smaller than $\mathrm{wt}(w_{\bla})=\varpi(\bla)$ and the $\bm\ell$-weight space of $\bm\ell$-weight $\bla$ in $\mc M(\sigma,\bla)$ is one-dimensional. It follows from the standard argument that $\mc M(\sigma,\bla)$ has a unique maximal submodule, the
quotient by which is irreducible and denoted by $\mc L(\sigma,\bla)$. Then $\mc L(\sigma,\bla)$ is the unique (up to isomorphism) irreducible highest $\bm\ell$-weight module of $\bm\ell$-weight $\bla$ over $\Y$. We shall always use $v_{\bla}$ to denote a highest $\bm\ell$-weight vector of $\bm\ell$-weight $\bla$ in $\mc L(\sigma,\bla)$. 

If $\sigma$ is the zero matrix, then we simply write $\mc M(\bla):=\mc M(\sigma,\bla)$ and $\mc L(\bla):=\mc L(\sigma,\bla)$.

Clearly, $\mc M(\sigma,\bla)$ is not in the category $\mc O_\sigma$. However, the irreducible $\mc L(\sigma,\bla)$ can be in $\mc O_\sigma$ with suitable $\bla$. Denote by $\mathscr R_{m|n}$ the subset of $\mathscr P_{m|n}$ consisting of all $\bla$ such that $\la_i(u)/\la_{i+1}(u)$ is a rational function for all $1\lle i<m+n$.

By the standard argument in \cite[Thm.~3.6]{MY14}, one has the following.

\begin{prop}\label{prop:O-rational}
For $\bla\in\mathscr P_{m|n}$, the $\Y$-module $\mc L(\sigma,\bla)$ has finite dimensional weight spaces (or equivalently $\mc L(\sigma,\bla)\in \mc O_\sigma$) if and only if $\bla\in\mathscr R_{m|n}$. Moreover, any simple object in $\mc O_\sigma$ is isomorphic to $\mc L(\sigma,\bla)$ for some $\bla\in \mathscr R_{m|n}$. Furthermore, all $\bm\ell$-weights of a module in $\mathcal O_\sigma$ are in $\mathscr R_{m|n}$.
\end{prop}

The following is a generalization of \cite[Prop.~3.8]{KTWWY19}.
\begin{lem}\label{lem:1-dim}
Let $\bm\zeta$ be an $\bm\ell$-weight such that
\[
\frac{\zeta_i(u)}{\zeta_{i+1}(u)}=\frac{u^{d_i}}{Q_i(u)}
\]
where $d_i:=s_{i,i+1}+s_{i+1,i}$ and $Q_i(u)$ is a monic polynomial in $u$ of degree $d_i$ for all $1\lle i<m+n$. Then the $\Y$-module $\mc L(\sigma,\bm\zeta)$ is 1-dimensional.    
\end{lem}
\begin{proof}
Consider the highest $\bm\ell$-weight vector $w_{\bm\zeta}$ in the Verma module $\mc M(\sigma,\bm\zeta)$. Then it follows from \eqref{dr5} and \eqref{eq:def-H} that
\[
E_{i}^{(r)}F_i^{(s)}w_{\bm\zeta}=[E_{i}^{(r)},F_i^{(s)}]w_{\bm\zeta}=
-\fks_{i+1}H_i^{(r+s-1)}w_{\bm\zeta}=0
\]
for all $1\lle i<m+n$, $r>s_{i,i+1}$, and $s>s_{i+1,i}$ due to our choice of $\bm\zeta$. Indeed, by \eqref{eq:def-H} we have  that $H_i(u)w_{\bm\zeta}=u^{-d_i}Q_i(u)w_{\bm\zeta}$ and hence $H_i^{(r+s-1)}w_{\bm\zeta}=0$ if $r>s_{i,i+1}$ and $s>s_{i+1,i}$. Thus the unique maximal proper submodule in $\mc M(\sigma,\bm\zeta)$ is generated by  all $F_i^{(s)}w_{\bm\zeta}$ for $1\lle i<m+n$ and $s>s_{i+1,i}$, completing the proof.
\end{proof}

\subsection{$q$-character}
Denote by $\widehat \bZ[\mathscr P_{m|n}]$ the \textit{completed group algebra} which consists of formal sums $$S=\sum_{\bla\in \mathscr P_{m|n}}a_{\bla}[\bla]$$ for integers $a_{\bla}$ satisfying
\begin{enumerate}
    \item the set $\{\varpi(\bla)\mid \bla\in \mathrm{supp}\,S\}$ is contained in a finite union of sets of the form $\mathfrak D(\beta)$ for $\beta\in\mathfrak c^*$;
    \item for each $\alpha\in \mathfrak c^*$ the set $\{\bla\in \mathrm{supp}\,S\mid \varpi(\bla)=\alpha\}$ is finite.
\end{enumerate}
Here $\mathrm{supp}\,S:=\{\bla\mid a_{\bla}\ne 0\}$ and $\widehat \bZ[\mathscr P_{m|n}]$ is multiplicative by extending $[\bla][\bm \xi]=[\bla\bm\xi]$.

Following \cite{Kn95} and \cite{FR99} (cf. also \cite{LM21} for the super Yangian case), define the $q$-character (also called \textit{Gelfand-Tsetlin character}) of a $\rY_{m|n}(\sigma)$-module $\mc M$ in the category $\mc O_\sigma$ by
\beq\label{eq:ch-def}
\mathrm{ch}\,\mc M:=\sum_{\bla\in \mathscr P_{m|n}}(\dim \mc M_{\bla})[\bla].
\eeq
Since $\mc M$ is in category $\mc O_\sigma$, we have $\mathrm{ch}\,\mc M\in \widehat \bZ[\mathscr P_{m|n}]$.

Recall the coproduct $\Delta:\Y\to\rY_{m|n}(\sigma')\otimes \rY_{m|n}(\sigma'')$ from \eqref{eq:coproduct}, where $\sigma'$ (resp. $\sigma''$) is the lower (resp. upper) triangular shift matrix such that $\sigma=\sigma'+\sigma''$. Let $\mc M'$ be a $\rY_{m|n}(\sigma')$-module and $\mc M''$ a $\rY_{m|n}(\sigma'')$-module. Then the tensor product of $\mc M'$ and $\mc M''$ can be regarded as a $\Y$-module via $\Delta$, which we denote it by the ``external" tensor product $\mc M'\boxtimes \mc M''$. Since $\Delta(D_i^{(1)})=D_i^{(1)}\otimes 1+1\otimes D_i^{(1)}$, we have
\[
\mc M'_\alpha\otimes \mc M''_{\beta}\subset (\mc M'\boxtimes \mc M'')_{\alpha+\beta}
\]
for all $\alpha,\beta\in\mathfrak c^*$. 

The following lemma is standard, following the same argument of \cite[Prop.~1]{Kn95} and \cite[Thm.~2]{FR99} with the help of Proposition \ref{prop:copro}.

\begin{lem}\label{ch-multiplicative}
Suppose $\mc M'\in \mc O_{\sigma'}$ and $\mc M''\in \mc O_{\sigma''}$. Then $\mc M'\boxtimes \mc M''\in \mc O_\sigma$ and 
\[
\mathrm{ch}(\mc M'\boxtimes \mc M'')=(\mathrm{ch}\, \mc M')(\mathrm{ch}\,\mc M'').
\]
\end{lem}

\subsection{Finite dimensional irreducible modules}
We first establish some preliminary results for the rank 1 case using the results of \cite{BBG13} for principal finite $W$-superalgebras.

\begin{prop}\label{prop:rk1}
Let $\bla\in \mathscr P_{1|1}$, the irreducible $\rY_{1|1}(\sigma)$-module $\mc L(\sigma,\bla)$ is finite dimensional if and only if $\la_1(u)/\la_2(u)$ is a rational function.
\end{prop}
\begin{proof}
($\Rightarrow$). This follows immediately from Proposition \ref{prop:O-rational}.

($\Leftarrow$). Suppose $\bla\in \mathscr P_{1|1}$ is such that $\la_1(u)/\la_2(u)$ is a rational function. If $\sigma$ is the zero matrix, then the statement follows from \cite[Thm.~4]{Zh95}. Now consider the general case. By the canonical embedding $\rY_{1|1}(\sigma)\hookrightarrow \rY_{1|1}$, the finite dimensional $\rY_{1|1}$-module $\mc L(\bla)$ is also a $\rY_{1|1}(\sigma)$-module by restriction. Moreover, $v_{\bla} \in \mc L(\bla)$ remains to be a highest $\bm\ell$-weight vector of $\bm\ell$-weight $\bla$ over $\rY_{1|1}(\sigma)$. Thus $\mc L(\sigma,\bla)$ is a sub-quotient of $\mc L(\bla)$ and hence finite dimensional as well.
\end{proof}

In the rest of this subsection, we write that $d:=s_{1,2}+s_{2,1}$. By Proposition \ref{prop:rk1}, we can write
\beq\label{eq:PQdef1}
\frac{\la_1(u)}{\la_2(u)}=\frac{f(u)}{g(u)}=\frac{u^{d}P(u)}{Q(u)},
\eeq
where $f(u)$, $g(u)$, $P(u)$ and $Q(u)$ are monic polynomials in $u$ such that
\beq\label{eq:gcd=1}
(f(u),g(u))=(P(u),Q(u))=1.
\eeq
This is always possible by taking $P(u)=f(u)/(f(u),u^d)$ and $Q(u)=g(u)u^d/(f(u),u^d)$.

Let $r=\deg P(u)$, then $\deg Q(u)=d+r$. We write further that
\beq\label{eq:PQdef2}
P(u)=\prod_{i=1}^{r}(u+a_{i}),\qquad Q(u)=\prod_{j=1}^{d+r}(u+b_{j}),
\eeq
where $a_i,b_j$ are complex numbers. Let $v_{\bla}$ be a highest $\bm\ell$-weight vector in $\mc L(\sigma,\bla)$.

\begin{prop}\label{prop:res-rk1}
Under the above assumption, we have the following.
\begin{enumerate}
    \item The dimension of $\mc L(\sigma,\bla)$ is $2^r$ with a basis given by 
    \beq\label{eq:basis-rk1}
    \{F_1^{(i_1)}\cdots F_1^{(i_k)}v_{\bla}\mid 0\lle k\lle r,s_{2,1}<i_1<\cdots<i_k\lle s_{2,1}+r\}.
    \eeq
    \item If $d>0$, then $\mc L(\sigma,\bla)$ is isomorphic to the restriction of the $\rY_{1|1}$-module $\mc L(\bla)$ if and only if $f(0)\ne 0$.
\end{enumerate}
\end{prop}
\begin{proof}
(1). Recall $\mu_{\wp}$ from \eqref{eq:mu-f} which fixes $F_1^{(r)}$ for all $r>0$. Twisting the module $\mc L(\sigma,\bla)$ by the automorphism $\mu_\wp$ with suitable $\wp(u)$, we can assume
\beq\label{eq:form}
\la_1(u)=u^{-r}P(u),\qquad \la_2(u)=u^{-d-r}Q(u).
\eeq

Let $W(\pi)$ be the corresponding finite $W$-superalgebra obtained via the quotient $\Omega$. It follows from \cite[\S7]{BBG13} that there exists an irreducible $W$-module denoted by $\overline L(\substack{a_1 \cdots a_{r} \\ b_1 \cdots b_{d+r}})$. Pulling back through the quotient homomorphism $\Omega$, $\overline L(\substack{a_1 \cdots a_{r} \\ b_1 \cdots b_{d+r}})$ remains to be irreducible as a $\mathrm{Y}_{1|1}(\sigma)$-module. Moreover, it is a highest $\bm\ell$-weight module with highest $\bm\ell$-weight $\bla$. Hence it is isomorphic to $\mc L(\sigma,\bla)$ and the statement follows from \cite[Thm.~7.3]{BBG13}.

(2). If $f(0)\ne 0$, then $(f(u),u^d)=1$ and hence $f(u)=P(u)$. This implies $r=\deg f(u)$. It is known from \cite[Thm.~4]{Zh95} that  the $\rY_{1|1}$-module $\mc L(\bla)$ has dimension $2^r$. Thus it follows from Part (1) that $\mc L(\sigma,\bla)$ and the $\rY_{1|1}$-module $\mc L(\bla)$ have the same dimension. The statement follows as $\mc L(\sigma,\bla)$ is a sub-quotient of the restriction module $\mc L(\bla)$.
\end{proof}

Now we classify finite dimensional irreducible modules over the shifted super Yangian $\Y$ associated with the standard parity sequence. We remark that such classification for non-shifted case is only available for standard parity in \cite[Thm.~4]{Zh96}. Though a criterion for $\mc L(\bla)$ to be finite dimensional with an arbitrary parity sequence $\fks$ can be
recursively deduced from \cite[Thm.~4]{Zh96} via the odd reflections of super Yangian, see \cite{Mo22,Lu22}.

\begin{thm}\label{thm:fd}
Let $\fks$ be the standard parity sequence. For $\bla\in\mathscr P_{m|n}$, the irreducible $\Y$-module $\mc L(\sigma,\bla)$ is finite dimensional if and only if there exist (necessarily unique) monic polynomials $P_i(u),Q_i(u)$ for $1\lle i<m+n$ such that $\deg P_m(u)=\deg Q_m(u)$, $(P_i(u),Q_i(u))=1$, $Q_j(u)$ is of degree $d_j:=s_{j,j+1}+s_{j+1,j}$, and
\beq\label{eq:fd-criterion}
\frac{\la_j(u)}{\la_{j+1}(u)}=\frac{P_j(u)}{P_{j}(u-\fks_j)}\times \frac{u^{d_j}}{Q_j(u)},\qquad 
\frac{\la_m(u)}{\la_{m+1}(u)}=\frac{P_m(u)}{Q_{m}(u)},
\eeq
for each $1\lle j<m+n$ and $j\ne m$.
\end{thm}
\begin{proof}
($\Rightarrow$). Let $\sigma_j$ denote the $2\times 2$ submatrix
\[
\begin{bmatrix}
s_{j,j} & s_{j,j+1}\\
s_{j+1,j} & s_{j+1,j+1}
\end{bmatrix}
\]
of the shift matrix $\sigma$. The map $\psi_j:\rY_{\fks_{j}\fks_{j+1}}(\sigma_j)\to \Y$ defined by
\[
\psi_j:\rY_{\fks_{j}\fks_{j+1}}(\sigma_j)\to \Y,\quad D_k^{(r)}\mapsto D_{j+k-1}^{(r)},~E_1^{(r)}\mapsto E_j^{(r)},~F_1^{(r)}\mapsto F_j^{(r)},\quad k=1,2,
\]
induces an embedding of the shifted (super) Yangian $\rY_{\fks_{j}\fks_{j+1}}(\sigma_j)$ into $\Y$. Then the highest $\bm\ell$-weight vector $v_{\bla}$ remains to be a highest $\bm\ell$-weight vector of highest $\bm\ell$-weight $\bla_j=(\la_j(u),\la_{j+1}(u))$ over the (super)algebra $\rY_{\fks_{j}\fks_{j+1}}(\sigma_j)$. Note that $\fks$ is the standard parity sequence. If $1\lle j<m+n$ and $j\ne m$, then $\fks_j=\fks_{j+1}$ and hence \eqref{eq:fd-criterion} with the corresponding restrictions follows from \cite[Coro.~7.5]{BK08}. If $j=m$, then $\fks_m\ne \fks_{m+1}$ and hence \eqref{eq:fd-criterion} follows from Proposition \ref{prop:rk1}.

($\Leftarrow$). Recalling the isomorphism $\iota$ from \eqref{iotadef}, the twist of a highest $\bm\ell$-weight $\rY_{m|n}(\vec{\sigma})$-module by $\iota$ is again a highest $\bm\ell$-weight $\Y$-module with the same highest $\bm\ell$-weight. Moreover, the twist of a module in $\mc O_{\vec{\sigma}}$ is in $\mc O_\sigma$. Thus it suffices to prove it for the case when $\sigma$ is an upper shift matrix. Then we have the coproduct $\Delta:\Y\to \rY_{m|n}\otimes\Y$. Rewrite $\bla$ as a product $\bla=\bm\gamma\bm\zeta$ such that
\[
\frac{\gamma_j(u)}{\gamma_{j+1}(u)}=\frac{P_j(u)}{P_j(u-\fks_j)}, \quad \frac{\gamma_m(u)}{\gamma_{m+1}(u)}=\frac{P_m(u)}{Q_m(u)},\quad \frac{\zeta_j(u)}{\zeta_{j+1}(u)}=\frac{u^{d_j}}{Q_j(u)},\quad \frac{\zeta_m(u)}{\zeta_{m+1}(u)}=1,
\]
for $1\lle j<m+n$ and $j\ne m$. Via the coproduct, we have the $\Y$-module $\mc L(\bm\gamma)\boxtimes \mc L(\sigma,\bm\zeta)$, where $\mc L(\bm\gamma)$ is the irreducible $\rY_{m|n}$-module with highest $\bm\ell$-weight $\bm\gamma$. Let $v_{\bm\gamma}$ and $v_{\bm \zeta}$ be the corresponding highest $\bm\ell$-weight vectors in $\mc L(\bm\gamma)$ and $\mc L(\sigma,\bm\zeta)$, respectively. It follows from Lemma \ref{lem:tensor} that $v_{\bm\gamma}\otimes v_{\bm \zeta}$ is a highest $\bm\ell$-weight vector of $\bm\ell$-weight $\bla$. Thus $\mc L(\sigma,\bla)$ is isomorphic to a sub-quotient of $\mc L(\bm\gamma)\boxtimes \mc L(\sigma,\bm\zeta)$. By \cite[Thm.~4]{Zh96}, $\mc L(\bm\gamma)$ is finite dimensional. Therefore, the statement follows as $\mc L(\sigma,\bm\zeta)$ is 1-dimensional by Lemma \ref{lem:1-dim}.  
\end{proof}

\section{Representations of finite $W$-superalgebras}\label{sec:rep-W}
In this section, we study highest $\bm\ell$-weight modules over the superalgebra $W(\pi)$. 

\subsection{Preliminaries} 
We start with notations that will be used to describe the representations of the finite $W$-superalgebra $W(\pi)$, where $\pi$ is a fixed pyramid $\pi=(q_1,\dots,q_\ell)$ of height $\lle m+n$. Denote as usual $(p_1,\dots,p_{m+n})$ the tuple of row lengths, and fix a choice of shift matrix $\sigma=(s_{i,j})_{1\lle i,j\lle m+n}$. 

A \textit{$\pi$-tableau} is the pyramid $\pi$ with arbitrary complex numbers inserted into each box of $\pi$. 
Two $\pi$-tableaux $\mathbf A$ and $\mathbf B$ are \textit{row equivalent}, denoted by $\mathbf A \sim_{\mathrm{row}} \mathbf B$, if one can be obtained from the other by permuting entries within rows. Denote by $\mathrm{Row}(\pi)$ the set of all row equivalence classes of $\pi$-tableaux. We call elements of $\mathrm{Row}(\pi)$ as \textit{row symmetrized $\pi$-tableaux}. 

Similarly, two $\pi$-tableaux $\mathbf A$ and $\mathbf B$ are \textit{column equivalent}, denoted by $\mathbf A \sim_{\mathrm{col}} \mathbf B$, if one can be obtained from the other by permuting entries in boxes of the same parity within columns. 

Let $\pi=\pi'\otimes \pi''$ for pyramids $\pi'$ and $\pi''$. If  $\mathbf A'$ and $\mathbf A''$ are $\pi'$-tableau and $\pi''$-tableau, respectively. We write $\mathbf A'\otimes\mathbf A''$ for the $\pi$-tableau obtained by concatenating $\mathbf A'$ and $\mathbf A''$.

Sometimes, we also consider finite $W$-superalgebras that are truncated shifted super Yangians associated to standard parity sequence. We call a pyramid $\pi$ is \textit{standard} if the first $m$ rows are even while the last $n$ rows are odd. 

When $\pi$ is standard, a $\pi$-tableau is \textit{column strict} if, for boxes in the same column from top to bottom, the entries in even boxes are strictly decreasing while the entries in the odd boxes are strictly increasing. Note that the partial order $\gge$ on $\mathbb C$ is defined by $a \gge b$ if
$(a-b)\in\bN$.
Denote $\mathrm{Col}(\pi)$ the set of all column strict $\pi$-tableaux.

Define weights, $\bm\ell$-weights, highest $\bm\ell$-weight vectors, $q$-characters and category $\mathcal O_\sigma$ exactly the same as in Section \ref{sec:rep-shift}, regarding $W(\pi)$-modules as $\Y$-modules via the quotient map $\Omega:\Y\twoheadrightarrow W(\pi)$. Then for a module $\mathcal M\in \mathcal O_\sigma$, we have $\mathrm{ch}\,\mc M\in \widehat \bZ[\mathscr P_{m|n}]$.

Let  $\pi=\pi'\otimes \pi''$ be a decomposition with $\pi'$ of level $\ell'$ and $\pi''$ of level $\ell''$, $\mc M'$ a $W(\pi')$-module and $\mc M''$ a $W(\pi'')$-module. One defines a $W(\pi)$-module structure on the tensor product $\mc M'\otimes \mc M''$ through the coproduct $\Delta_{\ell',\ell''}$ from \eqref{Wdel}. We denote this module by $\mc M'\boxtimes \mc M''$.

Let $\sigma=\sigma'+\sigma''$ where $\sigma'$ and $\sigma''$ are the corresponding shift matrices for $\pi'$ and $\pi''$, respectively. It follows from Corollary \ref{SWcop} and Lemma \ref{ch-multiplicative} that
\[
\mathrm{ch}(\mc M'\boxtimes \mc M'')=(\mathrm{ch}\, \mc M')(\mathrm{ch}\,\mc M'').
\]

\subsection{Highest weight modules}
Our main goal is to understand the universal highest $\bm\ell$-weight module of $\bm\ell$-weight $\bla\in\mathscr P_{m|n}$ for the superalgebra $W(\pi)$. Such a module is obviously the unique largest quotient of the $\rY_{m|n}(\sigma)$-module $\mc M(\sigma,\bla)$ on which the kernel of the homomorphism $\Omega$ acts as zero. Thus such a $W(\pi)$-module can be defined as $W(\pi)\otimes_{\rY_{m|n}(\sigma)}\mc M(\sigma,\bla)$. We write $v_+$ for the highest $\bm\ell$-weight vector $1\otimes w_{\bla}$ in $W(\pi)\otimes_{\rY_{m|n}(\sigma)}\mc M(\sigma,\bla)$.

\begin{thm}\label{thm:W-rep-basis}
For $\bla\in \mathscr P_{m|n}$, the $W(\pi)$-module $W(\pi)\otimes_{\rY_{m|n}(\sigma)}\mc M(\sigma,\bla)$ is non-zero if and only if $u^{p_i}\la_i(u)\in \bC[u]$ for all $1\lle i\lle m+n$. In that case, the following sets of vectors form bases for $W(\pi)\otimes_{\rY_{m|n}(\sigma)}\mc M(\sigma,\bla)$:
\begin{enumerate}
    \item $\{xv_+\mid x\in \mathcal X\}$, where $\mc X$ denotes the collection of all supermonomials in the elements $\{F_{j,i}^{(r)}\mid 1\lle i<j\lle m+n,s_{j,i}<r\lle s_{j,i}+p_i\}$ taken in some fixed order;
    \item $\{yv_+\mid y\in \mathcal Y\}$, where $\mc Y$ denotes the collection of all supermonomials in the elements $\{T_{j,i}^{(r)}\mid 1\lle i<j\lle m+n,s_{j,i}<r\lle s_{j,i}+p_i\}$ taken in some fixed order.
\end{enumerate}
\end{thm}
\begin{proof}
It suffices to prove Part (2) as Part (1) follows by  a similar filtration argument in the proof of \cite[Thm.~5.5]{BK08}. Recall that the Verma module $\mc M(\sigma,\bla)$ is the quotient of $\Y$ by the left ideal $\mc J$ generated by the elements
\[
\big\{D_i^{(r)}-\la_i^{(r)}\mid 1\lle i\lle m+n,r>0\big\}\cup\big\{E_{i}^{(r)}\mid 1\lle i \lle m+n,r>s_{i,i+1}\big\}.
\]
It is not hard to see from \eqref{gedef} and the Gauss decomposition \eqref{eq:GD-shift} the left ideal $\mc J$ is also equivalently generated by the elements
\[
P:=\big\{T_{i,i}^{(r)}-\la_i^{(r)}\mid 1\lle i\lle m+n,r>0\big\}\cup\big\{T_{i,j}^{(r)}\mid 1\lle i<j\lle m+n,r>s_{i,j}\big\}.
\]
Let $Q:=\big\{T_{j,i}^{(r)}\mid 1\lle i<j\lle m+n,r>s_{j,i}\big\}$. Pick any ordering on $P\cup Q$ such that the elements of $Q$ precede the elements of $P$. Clearly, it follows from Lemma \ref{lem:PBW-T} and Proposition \ref{prop:PBW-Verma-Y} that the ordered supermonomials in the elements $P\cup Q$ containing at least one element from $P$ form a basis of $\mc J$.

The $W(\pi)$-module $W(\pi)\otimes_{\rY_{m|n}(\sigma)}\mc M(\sigma,\bla)$ can be identified as the quotient of $W(\pi)$ by the image $\bar{\mc J}$ of $\mc J$ under the quotient map $\Omega:\Y\to W(\pi)$. 

If $\la_i^{(r)}\ne 0$ for some $1\lle i\lle m+n$ and $r>p_i$ (which is exactly $u^{p_i}\la_i(u)\notin\bC[u]$), then it follows from Theorem \ref{thm:vanish} that the image of $T_{i,i}^{(r)}-\la_i^{(r)}$ is a unit in $W(\pi)$ and hence $W(\pi)\otimes_{\rY_{m|n}(\sigma)}\mc M(\sigma,\bla)=0$.

Conversely, suppose $u^{p_i}\la_i(u)\in\bC[u]$. Consider two sets
\begin{align*}
\overline P:=&\,\big\{T_{i,i}^{(r)}-\la_i^{(r)}\mid 1\lle i\lle m+n,\,0<r\lle p_i\big\}\\
&\,\cup\big\{T_{i,j}^{(r)}\mid 1\lle i<j\lle m+n,\,s_{i,j}< r\lle s_{i,j}+p_i\big\} 
\end{align*}
and $\overline Q:=\big\{T_{j,i}^{(r)}\mid 1\lle i<j\lle m+n,\,s_{j,i}<r\lle s_{j,i}+p_i\big\}$. By Theorem \ref{thm:vanish} and Lemma \ref{prop:PBW-T-truncated}, the order supermonomials in elements of $\overline P\cup \overline Q$ containing at least one element from $\overline P$ form a basis of $\bar{\mc J}$. Therefore, the image of $\mc Y$ forms a basis for $W(\pi)/\bar{\mc J}$, completing the proof of Part (2).
\end{proof}

Recall $\ka_i$ defined in \eqref{eq:ka} for $1\lle i\lle m+n$. Now suppose that $v_+$ is a non-zero highest $\bm\ell$-weight vector in some $W(\pi)$-module $\mc M$. It follows from Theorem \ref{thm:W-rep-basis} that there exist numbers $(a_{i,j})_{1\lle i\lle m+n,1\lle j\lle p_i}$ such that
\beq\label{eq:Di-eigenvalues}
(u-\ka_i)^{p_i}D_i(u-\ka_i)v_+=\prod_{j=1}^{p_i}(u+a_{i,j})v_+, \qquad 1\lle i\lle m+n.
\eeq
In this way, the highest $\bm\ell$-weight vector $v_+$ defines a row symmetrized $\pi$-tableau $\bA$ so that the numbers $a_{i,1},\dots,a_{i,p_i}$ are on the $i$-th row of $\pi$. From now on we shall also say  the highest $\bm\ell$-weight vector $v_+$ is \textit{of type $\bA$} if the equations \eqref{eq:Di-eigenvalues} hold. 

Conversely, given $\bA\in \mathrm{Row}(\pi)$ with entries $a_{i,1},\dots,a_{i,p_i}$ on its $i$-th row. Define the \textit{Verma module} $\mc M(\bA)$ over $W(\pi)$ by
\beq\label{eq:Verma-W}
\mc M(\bA):=W(\pi)\otimes_{\Y} \mc M(\sigma,\bla),
\eeq
where $\bla=(\la_i(u))_{1\lle i\lle m+n}\in \mathscr P_{m|n}$ satisfies
\beq\label{eq:lam-tab}
(u-\ka_i)^{p_i}\la_i(u-\ka_i)=\prod_{j=1}^{p_i}(u+a_{i,j}),\qquad 1\lle i\lle m+n.
\eeq
Then $\mc M(\bA)$ is the universal highest $\bm\ell$-weight module of type $\bA$ (or $\bm\ell$-weight $\bla$). It follows from Theorem \ref{thm:W-rep-basis} that $v_+$ is non-zero and hence a highest $\bm\ell$-weight vector of $\bm\ell$-weight $\bla$. Moreover, $\mc M(\bA)$ is in the category $\mc O_{\sigma}$. Again, it has a unique maximal submodule denoted by $\mathrm{rad}\,\mc (\bA)$. The quotient \beq\label{eq:irred-W}
\mc L(\bA):= \mc M(\bA)/\mathrm{rad}\,\mc (\bA) \cong W(\pi)\otimes_{\Y} \mc L(\sigma,\bla),
\eeq
is the unique (up to isomorphism) irreducible highest $\bm\ell$-weight module of type $\bA$. The modules $L(\bA)$ for $\bA\in \mathrm{Row}(\pi)$ give a complete set of pairwise non-isomorphic irreducible modules over $W(\pi)$ in the category $\mc O_\sigma$, cf. Proposition \ref{prop:O-rational}.

\begin{eg}
The trivial $W(\pi)$-module, defined as the restriction of the trivial $\rU(\fkp)$-module, is isomorphic to the module $\mc L(\bA_0)$. Here, $\bA_0$ is the tableau with all entries in the $i$-{th} row equal to $-\ka_i$. Namely, the corresponding highest $\bm\ell$-weight $\bla=(\la_i(u))_{1\lle i\lle m+n}$ satisfies $\la_i(u)=1$ for all $1\lle i\lle m+n$.
\end{eg}

\begin{eg}\label{eg:1-col}
Suppose the pyramid $\pi$ consists of a single column of height $m'+n'\lle m+n$, where $m'$ and $n'$ denote respectively the numbers of even and odd boxes, then we have $W(\pi)=\rU(\gl_{m'|n'})$. Let $\bA$ be a $\pi$-tableau
with entries $a_1,\cdots,a_{m'+n'}\in \mathbb C$ from top to bottom. A highest $\bm\ell$-weight vector for $W(\pi)$ of type $\bA$ is a vector such that
\begin{enumerate}
    \item $e_{i,j}v^+=0$ for all $1\lle i<j\lle m'+n'$;
    \item $e_{i,i}v^+=\fks_{m+n-m'-n'+i}(a_i+\ka_{m+n-m'-n'+i})v^+$ for all $1\lle i\lle m'+n'$.
\end{enumerate}
Thus the module $\mc M(\bA)$ is exactly the Verma module $\mc M(\alpha)$ over $\rU(\gl_{m'|n'})$ with some suitable weight $\alpha$.
\end{eg}

We end this subsection with   the classification of finite dimensional irreducible representations of $W(\pi)$ when $\pi$ is standard.

\begin{thm}\label{thm:fd-W}
Let $\pi$ be standard. For $\mathbf A\in\mathrm{Row}(\pi)$, the irreducible $W(\pi)$-module $\mathcal L(\bA)$ is finite dimensional if and only $\bA$ has a representative belonging to $\mathrm{Col}(\pi)$.
\end{thm}
\begin{proof}
The proof is the same as that of \cite[Thm.~7.9]{BK08}. For the only if part, one only needs to repeat that argument for $1\lle i<m+n$ and $i\ne m$ using \cite[Coro.~7.8]{BK08}; see also the proof of Theorem \ref{thm:fd}. For the if part, it is also the same. Suppose that $\bA$ has a representative belonging to $\mathrm{Col}(\pi)$. Let
$\bA_1,\dots,\bA_\ell$ be the columns of this representative.
Since $\pi$ is standard and $\bA_i$ is column strict, the irreducible module $\mc L(\bA_i)$ is finite dimensional (see Example \ref{eg:1-col}). By
Lemma \ref{lem:tensor}, the tensor product $\mc L(\bA_1)\boxtimes \cdots\boxtimes \mc L(\bA_\ell)$ is a finite dimensional
$W(\pi)$-module containing a highest $\bm\ell$-weight vector of type $\bA$. Hence $\mc L(\bA)$ being its subquotient is also finite
dimensional.
\end{proof}
\begin{rem}
Let $\bla$ be the $\bm\ell$-weight related to the $\pi$-tableau $\bA$ as in \eqref{thm:fd}. By Theorems \ref{thm:fd} and \ref{thm:fd-W}, the $W(\pi)$-module $\mc L(\bA)$ is finite dimensional if and only if the $\Y$-module $\mc L(\sigma,\bla)$ is finite dimensional.
\end{rem}

\subsection{Character of Verma modules}
Introduce the following shorthand notation for some special elements of the completed group algebra $\widehat\bZ[\mathscr P_{m+n}]$:
\beq\label{eq:y}
y_{i,a}:=1+(a+\ka_i)u_i^{-1}
\eeq
for $1\lle i\lle m+n$.

A \textit{supertuple} $c=(c_{i,j,k})_{1\lle i<k\lle m+n,1\lle j\lle p_i}$ is a tuple that satisfies:
\begin{enumerate}
    \item all $c_{i,j,k}$ are natural numbers;
    \item if $\fks_i\ne \fks_{k}$, then $c_{i,j,k}\lle 1$ for all $1\lle j\lle p_i$.
\end{enumerate} 
One of our main results is the character formula of the Verma module.

\begin{thm}\label{thm:character}
For $\bA\in \mathrm{Row}(\pi)$ with entries $a_{i,1},\dots,a_{i,p_i}$ on the $i$-th row for each $1\lle i\lle m+n$, we have
\beq\label{eq:ch-Verma}
\begin{split}
&\mathrm{ch}\,\mc M(\bA)\\
=\,&\sum_{c}\prod_{i=1}^{m+n}\prod_{j=1}^{p_i}\left(y_{i,a_{i,j}-\fks_i(c_{i,j,i+1}+\dots+c_{i,j,m+n})}\prod_{k=i+1}^{m+n}\Big(\frac{y_{k,a_{i,j}-\fks_i(c_{i,j,k+1}+\dots+c_{i,j,m+n})}}{y_{k,a_{i,j}-\fks_i(c_{i,j,k}+\dots+c_{i,j,m+n})}}\Big)^{\fks_i\fks_k}\right),
\end{split}
\eeq
where the sum is over all supertuples $c=(c_{i,j,k})_{1\lle i<k\lle m+n,1\lle j\lle p_i}$.
\end{thm}

We shall prove Theorem \ref{thm:character} in the next two subsections. For the sake of the argument, we shall \textit{assume henceforth that the shift matrix $\sigma$ is upper triangular} (or equivalently, that \textit{$\pi$ is left-justified}). The general case follows immediately by twisting with the isomorphism $\iota$ defined in \eqref{iotadef}.

By interchanging the two products on the right-hand side of \eqref{eq:ch-Verma}, one obtains the following generalization of \cite[Coro. 6.3]{BK08}.

\begin{cor}\label{cor:char}
Let $\bA_1,\dots,\bA_\ell$ be the columns of any representative of the row-symmetrized $\pi$-tableau $\bA\in \mathrm{Row}(\pi)$, so that $\bA\sim_{\mathrm{row}}\bA_1\otimes \bA_2\otimes \cdots\otimes \bA_{\ell}$. Then
\[
\mathrm{ch}\,\mc M(\bA)=\prod_{k=1}^\ell \mathrm{ch}\,\mc M(\bA_k)=\mathrm{ch}\big(\mc M(\bA_1)\boxtimes\cdots \boxtimes \mc M(\bA_\ell)\big).
\]
\end{cor}

As an application, Corollary \ref{cor:char} allows us to establish an isomorphism between centers of the enveloping superalgebra $\rU(\gl_{M|N})$ and finite $W$-superalgebra $W(\pi)$ in Theorem \ref{isocent} below. Corollary \ref{cor:char} is also a key ingredient of \cite{CW26} for establishing a canonical basis character formula for the irreducible $W(\pi)$-modules in arbitrary parabolic BGG-type categories. This latter application, suggested by S.-J. Cheng and W. Wang, serves as the primary motivation for the present article.

\subsection{Branching rule for generic case}
We consider the \textit{generic} case, i.e. under the assumption that $a_{i,j_1}-a_{i,j_2}\notin \bZ$ for all $1\lle i< m+n$ and $1\lle j_1\ne j_2\lle p_i$. We shall assume this condition holds throughout this subsection and establish the following branching rule for the Verma modules.

Let $\bar\pi$ the pyramid obtained from $\pi$ by removing the bottom row. Then the row lengths of $\bar \pi$ are $p_1,\dots,p_{m+n-1}$. The submatrix $\bar\sigma=(s_{i,j})_{1\lle i,j<m+n}$ corresponds to the shift matrix for $\bar\pi$. Clearly, there exists a homomorphism
\begin{align*}
&W(\bar\pi)\to W(\pi),&\\
&D_i^{(r)}\mapsto D_i^{(r)}, &(1\lle i<m+n,r>0),\\
&E_i^{(r)}\mapsto E_i^{(r)}, &(1\lle i<m+n-1,r>s_{i,i+1}),\\
&F_i^{(r)}\mapsto F_i^{(r)}, &(1\lle i<m+n-1,r>s_{i+1,i}).
\end{align*}
It follows from Proposition \ref{prop:PBW-truncated} that this homomorphism is injective. Therefore, we can consider $W(\bar\pi)$ as a subalgebra of $W(\pi)$.

\begin{thm}\label{thm:branching}
Let $\bA\in\mathrm{Row}(\pi)$ with entries $a_{i,1},\dots,a_{i,p_i}$ on its $i$-th row for each $1\lle i\lle m+n$, where $a_{i,j_1}-a_{i,j_2}\notin \bZ$ for all $1\lle i\lle m+n$ and $1\lle j_1\ne j_2\lle p_i$. There is a filtration $0=M_0\subset M_1\subset \cdots$ of $\cM(\bA)$ as a $W(\bar\pi)$-module with $\bigcup_{i\gge 0}M_i=\cM(\bA)$ and subquotients isomorphic to the  Verma modules $\cM(\bB)$ for $\bB\in \mathrm{Row}(\bar\pi)$ such that $\bB$ has the entries $(a_{i,1}-\fks_ic_{i,1}),\dots,(a_{i,p_i}-\fks_ic_{i,p_i})$ on its $i$-th row for each $1\lle i<m+n$, one for each tuple $(c_{i,j})_{1\lle i<m+n,1\lle j\lle p_i}$ of natural numbers such that $c_{i,j}\lle 1$ if $\fks_i\ne\fks_{m+n}$.
\end{thm}

This branching rule allows us to prove Theorem \ref{thm:character} for generic case, see the beginning of \S \ref{sec:proof-char} below. Then we deduce Theorem \ref{thm:character} for the general case from the generic case by a continuation argument. We prove Theorem \ref{thm:branching} in the rest of this subsection. 

We prepare some lemmas that will be used in the proof.
\begin{lem}\label{lem:relations}
Assume $\pi$ is left-justified. The following relations hold in $W(\pi)$.
\begin{enumerate}
\item For all $i<j$, 
\begin{itemize}
    \item if $\fks_i=\fks_j$, we have $[F_{j,i}(u)D_i(u),F_{j,i}(v)D_i(v)]=0$;
    \item if $\fks_i\ne\fks_j$, we have
\[
(u-v-\fks_j)F_{j,i}(u)D_i(u)F_{j,i}(v)D_i(v)=(v-u-\fks_j)F_{j,i}(v)D_i(v)F_{j,i}(u)D_i(u).
\]
\end{itemize}
\item For all $i<j<k$, $(u-v)[F_{k,j}(u),F_{j,i}(v)]$ equals 
\beq\label{4.6-2-k}
\fks_j\sum_{r\gge 0}(-1)^r\Big(\sum_{\substack{i<i_1<\cdots<i_r\lle j\\ i_{r+1}=k}}F_{i_{r+1},i_r}(u)F_{i_r,i_{r-1}}(u)\cdots F_{i_2,i_1}(u)\big(F_{i_1,i}(v)-F_{i_1,i}(u)\big)\Big).
\eeq
\item For all $i<j$ and $k<i$ or $k>j$, $[D_k(u),F_{j,i}(v)]=0$.
\item For all $i<j$, $(u-v)[D_i(u),F_{j,i}(v)]=\fks_i(F_{j,i}(u)-F_{j,i}(v))D_i(u)$.
\item For all $i<j$, $(u-v)[D_j(u),F_{j,i}(v)]$ equals
\[
\fks_j\sum_{r\gge 0}(-1)^r\Big(\sum_{\substack{i<i_1<\cdots<i_r< j\\ i_{r+1}=j}}F_{i_{r+1},i_r}(u)F_{i_r,i_{r-1}}(u)\cdots F_{i_2,i_1}(u)\big(F_{i_1,i}(v)-F_{i_1,i}(u)\big)D_j(u)\Big).
\]
\item For all $i<j<k$, $(-1)^{|i||j|+|i||k|+|j||k|}(u-v)[D_j(u),F_{k,i}(v)]$ equals
\[
\sum_{r\gge 0}(-1)^r\Big(\sum_{i<i_1<\dots<i_r<i_{r+1}= j}F_{i_{r+1},i_r}(u)F_{i_r,i_{r-1}}(u)\cdots F_{i_2,i_1}(u)\big(F_{i_1,i}(v)-F_{i_1,i}(u)\big)F_{k,j}(u)D_j(u)\Big).
\]
\item For all $i<j<k$, $(-1)^{|i||j|+|i||k|+|j||k|}(u-v)[F_{k,j}(u)D_j(u),F_{k,i}(v)]$ equals
\[
\sum_{r\gge 0}(-1)^r\Big(\sum_{\substack{i<i_1<\cdots<i_r< j\\ i_{r+1}=k}}F_{i_{r+1},i_r}(u)F_{i_r,i_{r-1}}(u)\cdots F_{i_2,i_1}(u)\big(F_{i_1,i}(v)-F_{i_1,i}(u)\big)F_{k,j}(u)D_j(u)\Big).
\]
\end{enumerate}
\end{lem}

\begin{proof}
The proof follows directly from results of \cite{Pe16} together with some simple inductive argument. We prove \eqref{4.6-2-k} in detail and omit the proof for others. Since $\pi$ is left-justified, $\sigma$ is upper triangular ($s_{j,i}=0$ for all $j\gge i$). This implies that the element $F_{j,i}^{(r)}\in \rY_{m|n}(\sigma)$ can be identified with $\sfF_{j,i}^{(r)}\in\rY_{m|n}$. Consequently, 
$F_{j,i}^{(r)}=\fks_{j-1}[F_{j-1}^{(1)}, F_{j-1,i}^{(r)}]$ for all $j-i\gge 2$, 
or in terms of generating series 
\beq\label{pf4.6}
F_{j,i}(u) = \fks_{j-1}[F_{j-1}^{(1)}, F_{j-1,i}(u)].
\eeq

It suffices to establish \eqref{4.6-2-k} in the special case $k=j+1$:
\begin{multline}\label{4.6-2-1}
(u-v)[F_{j}(u),F_{j,i}(v)]= \\
\fks_j\sum_{r\gge 0}(-1)^r \Big(\sum_{\substack{i<i_1<\cdots<i_r\lle j\\ i_{r+1}=j+1}}F_{i_{r+1},i_r}(u)F_{i_r,i_{r-1}}(u)\cdots F_{i_2,i_1}(u)\big(F_{i_1,i}(v)-F_{i_1,i}(u)\big) \Big).
\end{multline}
Indeed, applying $[F_{j+1}^{(1)}, ? ]$ to both sides of \eqref{4.6-2-1}, using the facts that $[F_{j+1}^{(1)},F_{j,i}(v)]=0$, which follows from \eqref{dr10}, and $[F_{j+1}^{(1)},F_{j+1,h}(u)]=\fks_{j+1} F_{j+2,h}(u)$ for all $h\lle j$, which follows from \eqref{pf4.6}, one deduces \eqref{4.6-2-k} for $k=j+2$. Applying $[F_{t}^{(1)}, ? ]$ successively to the results for $t=j+2,\cdots,k-1$, one obtains \eqref{4.6-2-k} in general.

We prove \eqref{4.6-2-1} by induction on $j-i$. The initial case when $j-i=1$ is exactly the result of applying the shift map $\psi_{i-1}$ to \cite[(6.21)]{Pe16}, which also implies the following identity:
\beq\label{pf4.6-2}
[F_{j}(u), F_{j-1}^{(1)}]=\fks_{j}\big(F_{j+1,j-1}(u) - F_j(u) F_{j-1}(u)\big).
\eeq 
Assume now $j-i\gge 2$, then
\begin{align*}
&(u-v)[F_{j}(u), F_{j,i}(v)]= (u-v) \big[ F_{j}(u), [F_{j-1}^{(1)}, F_{j-1,i}(v)] \big]\fks_{j-1} \\
=~~&(u-v) \big[ [F_{j}(u), F_{j-1}^{(1)}], F_{j-1,i}(v) \big]\fks_{j-1}  \\
\stackrel{\eqref{pf4.6-2}}{=}&(u-v) \big[ F_{j+1,j-1}(u) - F_j(u) F_{j-1}(u), F_{j-1,i}(v) \big]\fks_{j-1}\fks_{j}\\
=~~&(u-v) \big[ F_{j+1,j-1}(u), F_{j-1,i}(v) \big]\fks_{j-1}\fks_{j} -(u-v)F_j(u) \big[ F_{j-1}(u), F_{j-1,i}(v) \big]\fks_{j-1}\fks_{j} \\
=~~& \fks_j \sum_{a\gge 0}(-1)^a \Big(\sum_{\substack{i<i_1<\cdots<i_a\lle j-1\\ i_{a+1}=j+1}}F_{i_{a+1},i_a}(u)F_{i_a,i_{a-1}}(u)\cdots F_{i_2,i_1}(u)\big(F_{i_1,i}(v)-F_{i_1,i}(u)\big) \Big) \\
&- \fks_j \sum_{b\gge 0}(-1)^b F_{j}(u)\Big(\sum_{\substack{i<i_1<\cdots<i_b\lle j-1\\ i_{b+1}=j}}F_{i_{b+1},i_b}(u)F_{i_b,i_{r-1}}(u)\cdots F_{i_2,i_1}(u)\big(F_{i_1,i}(v)-F_{i_1,i}(u)\big) \Big) 
\end{align*}
Note that the last equality is due to the induction hypothesis. The two summations correspond to the cases when $i_r\neq j$ and $i_r=j$, respectively, in the RHS of \eqref{4.6-2-1}, completing the proof.
\end{proof}

\begin{cor}\label{cor:string}
For $1\lle i<j\lle m+n$, if $\fks_i\ne \fks_j$ , then
\[
F_{ji}(u)D_i(u)F_{ji}(u-\fks_i)D_{i}(u-\fks_i)=0.
\]
\end{cor}
\begin{proof}
Follows immediately from Part (1) of Lemma \ref{lem:relations} .
\end{proof}

Let
\beq
L_i(u)=\sum_{r=1}^{p_i}L_i^{(r)}u^{p_i-r}:=u^{p_i}T_{m+n,i}(u)\in W(\pi)[u]
\eeq
for each $1\lle i<m+n$. 
We shall apply the following observation repeatedly from now on: given a vector $\mathfrak m$ of weight $\alpha$ in a $W(\pi)$-module $\mathfrak M$ with the property that $\alpha+\ve_j-\ve_i$ is not a weight of $\mathfrak M$ for any $1\lle j< i$, then we have by \eqref{eq:GD-shift} that
\[
L_i(u)\mathfrak m=u^{p_i}F_{m+n,i}(u)D_i(u)\mathfrak m.
\]

\begin{lem}\label{lem:simple}
Suppose we are given $1\lle i<m+n$ and a vector $\mathfrak m$ of weight $\alpha$ in a $W(\pi)$-module $\mathfrak M$ such that
\begin{enumerate}
    \item $\alpha-d(\ve_i-\ve_{m+n})+\ve_j-\ve_i$ is not a weight of $\mathfrak M$ for any $1\lle j<i$ and $d\in\bZ_{\gge 0}$;
    \item $u^{p_i}D_i(u)\mathfrak m\equiv (u+a_1)\cdots(u+a_{p_i})\mathfrak m \pmod{ \mathfrak M'[u]}$ for some $a_1,\dots,a_{p_i}\in \bC$ and some subspace $\mathfrak M'$ of $\mathfrak M$, where $a_{j_1}-a_{j_2}\notin \bZ$ for $j_1\ne j_2$.
\end{enumerate}
For $1\lle j\lle p_i$, define $\fkm_j:=L_{i}(-a_j)\fkm$. Then we have
\begin{align*}
&u^{p_i}D_i(u)\fkm_j\\
\equiv~& (u+a_1)\cdots(u+a_{j-1})(u+a_j-\fks_i)(u+a_{j+1})\cdots (u+a_{p_i})\fkm_j \pmod{\sum_{r=1}^{p_i}L_i^{(r)}\fkM'[u]}.
\end{align*}
Moreover, the subspace of $\fkM$ spanned by the vectors $\fkm_1,\dots,\fkm_{p_i}$ coincides with the subspace spanned by the vectors $L_i^{(1)}\fkm,\dots,L_{i}^{(p_i)}\fkm$.
\end{lem}
\begin{proof}
The proof is parallel to that of \cite[Lem. 6.16]{BK08} with the help of Part (4) from Lemma \ref{lem:relations}, which clearly shows how $\fks_i$ shows up. 
\end{proof}

Let $C_d$ denote the set all $p_i$-tuples $c=(c_1,\dots,c_{p_i})$ of natural numbers summing to $d$. Put a total order on $C_d$ so that $c'<c$ if $c'$ is lexicographically greater than $c$. If $\fks_i\ne \fks_{m+n}$, we suppose further that $c_j\lle 1$ for all $1\lle j\lle p_i$.

\begin{lem}\label{lem:general}
Under the same assumptions as in Lemma \ref{lem:simple}, define the vector $\fkm_c$ for $c\in C_d$ by 
\[
\fkm_c:=\prod_{j=1}^{p_i}\prod_{k=1}^{c_j}L_{i}\big(-a_j+\fks_i(k-1)\big)\fkm,
\]
where the products are taken in order of increasing indices from left to right. Then we have
\[
u^{p_i}D_i(u)\fkm_c\equiv (u+a_1-\fks_ic_1)\cdots(u+a_{p_i}-\fks_ic_{p_i})\fkm_c\quad \pmod{\fkM'_c[u]}
\]
where $\fkM'_c$ is the subspace of $\fkM$ spanned by all the vectors $\fkm_{c'}$ for $c'<c$ and $L_i^{(r_1)}\cdots L_i^{(r_d)}\fkM'$ for $1\lle r_1,\dots,r_d\lle p_i$, Moreover, the vectors $\{\fkm_c\mid c\in C_d\}$ span the same subspace of $\fkM$ as the vectors $L_i^{(r_1)}\cdots L_i^{(r_d)}\fkm$ for all $1\lle r_1,\dots,r_d\lle p_i$. In particular, if $\fks_i\ne \fks_{m+n}$, then the vectors $\{\fkm_c\mid c\in C_d\}$ span the same subspace of $\fkM$ as the vectors $L_i^{(r_1)}\cdots L_i^{(r_d)}\fkm$ for all $1\lle r_1<\dots<r_d\lle p_i$.
\end{lem}
\begin{proof}
We first note the following simple observation. If $\fks_i=\fks_{m+n}$, then by Part (1) of Lemma \ref{lem:relations} the definition of the vectors $\fkm_c$ does not depend on the order taken in the products. If $\fks_i\ne\fks_{m+n}$, then by Part (1) of Lemma \ref{lem:relations} the definition of $\fkm_c$ does depend on the order taken in the products. However, under our genericness assumption, it follows from Part (1) of Lemma \ref{lem:relations} that different orders in the products give rise to proportional vectors $\fkm_c$. The rest of the proof is similar (actually slightly simpler as $a_i$ are generic) to that of \cite[Lem.~6.17]{BK08}.\end{proof}

Now we are ready to prove Theorem \ref{thm:branching}. 
Denote by $C$ the set of all supertuples of natural numbers $c=(c_{i,j})_{1\lle i<m+n,1\lle j\lle p_i}$. Let 
$$|c|_i:=\sum_{j=1}^{p_i}c_{i,j},\qquad |c|:=\sum_{i=1}^{m+n-1}|c|_i.$$
Introduce a total order on $C$ as follows. We say that $c'<c$ if  one of the following hold:
\begin{enumerate}
    \item $|c'|<|c|$;
    \item $|c'|=|c|$ and there exists $1\lle i<m+n$ such that $|c'|_i>|c|_i$ and $|c'|_j=|c|_j$ for $i< j<m+n$;
    \item $|c'|_j=|c|_j$ for all $1\lle j<m+n$, and there exists $1\lle i<m+n$ such that the tuple $(c_{j,1}',\dots,c_{j,p_j}')$ is equal to the tuple $(c_{j,1},\dots,c_{j,p_j})$ for all $1\lle j<i$ while the tuple $(c_{i,1}',\dots,c_{i,p_i}')$ is lexicographically greater than the tuple $(c_{i,1},\dots,c_{i,p_i})$.
\end{enumerate}

Let $\fkM:=\mathcal M(\bA)$. For every supertuple $c\in C$, define a vector $\fkm_c\in\fkM$ by
\[
\fkm_c:=\prod_{i=1}^{m+n-1} \prod_{j=1}^{p_i}\prod_{k=1}^{c_{i,j}} L_i(-a_{i,j}+\fks_i(k-1))\, v_+
\]
where the products are taken in increasing order from left to right. It follows from Theorem \ref{thm:W-rep-basis} and Lemma \ref{lem:general} that the vectors $\{y\fkm_c\mid y\in \mathcal Y,c\in C\}$ form a basis of $\fkM$, where $\mathcal Y$ is the set of all supermonomials in the elements $\{T_{j,i}^{(r)}\mid 1\lle i<j< m+n,0<r\lle p_i\}$ taken in some fixed order.

For every supertuple $c\in C$, let 
$$
\fkM_c:=\mathrm{span}\{y\fkm_{c'}\mid y\in \mathcal Y,c'\lle c\},\quad \fkM_c':=\mathrm{span}\{y\fkm_{c'}\mid y\in \mathcal Y,c'< c\}.
$$ 
Then $\fkM=\bigcup_{c\in C}\fkM_c$. To complete the proof, it suffices to prove that $\fkM_c$ is a $W(\bar\pi)$-submodule of $\fkM$ with $\fkM_c/\fkM_c'\cong \cM(\bB)$ for $\bB\in\mathrm{Row}(\bar\pi)$ such that the $i$-th row of $\bB$ is $(a_{i,1}-\fks_ic_{i,1},\dots,a_{i,p_i}-\fks_i c_{i,p_i})$ for $1\lle i<m+n$. We prove it by induction on the total order on $C$. By induction hypothesis, $\fkM_c'$ is a $W(\bar\pi)$-submodule of $\fkM$. The vectors
\[
\big\{y\fkm_{c'}+\fkM_c'\mid y\in\mathcal Y,c'\gge c\big\}.
\]
form a basis of the $W(\bar\pi)$-module $\fkM/\fkM_c'$. It follows from the definition of the total order on $C$ that the vector $\overline{\fkm}_c:=\fkm_c+\fkM_c'$ is a vector of maximal $\gl_{m|n}$-weight in $\fkM/\fkM_c'$. Therefore, the vector $\overline{\fkm}_c$ is annihilated by all $E_i^{(r)}$ for $1\lle i<m+n-1$ and $r>s_{i,i+1}$. By Parts (3), (6), (7) of Lemma \ref{lem:relations}, Lemma \ref{lem:general}, and the PBW theorem for $\mathrm Y_{m|n}^{\lle 0}(\sigma)$, we have
\[
u^{p_i}D_i(u)\overline\fkm_c =\prod_{j=1}^{p_i}(u+a_{i,j}-\fks_i c_{i,j})\overline \fkm_c.
\]
Thus, $\overline\fkm_c$ is a highest $\bm\ell$-weight vector of highest $\bm\ell$-weight corresponding to the pyramid $\bar\pi$. By Part (2) of Theorem \ref{thm:W-rep-basis} and the universal property of the Verma modules, we conclude that $\fkM_c$ is a $W(\bar\pi)$-submodule of $\fkM$ and moreover $\fkM_c/\fkM_c'\cong \cM(\bB)$ as they have the same size. 

\subsection{Proof of Theorem \ref{thm:character}}\label{sec:proof-char}
We first prove Theorem \ref{thm:character} under the genericness assumption, i.e. $a_{i,j_1}-a_{i,j_2}\notin \bZ$ for all $1\lle i< m+n$ and $1\lle j_1\ne j_2\lle p_i$. We proceed by induction on $m+n$. The base case $m+n=1$ is trivial. For the inductive step, it follows from Theorem \ref{thm:branching} and the induction hypothesis that the character of $\mathrm{Res}_{W(\bar \pi)}^{W(\pi)}\mc (\bA)$ equals
\[
\sum_{c}\prod_{i=1}^{m+n-1}\prod_{j=1}^{p_i}\left(y_{i,a_{i,j}-\fks_i(c_{i,j,i+1}+\dots+c_{i,j,m+n})}\prod_{k=i+1}^{m+n-1}\Big(\frac{y_{k,a_{i,j}-\fks_i(c_{i,j,k+1}+\dots+c_{i,j,m+n})}}{y_{k,a_{i,j}-\fks_i(c_{i,j,k}+\dots+c_{i,j,m+n})}}\Big)^{\fks_i\fks_k}\right),
\]
where the sum is over all supertuples $c=(c_{i,j,k})_{1\lle i<k< m+n,1\lle j\lle p_i}$. 

Note that by Proposition \ref{prop:center} the operator
\[
\prod_{i=1}^{m+n}\Big((u-\ka_i)^{p_i}D_i(u-\ka_i)\Big)^{\fks_i}
\]
acts on $\mc M(\bA)$ as the scalar $\prod_{i=1}^{m+n}\prod_{j=1}^{p_i}(u+a_{i,j})^{\fks_i}$. Recall from \eqref{eq:y}, every monomial appearing in $\mathrm{ch}\,\mc M(\bA)$ has to simplify to $\prod_{i=1}^{m+n}\prod_{j=1}^{p_i}(u+a_{i,j})^{\fks_i}$ if we replace $y_{i,a}$ by $(u+a)^{\fks_i}$ everywhere. Therefore, one recovers $\mathrm{ch}\,\mc M(\bA)$ uniquely from the above expression, completing the proof of Theorem \ref{thm:character} under the genericness assumption.

Now let us consider the general case where $a_{i,j_1}-a_{i,j_2}\in\bZ$ might happen for $j_1\ne j_2$. Fix an arbitrary tuple of complex numbers $\varsigma=(\varsigma_{i,j})_{1\lle i< m+n,1\lle j\lle p_i}$ such that 
\beq\label{eq:generic}
(a_{i,j_1}+\varsigma_{i,j_1}t)-(a_{i,j_2}+\varsigma_{i,j_2}t)\notin \bZ
\eeq
for all $1\lle i< m+n$, $1\lle j_1\ne j_2\lle p_2$, and $t\in (0,1]$. We write $a_{i,j}^\varsigma(t):=a_{i,j}+\varsigma_{i,j}t$ and $\bA^\varsigma(t)$ for the symmetrized $\pi$-tableau with $(a_{i,j}^\varsigma(t))_{1\lle i< m+n,1\lle j\lle p_i}$ in its $i$-th row. Clearly, we have $a_{i,j}^\varsigma(0)=a_{i,j}$ and $\bA^\varsigma(0)=\bA$.

Define an order on the elements $\bm F:=\{F_{k,i}^{(j)}\mid 1\lle i<k\lle m+n,1\lle j\lle p_i\}$ by the lexicographical order of the tuple $(k,i,j)$. Fix an arbitrary order for the elements $\bm D:=\{D_{i}^{(r)}\mid 1\lle i\lle m+n,1\lle r\lle p_i\}$. Extend the orders so that the $F$'s are always smaller than the $D$'s. 

Let $W^-(\pi)$ and $W^{0}(\pi)$ be the subalgebras of $W(\pi)$ generated by elements from $\bm F$ and $\bm D$, respectively. Let $W^{\lle 0}(\pi)$ be the subalgebra of $W(\pi)$ generated by elements from $\bm F\cup\bm D$.

For a supertuple $\theta=(\theta_{i,j,k})_{1\lle i<k\lle m+n,1\lle j\lle p_i}$, define the ordered monomial
\[
F_\theta =\prod_{k=2}^{m+n}\prod_{i=1}^{k-1}\prod_{j=1}^{p_i}\big(F_{k,i}^{(j)}\big)^{\theta_{i,j,k}}.
\]
Then it follows from Proposition \ref{prop:PBW-truncated} that the ordered monomials $F_\theta$ for all supertuples $\theta$ form a basis of $W^-(\pi)$, where we put elements of smaller order to the left of larger ones. Similarly, we have a PBW type basis for $W^{\lle 0}(\pi)$. Note that $W(\pi)$ is $\mathbf Q_{m|n}$-graded. For $\alpha\in \mathbf Q_{m|n}^+$, denote by $C_\alpha$ the set of all supertuples $\theta$ such that the degree of $F_\theta$ is $-\alpha$. Clearly, $C_\alpha$ is a finite set. Since $\deg D_i^{(r)}=0$, by \eqref{dr4} and the PBW basis, there exist unique constants $\Gamma^{(r,s)}_{i;\theta,\theta'}$ for $1\lle i\lle m+n$, $\theta,\theta'\in C_\alpha$ and $1\lle r,s\lle p_i$ such that 
\beq\label{eq:DFtheta}
D_i^{(r)}F_\theta=\sum_{s=1}^{p_i}\sum_{\theta'\in C_\alpha}\Gamma^{(r,s)}_{i;\theta,\theta'}F_{\theta'} D_i^{(s)}.
\eeq

For the Verma modules $\mc M(\bA^\varsigma(t))$, it follows from Theorem \ref{thm:W-rep-basis} that we can identify all these modules with the same vector superspace $W^-(\pi)$. Due to \eqref{eq:Di-eigenvalues}, for each $\alpha\in \mathbf Q_{m|n}^+$, the subspace 
$$
W_\alpha^-(\pi):=\mathrm{span}\{F_\theta\mid \theta\in C_\alpha\}
$$
is invariant under the action of $D_i^{(r)}$. It follows from \eqref{eq:Di-eigenvalues} and \eqref{eq:DFtheta} that the element $D_i^{(r)}$ acts on the vector $F_\theta$ by
\[
D_i^{(r)}F_\theta=\sum_{s=1}^{p_i}\sum_{\theta'\in C_\alpha}\Gamma^{(r,s)}_{i;\theta,\theta'} F_{\theta'} e_s\big(a_{i,1}^\varsigma(t)+\ka_i,\dots,a_{i,p_i}^\varsigma(t)+\ka_i\big),
\]
where $e_s(x_1,\dots,x_p)$ is the $s$-th elementary symmetric polynomials. We can also think this is obtained from \eqref{eq:DFtheta} applying to the highest $\bm\ell$-weight vector $v_+$. Hence all the entries of the matrix corresponding to the action of $D_i^{(r)}$ on $W_\alpha^-(\pi)\subset \mc M(\bA^\varsigma(t))$ with respect to the basis $\mathrm{span}\{F_\theta\mid \theta\in C_\alpha\}$ are polynomials in $t$. Moreover, it follows from \eqref{eq:generic}
and the generic case established above that the $q$-character of $\mathrm{ch}\big(\mc M(\bA^\varsigma(t))\big)$ for $t\in (0,1]$ restricted to the subspace $W_\alpha^-(\pi)$ is given by
\[
\sum_{c\in C_\alpha}\prod_{i=1}^{m+n}\prod_{j=1}^{p_i}\left(y_{i,a_{i,j}^\varsigma(t)-\fks_i(c_{i,j,i+1}+\dots+c_{i,j,m+n})}\prod_{k=i+1}^{m+n}\Big(\frac{y_{k,a_{i,j}^\varsigma(t)-\fks_i(c_{i,j,k+1}+\dots+c_{i,j,m+n})}}{y_{k,a_{i,j}^\varsigma(t)-\fks_i(c_{i,j,k}+\dots+c_{i,j,m+n})}}\Big)^{\fks_i\fks_k}\right).
\]
It is well known that the eigenvalues of a matrix depend continuously on its entries, see e.g. \cite[Coro.~VI.1.6]{Bha97}. 

\begin{lem}\label{lem:continuous}
Let $t\to X(t)$ be a continuous map from $[0,1]$ into the space of $r\times r$ complex matrices. Then there exist continuous functions $\la_1(t),\dots,\la_r(t)$ that, for each $t\in [0,1]$, are eigenvalues of $X(t)$.
\end{lem}

Note that we know explicitly the generalized joint eigenvalues of $\{D_i^{(r)}\mid 1\lle i\lle m+n,0<r\lle p_i\}$ acting on the subspace $W_\alpha^-(\pi)$ for $t\in(0,1]$. By applying Lemma \ref{lem:continuous} to all possible linear combinations of the elements $\{D_i^{(r)}\mid 1\lle i\lle m+n,0<r\lle p_i\}$, we concluded that the generalized joint eigenvalues (with multiplicities) of $\{D_i^{(r)}\mid 1\lle i\lle m+n,0<r\lle p_i\}$ acting on the subspace $W_\alpha^-(\pi)$ for $t=0$ are encoded by
\[
\sum_{c\in C_\alpha}\prod_{i=1}^{m+n}\prod_{j=1}^{p_i}\left(y_{i,a_{i,j}-\fks_i(c_{i,j,i+1}+\dots+c_{i,j,m+n})}\prod_{k=i+1}^{m+n}\Big(\frac{y_{k,a_{i,j}-\fks_i(c_{i,j,k+1}+\dots+c_{i,j,m+n})}}{y_{k,a_{i,j}-\fks_i(c_{i,j,k}+\dots+c_{i,j,m+n})}}\Big)^{\fks_i\fks_k}\right).
\]
Thus
we have $\mathrm{ch}\,\mc M(\bA)=\lim_{t\rightarrow 0}\mathrm{ch}\big(\mc M(\bA^\varsigma(t))\big)$, completing the proof for the general case.

\subsection{Isomorphism of Center}
As an application of the character formula, we show that the center of $W$-superalgebra depends only on the superdimension $(M|N)$ up to isomorphism, and is independent of the shape of the pyramid. 
Our approach combines the arguments in \cite{BK08, BG19}.

The following simple lemma shows that generic Verma modules are irreducible.
\begin{lem}\label{lem:simple-generic}
Let $\bA\in\mathrm{Row}(\pi)$ with entries $a_{i,1},\dots,a_{i,p_i}$ on its $i$-{th} row for each $1\lle i\lle m+n$, where $a_{i_1,j_1}-a_{i_2,j_2}\notin \bZ$ for all $1\lle i_1,i_2\lle m+n$ and $1\lle j_1,j_2\lle p_i$ provided $(i_1,j_1)\ne (i_2,j_2)$. Then we have the following.
\begin{enumerate}
    \item The Verma module $\mc M(\bA)$ is irreducible.
    \item Let $\bA_1,\dots,\bA_\ell$ be the columns of any representative of the row-symmetrized $\pi$-tableau $\bA\in \mathrm{Row}(\pi)$, so that $\bA\sim_{\mathrm{row}}\bA_1\otimes \bA_2\otimes \cdots\otimes \bA_{\ell}$. Then
\[
\mc M(\bA)\cong \mc M(\bA_1)\boxtimes\cdots \boxtimes \mc M(\bA_\ell).
\]
\end{enumerate}
\end{lem}
\begin{proof}Part (2) follows from Part (1) and Corollary \ref{cor:char}. Thus it suffices to prove Part (1).

Suppose $\mc M(\bA)$ is reducible, then we can assume $\mc L(\bB)$ is a subquotient of $\mc M(\bA)$, where $\bB\in\mathrm{Row}(\pi)$ with entries $b_{i,1},\dots,b_{i,p_i}$ on its $i$-th row for each $1\lle i\lle m+n$ and $\bA\ne \bB$. Therefore the  $\bm\ell$-weight corresponding to the highest $\bm\ell$-weight vector of $\mc L(\bB)$ is $\prod_{i=1}^{m+n}\prod_{j=1}^{p_i}y_{i,b_{i,j}}$. Hence there exists a supertuple $c=(c_{i,j,k})_{1\lle i<k\lle m+n,1\lle j\lle p_i}$ such that
\[
\prod_{i=1}^{m+n}\prod_{j=1}^{p_i}y_{i,b_{i,j}}=\prod_{i=1}^{m+n}\prod_{j=1}^{p_i}\left(y_{i,a_{i,j}-\fks_i(c_{i,j,i+1}+\dots+c_{i,j,m+n})}\prod_{k=i+1}^{m+n}\Big(\frac{y_{k,a_{i,j}-\fks_i(c_{i,j,k+1}+\dots+c_{i,j,m+n})}}{y_{k,a_{i,j}-\fks_i(c_{i,j,k}+\dots+c_{i,j,m+n})}}\Big)^{\fks_i\fks_k}\right).
\]
Since $\bA\ne \bB$, the supertuple $c$ is non-trivial. Let ${\mathfrak i}$ be maximal such that $c_{i,j,{\mathfrak i}}\ne 0$ for some $1\lle i<{\mathfrak i}$ and $1\lle j\lle p_i$. Then we have
\[
\prod_{j=1}^{p_{\mathfrak i}}y_{{\mathfrak i},b_{{\mathfrak i},j}}=\prod_{j=1}^{p_{\mathfrak i}} y_{{\mathfrak i},a_{{\mathfrak i},j}}\prod_{i=1}^{{\mathfrak i}-1}\prod_{j=1}^{p_i}\Big(\frac{y_{{\mathfrak i},a_{i,j}}}{y_{{\mathfrak i},a_{i,j}-\fks_{\mathfrak i}c_{i,j,{\mathfrak i}}}}\Big)^{\fks_i\fks_{\mathfrak i}}.
\]
Since $a_{i_1,j_1}-a_{i_2,j_2}\notin \bZ$ for $(i_1,j_1)\ne (i_2,j_2)$, there is no way to rewrite the RHS as a monomial in $y_{{\mathfrak i},b_{{\mathfrak i},j}}$ as the LHS. Hence we obtain a contradiction, completing the proof.
\end{proof}

Recall the notations from Section \ref{sec:W-alg} and, in particular, the Miura transform $\mu$ from \eqref{miura}.
\begin{lem}\label{lem:mucent1}
We have $\mu\big(Z(W(\pi))\big)\subset Z(\rU(\h))$.
\end{lem}
\begin{proof}
We need to show that $[\mu(z),h]=0$ for $z\in Z(W(\pi))$ and $h\in \rU(\h)$. By the standard fact that the annihilator of any Zariski dense set of Verma (super)modules is zero (see e.g. \cite[Proof of Lem.~13.1.4]{Mus12}), it suffices to check that $[\mu(z),h]$ annihilates $\mc M(\bA_1)\boxtimes\cdots \boxtimes\mc M(\bA_\ell)$ for generic $\bA\in \mathrm{Tab}(\pi)$ with columns $\bA_1,\dots,\bA_{\ell}$. By Lemma \ref{lem:simple-generic}, $\mc M(\bA_1)\boxtimes\cdots \boxtimes\mc M(\bA_\ell)$ is irreducible for generic $\bA\in \mathrm{Tab}(\pi)$ when viewed as a $W(\pi)$-module via pulling back through $\mu$. Therefore $\mu(z)$ acts on $\mc M(\bA_1)\boxtimes\cdots \boxtimes\mc M(\bA_\ell)$ as a scalar, implying that $[\mu(z),h]$ acts on it as zero.
\end{proof}

Let $\mathsterling:\mathrm U(\fkp)\rightarrow \mathrm U(\fkp)$ be the superalgebra automorphism given by 
\[
\mathsterling(e_{ij})=\tilde e_{i,j}=e_{ij}+\delta_{i,j} \fks_{\row(i)}(\hbar-\check{q}_{\col(i)} - \check{q}_{\col(i)+1} -\cdots-\check{q}_\ell),
\]
where $\hbar=m-n$ and $\check{q}_{a}=m_a-n_a$, for all $1\lle a\lle \ell$. Let $\psi:\mathrm U(\g)\twoheadrightarrow \mathrm U(\fkp)$ be the linear map defined as the composition first of the projection 
$\mathrm{pr}_\chi:\mathrm U(\g)\to \mathrm U(\fkp)$ along the direct sum decomposition $\mathrm U(\g)=\mathrm U(\fkp)\oplus I_\chi$ then the automorphism $\mathsterling$. The restriction of $\psi$ to $Z(\mathrm U(\g))$ gives a well-defined algebra homomorphism
\beq\label{ceniso}
\psi:Z(\mathrm U(\g))\to Z(W(\pi)),
\eeq
where the image belongs to the center of $W(\pi)$.    

Let $\fks(\pi)$ be the parity sequence obtained from reading the parity of boxes of $\pi$ down-columns from left to right.  
Recall the series $z_{M|N,\fks(\pi)}(u)$ defined in \eqref{eq:zdef}.
We calculate the image of $z_{M|N,\fks(\pi)}(u)$ under the composition $\mu\circ\psi:Z(\mathrm U(\g))\to \mathrm U(\h)$ as follows.
Since the Miura map $\mu$ can be understood as the iteration of the parabolic induction \eqref{Wdel} by $(\ell-1)$ times, the following result can be deduced as in the one-column case discussed in \S \ref{1col-case}:
\beq\label{comd1}
\begin{split}
\mu\circ\psi\big( z_{M|N,\fks(\pi)}(u) \big)=&\,\mathsterling_1\big( z_{m_1|n_1,\fks(\pi_1)}(u)\big) \otimes \mathsterling_2\big( z_{m_2|n_2,\fks(\pi_2)}(u-\check{q_1}) \big)\otimes \cdots\\
&\,\cdots\otimes \mathsterling_\ell\big( z_{m_\ell|n_\ell,\fks(\pi_\ell)}(u-\check{q_1}-\check{q_2}-\cdots-\check{q}_{\ell-1})\big)\in \mathrm{U}(\h)[\![u^{-1}]\!],
\end{split}
\eeq
where $\fks(\pi_a)$ stands for the parity sequence obtained from the reading of the $a$-{th} column of $\pi$ and $\mathsterling_a:\mathrm{U}(\gl_{m_a|n_a})\to \mathrm{U}(\gl_{m_a|n_a})$ is the automorphism sending $e_{i,j}$ to $e_{i,j}+\delta_{i,j}(-1)^{\tp(i)}(\hbar-\check{q}_a)$, for each $1\lle a\lle \ell$. 

Under the identification $\mathrm{U}(\h)\cong \mathrm{U}(\gl_{m_1|n_1})\otimes \cdots \otimes \mathrm{U}(\gl_{m_\ell|n_\ell})$, the subspace $S(\fkc)\subseteq \mathrm{U}(\h)$ is identified with 
\be
S(\fkc)\cong S(\fkc_1)\otimes S(\fkc_2) \otimes \cdots \otimes S(\fkc_\ell),
\ee
where $\fkc_a$ denotes the Cartan subalgebra of $\gl_{m_a|n_a}$, corresponding to the $a$-{th} column of $\pi$.
Let $\hc_a:Z(\mathrm{U}(\gl_{m_a|n_a}))\stackrel{\sim}{\rightarrow} I(\fkc_a)\subseteq S(\fkc_a)$ denote the Harish-Chandra map with respect to the parity sequence $\fks(\pi_a)$. 
Consider the composition 
\beq\label{comd2}
\hc:=(\hc_1\otimes\cdots\otimes\hc_{\ell} )\circ ( {\mathsterling_1}^{-1} \otimes \cdots \otimes {\mathsterling_\ell^{-1})} :\mathrm{U}(\h)\rightarrow S(\fkc).
\eeq
By \eqref{comd1} and \eqref{comd2}, we see that 
\[
\HC^{\fks(\pi)}\big(z_{M|N,\fks(\pi)}(u)\big) = \hc \circ \mu\circ\psi\big( z_{M|N,\fks(\pi)}(u) \big).
\]
Together with Lemma \ref{lem:mucent1}, the following diagram commutes:
\beq\label{HChc}
\begin{CD}
Z(\mathrm{U}(\g))&& @> \HC >> && S(\mathfrak{c})\\
@V \psi VV &  && &@AA \hc A\\
Z( W(\pi) )  && @> \mu >> &&  Z(\mathrm{U}(\mathfrak{h})) 
\end{CD}
\eeq

\begin{thm}\label{isocent}
The map $\psi:Z(\mathrm{U}(\g))\to Z( W(\pi) )$ is an algebra isomorphism. In particular, the elements $\psi(z_{M|N}^{(r)})$ for all $r\gge 1$ generate the center $Z( W(\pi) )$.
\end{thm}
\begin{proof}
The injectivity of $\psi$ follows from that of $\HC$ and the commutative diagram \eqref{HChc}.
It remains to prove that $\psi$ is surjective. Since $\hc\circ\mu$ is injective and we know $\HC\big( Z(\mathrm{U}(\g)) \big)=I(\fkc)$ by Theorem \ref{thmHC}, it suffices to show $\hc\circ \mu\big(Z(W(\pi))\big)\subseteq I(\fkc)$.  Note that the image of $\hc$ is contained in $I(\fkc_1) \otimes \cdots \otimes I(\fkc_\ell)$. It reduces to show that $\hc\circ\mu\big(Z(W(\pi))\big)\subset S(\mathfrak c)^{\fkS_M\times \fkS_N}$ as $I(\fkc)=\big(I(\fkc_1) \otimes \cdots \otimes I(\fkc_\ell)\big)\cap \big(S(\mathfrak c)^{\fkS_M\times \fkS_N}\big)$. 

Let $\bA$ and $\bB$ be $\pi$-tableaux that share the same content in even and odd boxes, respectively. By the definition of the Harish-Chandra homomorphism, to show $\hc\circ\mu\big(Z(W(\pi))\big)\subset S(\mathfrak c)^{\fkS_M\times \fkS_N}$, we need to show that any element $z\in Z( W(\pi) )$ acts on the modules $\mc M(\bA_1)\boxtimes\cdots \boxtimes\mc M(\bA_\ell)$ and $\mc M(\bB_1)\boxtimes\cdots \boxtimes\mc M(\bB_\ell)$ by the same scalar, where $\bA_i$ and $\bB_i$ denote the $i$-th column of $\bA$ and $\bB$, respectively.
Similar to the argument in the proof of Lemma \ref{lem:mucent1}, it suffices to prove the statement for generic $\bA$ and $\bB$, since the set of generic $\pi$-tableaux is Zariski dense in the set of all $\pi$-tableaux.

If $\bB$ is obtained from $\bA$ by permuting even (resp. odd) entries within columns, then this is obvious since the image of $\hc$ is contained in $I(\fkc_1) \otimes \cdots \otimes I(\fkc_\ell)$.

If $\bB$ is obtained from $\bA$ by permuting even (resp. odd) entries within rows, then by \eqref{eq:Verma-W}--\eqref{eq:lam-tab} we have $\mc M(\bA)=\mc M(\bB)$. Since $\bA$ and $\bB$ are generic, then by Lemma \ref{lem:simple-generic} we have 
\[
\mc M(\bA_1)\boxtimes\cdots \boxtimes \mc M(\bA_\ell)\cong \mc M(\bA)=\mc M(\bB)\cong \mc M(\bB_1)\boxtimes\cdots \boxtimes \mc M(\bB_\ell),
\] 
as irreducible $W(\pi)$-modules. As a consequence, any element $z\in Z( W(\pi) )$ acts on these modules by exactly the same scalar. The general case follows from these two special cases.
\end{proof}

\appendix
\section{Proof of Theorem \ref{thm:vanish}}\label{sec:app}
We follow the same strategy of \cite[\S2.4 \& \S3.7]{BK08}.
\subsection{Parabolic presentation}
We recall the parabolic presentations of $\Y$ from \cite[Proof of Prop.~5.15]{Pe21} which is slight different from \cite{Pe16}. By a \textit{shape} we mean a tuple $\nu=(\nu_1,\dots,\nu_z)$ of positive integers summing to $m+n$. We say that a shaper $\nu=(\nu_1,\dots,\nu_z)$ is admissible for $\sigma$ if $s_{i,j}=0$ for all $\nu_1+\cdots+\nu_{a-1}<i,j\lle \nu_1+\cdots+\nu_{a}$ and $1\lle a \lle z$, in which case we define
\beq\label{eq:sabnu}
s_{a,b}(\nu):=s_{\nu_1+\cdots+\nu_a,\nu_1+\cdots+\nu_b}
\eeq
for $1\lle a,b\lle z$.

Given an admissible shape $\nu=(\nu_1,\dots,\nu_z)$. Let $^\nu T_{a,b}(u)$ be the $\nu_a\times\nu_b$ matrix whose $ij$-entry is given by $T_{\nu_1+\cdots+\nu_{a-1}+i,\nu_1+\cdots+\nu_{b-1}+j}(u)$ for each $1\lle a,b\lle z$. Denote $^\nu T(u)$ the $z\times z$-block matrix with the $ab$-block given by $^\nu T_{a,b}(u)$. Exploit the block-wise Gauss decomposition of $^\nu T(u)$ to define the matrices $^{\nu}D(u)$, $^{\nu}E(u)$, and $^{\nu}F(u)$ such that
\[
^{\nu}T(u)={}^{\nu}F(u){}^{\nu}D(u){}^{\nu}E(u),
\]
where ${}^{\nu}D(u)$ is a $z\times z$ diagonal block matrix with $aa$-block $^\nu D_a(u)$ being a matrix of size $\nu_a\times\nu_a$, ${}^{\nu}E(u)$ is a $z\times z$ upper unitriangular block matrix with $ab$-block $^\nu E_{a,b}(u)$ being a matrix of size $\nu_a\times\nu_b$, and ${}^{\nu}F(u)$ is a $z\times z$ lower unitriangular block matrix with $ba$-entry $^\nu F_{b,a}(u)$ being a matrix of size $\nu_b\times\nu_a$. Thus  $^\nu E_{a,a}(u)={}^\nu F_{a,a}(u)=1$ are the identity matrix of size $\nu_a\times \nu_a$ for each $1\lle a\lle z$ and
\beq\label{eq:GD-comp}
^\nu T_{a,b}(u)=\sum_{c=1}^{\min(a,b)} {}^\nu F_{a,c}(u)^\nu D_c(u) ^\nu E_{c,b}(u).
\eeq

Denote the $(i,j)$-entry of $^\nu D_a(u)$, $^\nu E_{a,b}(u)$, and $^\nu F_{b,a}(u)$ by
\begin{align*}
^\nu D_{a;i,j}(u)&=\sum_{r\gge 0} {}^\nu D_{a;i,j}^{(r)}u^{-r},\qquad 1\lle i,j\lle \nu_a,\\
^\nu E_{a,b;i,j}(u)&=\sum_{r\gge 0} {}^\nu E_{a,b;i,j}^{(r)}u^{-r},\qquad 1\lle i\lle \nu_a,1\lle j\lle \nu_b,\\
^\nu F_{b,a;i,j}(u)&=\sum_{r\gge 0} {}^\nu F_{b,a;i,j}^{(r)}u^{-r},\qquad 1\lle j\lle \nu_a,1\lle i\lle \nu_b,
\end{align*}
respectively, where ${}^\nu D_{a;i,j}^{(r)}$, ${}^\nu E_{a,b;i,j}^{(r)}$, ${}^\nu F_{b,a;i,j}^{(r)}$ are elements of $\Y$. 

Define $^{\nu}|i|_a$ for $1\lle a\lle z$, $1\lle i\lle \nu_a$ by the rule:
\[
^{\nu}|i|_a:=|\nu_1+\cdots+\nu_{a-1}+i|.
\]

We shall often omit the superscript $\nu$ and write simply $D_{a;i,j}^{(r)}$, $E_{a,b;i,j}^{(r)}$, $F_{b,a;i,j}^{(r)}$. We also use the abbreviation: $E_{a;i,j}^{(r)}:=E_{a,a+1;i,j}^{(r)}$, $F_{a;i,j}^{(r)}:=F_{a+1,a;i,j}^{(r)}$. 

The anti-isomorphism $\tau$ from \eqref{taudef} satisfies
\begin{align*}
\tau(D_{a;i,j}^{(r)})&=  (-1)^{|i|_a|j|_a+|j|_a} D_{a;i,j}^{(r)},\\
\tau(E_{a,b;i,j}^{(r)})&=  (-1)^{|i|_a|j|_b+|j|_b} E_{a,b;i,j}^{(r)},\\
\tau(F_{b,a;i,j}^{(r)})&=  (-1)^{|i|_b|j|_a+|j|_a} F_{b,a;i,j}^{(r)}.
\end{align*}

It is known from \cite[\S3]{Pe16} that $\Y$ is generated by the elements
\begin{align*}
&\big\{D_{a;i,j}^{(r)}\mid 1\lle a\lle z,1\lle i,j\lle \nu_a,r>0\big\},\\
&\big\{E_{a;i,j}^{(r)}\mid 1\lle a<m,1\lle i\lle \nu_a,1\lle j\lle \nu_{a+1},r>s_{a,a+1}(\nu)\big\},\\
&\big\{F_{a;i,j}^{(r)}\mid 1\lle a<m,1\lle i\lle \nu_{a+1},1\lle j\lle \nu_{a},r>s_{a+1,a}(\nu)\big\}
\end{align*}
subject to relations recorded explicitly in \cite[(5.3)--(5.18)]{Pe21} (see also \cite{Pe16}). Moreover, the supermonomials in the elements
\begin{align*}
&\big\{D_{a;i,j}^{(r)}\mid 1\lle a\lle z,1\lle i,j\lle \nu_a,r>0\big\},\\
&\big\{E_{a,b;i,j}^{(r)}\mid 1\lle a<b\lle z,1\lle i\lle \nu_a,1\lle j\lle \nu_{b},r>s_{a,b}(\nu)\big\},\\
&\big\{F_{b,a;i,j}^{(r)}\mid 1\lle a<b\lle z,1\lle i\lle \nu_{b},1\lle j\lle \nu_{a},r>s_{b,a}(\nu)\big\}
\end{align*}
taken in any fixed order form a basis of $\Y$. Note that the definition of the higher order root elements $E_{a,b;i,j}^{(r)}$ and $E_{a,b;i,j}^{(r)}$ given here is different from \cite{Pe16} but the same as in \cite{Pe21}. The equivalence of the two definitions can be verified by Lemma \ref{lem:BK08-2.2} below.

\subsection{Preliminary lemmas}
\begin{lem}[{\cite[Lem.~2.2]{BK08}}]\label{lem:BK08-2.2}
For $1\lle a<b-1<m+n$, $1\lle i\lle \nu_a$, $1\lle j\lle \nu_b$ and $r>s_{a,b}(\nu)$, we have
\[
E_{a,b;i,j}^{(r)}=(-1)^{|g|_{b-1}}\big[E_{a,b-1;i,g}^{(r-s_{b-1,b}(\nu))},E_{b-1;g,j}^{(s_{b-1,b}(\nu)+1)}\big]
\]
for any $1\lle g\lle \nu_{b-1}$. Similarly, for $1\lle a<b-1<m+n$, $1\lle i\lle \nu_b$, $1\lle j\lle \nu_a$ and $r>s_{b,a}(\nu)$, we have
\[
F_{b,a;i,j}^{(r)}=(-1)^{|g|_{b-1}}\big[F_{b-1;i,g}^{(s_{b-1,b}(\nu)+1)},F_{b-1,a;g,j}^{(r-s_{b-1,b}(\nu))}\big]
\]
for any $1\lle g\lle \nu_{b-1}$.
\end{lem}
\begin{proof}
It suffices to prove the identities for the $E$'s as that of $F$'s follow by applying the anti-isomorphism $\tau$ from \eqref{taudef}. We proceed by backward induction on the length of the admissible shape $\nu=(\nu_1,\dots,\nu_z)$. The base case $z=m+n$ follows by definition \eqref{gedef}. Now suppose $z<m+n$. Pick $1\lle p\lle m+n$ and $x,y>0$ such that $\nu_p=x+y$, and then let $\mu=(\nu_1,\dots,\nu_{p-1},x,y,\nu_{p+1},\dots,\nu_z)$, an admissible shape of length strictly less than that of $\nu$. Then a straightforward calculation implies that for each $1\lle a<b\lle m+n$, $1\lle i\lle \nu_a$ and $1\lle j\lle \nu_b$  that
\be
^\nu E_{a,b;i,j}(u)=
\begin{cases}
^\mu E_{a,b;i,j}(u)    &\text{ if }b<p,\\
^\mu E_{a,b;i,j}(u)    &\text{ if }b=p,j\lle x,\\
^\mu E_{a,b+1;i,j-x}(u)    &\text{ if }b=p,j> x,\\
^\mu E_{a,b+1;i,j}(u)    &\text{ if }a<p,b> p,\\
^\mu E_{a,b+1;i,j}(u)    -\sum_{l=1}^y {}^\mu E_{a;i,l}(u) {}^\mu E_{a+1,b+1;l,j}(u) & \text{ if }a=p,i\lle x,\\
^\mu E_{a+1,b+1;i-x,j}(u)    &\text{ if }a=p,i>x,\\
^\mu E_{a+1,b+1;i,j}(u)    &\text{ if }a>p.
\end{cases}
\ee
For $b>a+1$, we need to prove that
\beq\label{eq:app1}
^\nu E_{a,b;i,j}(u)=(-1)^{{}^\nu|g|_{b-1}}\big[{}^\nu E_{a,b-1;i,g}(u),{}^\nu E_{b-1;g,j}^{s_{b-1,b}(\nu)+1}\big]u^{-s_{b-1,b}(\nu)}
\eeq
for each $1\lle g\lle \nu_{b-1}$. We shall rewrite both sides of the identity that we need to prove in terms of the $^\mu E$'s and then use the induction hypothesis
\[
^\mu E_{a,b;i,j}(u)=(-1)^{{}^\mu|g|_{b-1}}\big[{}^\mu E_{a,b-1;i,g}(u),{}^\mu E_{b-1;g,j}^{s_{b-1,b}(\mu)+1}\big]u^{-s_{b-1,b}(\mu)}
\]
for each $1\lle a<b-1\lle m+n$, $1\lle i\lle \mu_a$, $1\lle j\lle \mu_b$ and $1\lle g\lle \mu_{b-1}$. 

We proceed case by case.

(1) If $b<p$ or $b=p,j\lle x$, then it is obvious.

(2) Suppose $b=p,j>x$. By induction hypothesis, we have
    \[
    ^\mu E_{a,b+1;i,j-x}(u)=(-1)^{{}^\mu|h|_{b}}\big[{}^\mu E_{a,b;i,h}(u),{}^\mu E_{b;h,j-x}\big]u^{-s_{b,b+1}(\mu)}
    \]
    for $1\lle h\lle x$. Since $s_{b,b+1}(\mu)=0$ (as $\nu$ is admissible) and $^{\mu} E_{b;h,j-x}^{(1)}={}^\nu D_{b;h,j}^{(1)}$, we find that
    \beq\label{eq:helper01}
    ^{\nu} E_{a,b;i,j}(u)=(-1)^{{}^\nu|h|_{b}}\big[{}^\nu E_{a,b;i,h}(u),{}^\nu D_{b;h,j}^{(1)}\big].
    \eeq
    Thus using the relations established in the previous cases and the defining relations of $\Y$, we find that
    \begin{align*}
    &(-1)^{{}^\nu|h|_{b}}\big[{}^\nu E_{a,b;i,h}(u),{}^\nu D_{b;h,j}^{(1)}\big]\\
    =\,&(-1)^{{}^\nu|h|_{b}+{}^\nu|g|_{b-1}}\Big[\big[{}^\nu E_{a,b-1;i,g}(u),{}^\nu E_{b-1;g,h}^{(s_{b-1,b}(\nu)+1)}\big],{}^\nu D_{b;h,j}^{(1)}\Big]u^{-s_{b-1,b}(\nu)}\\
    =\,&(-1)^{{}^\nu|h|_{b}+{}^\nu|g|_{b-1}}\Big[{}^\nu E_{a,b-1;i,g}(u),\big[{}^\nu E_{b-1;g,h}^{(s_{b-1,b}(\nu)+1)},{}^\nu D_{b;h,j}^{(1)}\big]\Big]u^{-s_{b-1,b}(\nu)}\\
    =\,&(-1)^{{}^\nu|g|_{b-1}}\big[{}^\nu E_{a,b-1;i,g}(u),{}^\nu E_{b-1;g,j}^{(s_{b-1,b}(\nu)+1)}\big]u^{-s_{b-1,b}(\nu)}
    \end{align*}
    for any $1\lle g\lle \nu_{b-1}$.

(3) Suppose $a<p,b>p$. If $b>p+1$ or $b=p+1,g>x$, then it is easy. Now consider the case $b=p+1$ and $g\lle x$. It is known in case (2) that
\beq\label{eq:helper02}
^\nu E_{a,b;i,j}(u)=(-1)^{{}^\nu|x+1|_{b-1}}\big[{}^\nu E_{a,b-1;i,x+1}(u),{}^\nu E_{b-1;x+1,j}^{(s_{b-1,b}(\nu)+1)}\big]u^{-s_{b-1,b}(\nu)}.
\eeq
Using the cases already considered to express $^\nu E_{a,b-1;i,g}^{(r)}$ as a commutator then using the relation \cite[(5.13)]{Pe21}, one has that $[{}^\nu E_{a,b-1;i,g}^{(r)},{}^\nu E_{b-1;x+1,j}^{(s)}]=0$. Applying $[{}^\nu D_{b-1;g,x+1}^{(1)},?]$ to it and using the relation \eqref{eq:helper01}, we have
\[
(-1)^{{}^\nu|x+1|_{b-1}}\big[{}^\nu E_{a,b-1;i,x+1}^{(r)},{}^\nu E_{b-1;x+1,j}^{(s)}\big]=(-1)^{{}^\nu|g|_{b-1}}\big[{}^\nu E_{a,b-1;i,g}^{(r)},{}^\nu E_{b-1;g,j}^{(s)}\big].
\]
In particular,
\[
(-1)^{{}^\nu|x+1|_{b-1}}\big[{}^\nu E_{a,b-1;i,x+1}(u),{}^\nu E_{b-1;x+1,j}^{(s_{b-1,b}(\nu)+1)}\big]=(-1)^{{}^\nu|g|_{b-1}}\big[{}^\nu E_{a,b-1;i,g}(u),{}^\nu E_{b-1;g,j}^{(s_{b-1,b}(\nu)+1)}\big].
\]
Combining this with \eqref{eq:helper01}, we get that 
$$
^\nu E_{a,b;i,j}(u)=(-1)^{{}^\nu|g|_{b-1}}\big[{}^\nu E_{a,b-1;i,g}(u),{}^\nu E_{b-1;g,j}^{(s_{b-1,b}(\nu)+1)}\big]u^{-s_{b-1,b}(\nu)}
$$
as required.

(4) Suppose $a=p,i\lle x$. The LHS of \eqref{eq:app1} is equal to
\[
^\mu E_{a,b+1;i,j}(u)-\sum_{l=1}^y{}^\mu E_{a;i,l}(u){}^\mu E_{a+1,b+1;l,j}(u)
\]
while the RHS equals
\[
(-1)^{{}^\mu|g|_b}\Big[
^\mu E_{a,b;i,g}(u)-\sum_{l=1}^y{}^\mu E_{a;i,l}(u){}^\mu E_{a+1,b;l,g}(u),{}^\mu E_{b;g,j}^{(s_{b,b+1}(\mu)+1)}
\Big]u^{-s_{b,b+1}(\mu)}.
\]
Note that ${}^\mu E_{a;i,l}(u)$ and ${}^\mu E_{b;g,j}^{(s_{b,b+1}(\mu)+1)}$ supercommute. The equality \eqref{eq:app1} follows from the induction hypothesis.

(5) If $a=p,i>x$ or $a>p$, then it is also easy.\end{proof}

Introduce another family of elements of $\Y$. Recall that ${}^\nu E_{a,a}(u)$ and $^\nu F_{a,a}(u)$ are both the identity. Define
\begin{align}
^\nu \overline E_{a,b}(u)&:= {}^\nu E_{a,b}(u)-\sum_{c=a}^{b-1}{}^\nu E_{a,c}(u){}^{\nu}E_{c,b}^{(s_{c,b}(\nu)+1)}u^{-s_{c,b}(\nu)-1},\label{eq:barE-def}\\
^\nu \overline F_{b,a}(u)&:= {}^\nu F_{b,a}(u)-\sum_{c=a}^{b-1}{}^\nu F_{b,c}^{(s_{b,c}(\nu)+1)}{}^{\nu}F_{c,a}(u)u^{-s_{b,c}(\nu)-1},\label{eq:barF-def}
\end{align}
for $1\lle a<b\lle z$. By convention, we have $$^\nu \overline E_{a,a+1}(u)={}^\nu \overline E_{a,a+1}(u),\qquad ^\nu \overline F_{a+1,a}(u)={}^\nu F_{a+1,a}(u).$$ As usual, the $ij$-entry of $^\nu \overline E_{a,b}(u)$ and $^\nu \overline F_{b,a}(u)$ are denoted by
\begin{align*}
^\nu \overline E_{a,b;i,j}(u)=\sum_{r>s_{a,b}(\nu)}^\nu \overline E_{a,b;i,j}^{(r)}u^{-r},\quad 
^\nu \overline F_{b,a;j,i}(u)=\sum_{r>s_{b,a}(\nu)}^\nu \overline F_{b,a;j,i}^{(r)}u^{-r},
\end{align*}
for $1\lle i\lle \nu_a$ and $1\lle j\lle \nu_b$. Again we shall drop the superscript $\nu$.

\begin{lem}[{\cite[Lem.~2.3]{BK08}}]\label{lem:BK08-2.3}
For $1\lle a<b-1<m+n$, $1\lle i\lle \nu_a$, $1\lle j\lle \nu_b$ and $r>s_{a,b}(\nu)+1$, we have
\[
\overline E_{a,b;i,j}^{(r)}=(-1)^{|g|_{b-1}}[E_{a,b-1;i,g}^{(r-s_{b-1,b}(\nu)-1)},E_{b-1;g,j}^{(s_{b-1,b}(\nu)+2)}]
\]
for any $1\lle g\lle \nu_{b-1}$. Similarly, for $1\lle a<b-1<m+n$, $1\lle i\lle \nu_b$, $1\lle j\lle \nu_a$ and $r>s_{b,a}(\nu)$, we have
\[
\overline F_{b,a;i,j}^{(r)}=(-1)^{|g|_{b-1}}[F_{b-1;i,g}^{(s_{b-1,b}(\nu)+2)},F_{b-1,a;g,j}^{(r-s_{b-1,b}(\nu)-1)}]
\]
for any $1\lle g\lle \nu_{b-1}$.
\end{lem}
\begin{proof}
Again it suffices to prove the identities for the $E$'s and the identities for the $F$'s will follow by applying the anti-isomorphism $\tau$ from \eqref{taudef2}. It is equivalent to prove
\[
\overline E_{a,b;i,j}(u)=\big[E_{a,b;i,g}(u),E_{b-1;g,j}^{(s_{b-1,b}(\nu)+2)}\big]u^{-(s_{b-1,b}(\nu)-1}.
\]
We prove it by induction on $b\gge a+2$. For the base case $b=a+2$, we have by \cite[(5.11)]{Pe21} that
\begin{align*}
&\big[E_{a,b;i,g}(u),E_{b-1;g,j}^{(s_{b-1,b}(\nu)+2)}\big]-\big[E_{a,b;i,g}(u),E_{b-1;g,j}^{(s_{b-1,b}(\nu)+1)}\big]u\\&=
\big[E_{a,b;i,g}^{(s_{a,b-1}(\nu)+1)},E_{b-1;g,j}^{(s_{b-1,b}(\nu)+1)}\big]u^{-s_{b-1,b}(\nu)}-(-1)^{|g|_{b-1}}\sum_{h=1}^{\nu_{b-1}}E_{a,b-1;i,h}(u)E_{b-1;h,j}^{(s_{b-1,b}(\nu)+1)}.
\end{align*}
Multiplying by $u^{-s_{b-1,b}(\nu)-1}$ and applying Lemma \ref{lem:BK08-2.2}, we obtain
\begin{align*}
\big[E_{a,b;i,g}(u),&\,E_{b-1;g,j}^{(s_{b-1,b}(\nu)+2)}\big]u^{-s_{b-1,b}(\nu)-1}=(-1)^{|g|_{b-1}} E_{a,b;i,j}(u) \\&
-(-1)^{|g|_{b-1}}E_{a,b;i,j}^{(s_{a,b}(\nu)+1)}u^{-s_{a,b}(\nu)-1}-(-1)^{|g|_{b-1}}\sum_{h=1}^{\nu_{b-1}}E_{a,b-1;i,h}(u)E_{b-1;h,j}^{(s_{b-1,b}(\nu)+1)}.
\end{align*}
It follows from \eqref{eq:barE-def} that the RHS is exactly $(-1)^{|g|_{b-1}} \overline E_{a,b;i,j}(u)$.

Now suppose that $b>a+2$. Using Lemma \ref{lem:BK08-2.2}, the relations \cite[(5.11),~(5.13)]{Pe21}, and the induction hypothesis, we have
\begin{align*}
&[E_{a,b-1;i,g}(u),E_{b-1;g,j}^{(s_{b-1,b}(\nu)+2)}]u^{-1}\\
=~&(-1)^{|1|_{b-2}}\big[[E_{a,b-2;i,1}(u),E_{b-2;i,g}^{(s_{b-2,b-1}(\nu)+1)}],E_{b-1;g,j}^{(s_{b-1,b}(\nu)+2)}\big]u^{-s_{b-2,b-1}(\nu)-1}\\
=~&(-1)^{|1|_{b-2}}\big[E_{a,b-2;i,1}(u),[E_{b-2;i,g}^{(s_{b-2,b-1}(\nu)+1)},E_{b-1;g,j}^{(s_{b-1,b}(\nu)+2)}]\big]u^{-s_{b-2,b-1}(\nu)-1}\\
=~&(-1)^{|1|_{b-2}}\big[E_{a,b-2;i,1}(u),[E_{b-2;i,g}^{(s_{b-2,b-1}(\nu)+2)},E_{b-1;g,j}^{(s_{b-1,b}(\nu)+1)}]\big]u^{-s_{b-2,b-1}(\nu)-1}\\
&-(-1)^{|1|_{b-2}+|g|_{b-1}}\sum_{h=1}^{\nu_{b-1}}\big[E_{a,b-2;i,1}(u),E_{b-2;1,h}^{(s_{b-2,b-1}(\nu)+1)}E_{b-1;h,j}^{(s_{b-1,b}(\nu)+1)}\big]u^{-s_{b-2,b-1}(\nu)-1}\\
=~&(-1)^{|1|_{b-2}}\big[[E_{a,b-2;i,1}(u),E_{b-2;i,g}^{(s_{b-2,b-1}(\nu)+2)}],E_{b-1;g,j}^{(s_{b-1,b}(\nu)+1)}\big]u^{-s_{b-2,b-1}(\nu)-1}\\
&-(-1)^{|1|_{b-2}+|g|_{b-1}}\sum_{h=1}^{\nu_{b-1}}\big[E_{a,b-2;i,1}(u),E_{b-2;1,h}^{(s_{b-2,b-1}(\nu)+1)}\big]E_{b-1;h,j}^{(s_{b-1,b}(\nu)+1)}u^{-s_{b-2,b-1}(\nu)-1}\\
=~&\big[\overline E_{a,b-1;i,g}(u),E_{b-1;g,j}^{(s_{b-1,b}(\nu)+1)}\big]-(-1)^{|g|_{b-1}}\sum_{h=1}^{\nu_{b-1}}E_{a,b-1;i,h}(u)E_{b-1;h,j}^{(s_{b-1,b}(\nu)+1)}u^{-1}.
\end{align*}
Now multiplying both sides by $u^{-s_{b-1,b}(\nu)}$ and using \eqref{eq:barE-def} together with Lemma \ref{lem:BK08-2.2} completes the proof. 
\end{proof}

\subsection{Finishing the proof}
Now we are in the position to prove that $T_{i,j}^{(r)}$ in $W(\pi)$ is zero whenever $r>s_{i,j}+p_{\min(i,j)}$. 

Recall the fixed choice of $k$ from \eqref{eq:k-def} and other notations from Section \ref{sec:W-alg}. Given $k\lle \theta\lle \ell$, let $\pi_{\theta}$ denote the pyramid $(q_1,\dots,q_{\theta})$ of level $\theta$, i.e., the first $\theta$ columns of $\pi$. Let $\sigma_\theta$ be the shift matrix for $\pi_\theta$ defined by \eqref{eq:sigma-def} with the chosen $k$. Denote by $\g_{\theta}$ the Lie superalgebra $\gl_{m_1+\dots+m_\theta|n_1+\dots+n_\theta}$ and regard it as the subsuperalgebra of $\g$ via the canonical embedding. Let $\mathtt I_\theta:\rU(\g_\theta)\hookrightarrow \rU(\g)$ be the induced embedding for the universal enveloping superalgebras. We consider elements both of $W(\pi)\subset \rU(\g)$ and of $W(\pi_\theta)\subset \rU(\g_{\theta})\subset \rU(\g)$. To avoid the confusion, we proceed the latter by the embedding $\mathtt I_\theta$. For example, the notation $\mathtt I_{\ell-1}(\overline E_{a,b;i,j}^{(r)})$ in the lemma below  means the image of the elements $\overline E_{a,b;i,j}^{(r)}$ of $W(\pi_{\ell-1})$ under the embedding $\mathtt I_{\ell-1}$. We always work relative to the minimal admissible shape $\nu=(\nu_1,\dots,\nu_z)$ for $\sigma$.

\begin{lem}\label{lem:BK08-3.3}
Assume that $q_1\gge q_{\ell}$ and $k<\ell$. Then for all meaningful $a,b,i,j$ and $r$, we have
\begin{align}
    D_{a;i,j}^{(r)}&=\mathtt I_{\ell-1}(D_{a;i,j}^{(r)})+\delta_{a,z}\Big(\sum_{h=1}^{\nu_z} (-1)^{|h|_z}\mathtt I_{\ell-1}(D_{z;i,h}^{(r-1)})e_{q_1+\dots+q_{\ell-1}+h,q_1+\dots+q_{\ell-1}+j}\notag\\&
    \hskip4.5cm -\big[\mathtt I_{\ell-1}(D_{z;i,j}^{(r-1)}),e_{q_1+\dots+q_{\ell-1}+j-q_{\ell},q_1+\dots+q_{\ell-1}+j}\big]
    \Big),\label{eq:app5}\\
    E_{a,z;i,j}^{(r)}&=\mathtt I_{\ell-1}(\overline E_{a,z;i,j}^{(r)})+\Big(\sum_{h=1}^{\nu_z} (-1)^{|h|_z}\mathtt I_{\ell-1}(E_{a,z;i,h}^{(r-1)})e_{q_1+\dots+q_{\ell-1}+h,q_1+\dots+q_{\ell-1}+j}\notag\\&
    \hskip4.5cm -\big[\mathtt I_{\ell-1}(E_{a,z;i,j}^{(r-1)}),e_{q_1+\dots+q_{\ell-1}+j-q_{\ell},q_1+\dots+q_{\ell-1}+j}\big]
    \Big),\label{eq:app2}\\
    E_{a,b;i,j}^{(r)}&=\mathtt I_{\ell-1}(E_{a,b;i,j}^{(r)}),\hskip 8.1cm \text{ for }b<z,\label{eq:app3}\\
    F_{b,a;i,j}^{(r)}&=\mathtt I_{\ell-1}(F_{b,a;i,j}^{(r)}).\label{eq:app4}
\end{align}
\end{lem}
\begin{proof}
The equations \eqref{eq:app5}, \eqref{eq:app2} for $a=z-1$, and \eqref{eq:app3}--\eqref{eq:app4} for the special case $b=a+1$ follow from \cite[Lem.~10.2]{Pe21}. For the cases $b>a+1$, \eqref{eq:app3}--\eqref{eq:app4} are proved with the help of Lemma \ref{lem:BK08-2.2} by a simple induction on $b-a$. 

Then we consider the difficult \eqref{eq:app2}. Using \cite[Thm.~9.2 \& (9.5)]{Pe21} and Lemma \ref{lem:BK08-2.2}, it is not hard to see that 
\begin{align*}
&\big[\mathtt I_{\ell-1}(E_{a,z-1;i,g}^{(r)}),e_{q_1+\dots+q_{\ell-1}+j-q_{\ell},q_1+\dots+q_{\ell-1}+j}\big]\\
=\,&\big[\mathtt I_{\ell-1}(E_{a,z-1;i,g}^{(r)}),e_{q_1+\dots+q_{\ell-1}+h,q_1+\dots+q_{\ell-1}+j}\big]=0,
\end{align*}
for any $1\lle g\lle \nu_{z-1}$. Thus, \eqref{eq:app2} follows from the above observation and Lemma \ref{lem:BK08-2.3}.
\end{proof}

\begin{lem}\label{lem:BK08-3.4}
Assume that $q_1\gge q_\ell$ and $k<\ell$. We have
\begin{enumerate}
    \item for all $1\lle i\lle m+n$, $1\lle j\lle m+n-q_{\ell}$ and $r>0$ we have $T_{i,j}^{(r)}=\mathtt I_{\ell-1}(T_{i,j}^{(r)})$;
    \item for all $1\lle i\lle m+n$, $ j>m+n-q_{\ell}$ and $r>0$ we have 
\begin{align*}
T_{i,j}^{(r)}&=\mathtt I_{\ell-1}(T_{i,j}^{(r)})-\sum_{\substack{1\lle h\lle m+n-q_{\ell}\\ s_{h,j}\lle r}}  \mathtt I_{\ell-1}(T_{i,h}^{(r-s_{h,j})})\mathtt I_{\ell-1}(T_{h,j}^{(s_{h,j})})\\
&\hskip0.65cm  +\sum_{m+n-q_{\ell}<h\lle m+n}(-1)^{|h|}\mathtt I_{\ell-1}(T_{i,h}^{(r-1)})e_{q_1+\dots+q_{\ell}+h-m-n,q_1+\dots+q_{\ell}+j-m-n}
\\
&\hskip4.15cm -\big[\mathtt I_{\ell-1}(T_{i,j}^{(r-1)}),e_{q_1+\dots+q_{\ell}+j-m-n,q_1+\dots+q_{\ell}+j-m-n}\big].
\end{align*}
\end{enumerate}
\end{lem}
\begin{proof}
For $1\lle a,b\lle z$, $1\lle i\lle \nu_a$, and $1\lle j\lle \nu_b$, we have by definition that
\[
T_{\nu_1+\dots+\nu_{a-1}+i,\nu_1+\dots+\nu_{b-1}+j}(u)
=\sum_{c=1}^{\min(a,b)}\sum_{s,t=1}^{\nu_c} F_{a,c;i,s}(u)D_{c;s,t}(u)E_{c,b;t,j}(u).
\]
Apply Lemma \ref{lem:BK08-3.3} to rewrite the RHS along with \eqref{eq:barE-def}--\eqref{eq:barF-def}. After a straightforward calculation 
\end{proof}

\begin{proof}[Proof of Theorem \ref{thm:vanish}]
We prove it by induction on $\ell$. The base case $\ell=1$ is verified directly by definition. Using the anti-automorphism $\tau$ if necessary, we can assume that $q_1\gge q_\ell$ and hence fix a choice of $k$ such that $k<\ell$. Note that if $i,h\lle j$, then
\[
(s_{i,j}+p_i)-s_{h,j}=s_{i,h}+p_{\min(i,h)}.
\]
Thus, the induction hypothesis implies that all the terms in the RHS of Lemma \ref{lem:BK08-3.4} are zero if $r>s_{i,j}+p_{\min(i,j)}$, completing the proof.
\end{proof}

\bibliographystyle{amsalpha}
\bibliography{reference}

@incollection{Losev10,
author = {I. Losev},
title = {{Finite $W$-algebras}},
booktitle = {{Proceedings of the International Congress of Mathematicians 2010 (ICM 2010)}},
chapter = {},
year = {2010},
pages = {1281-1307},
doi = {10.1142/9789814324359_0096}
}

@incollection {EK05,
    AUTHOR = {Elashvili, A. G. and Kac, V. G.},
     TITLE = {Classification of good gradings of simple {L}ie algebras},
 BOOKTITLE = {Lie groups and invariant theory},
    SERIES = {Amer. Math. Soc. Transl. Ser. 2},
    VOLUME = {213},
     PAGES = {85--104},
 PUBLISHER = {Amer. Math. Soc., Providence, RI},
      YEAR = {2005},
      ISBN = {0-8218-3733-8},
   MRCLASS = {17B20 (17B70)},
  MRNUMBER = {2140715},
MRREVIEWER = {Rutwig\ Campoamor-Stursberg},
       DOI = {10.1090/trans2/213/05},
       URL = {https://doi.org/10.1090/trans2/213/05},
}

@article{CC24,
     AUTHOR = {Chen, C. and Cheng, S.},
     TITLE = {{W}hittaker {C}ategories of {Q}uasi-reductive {L}ie {S}uperalgebras and {P}rincipal {F}inite {W}-superalgebras},
   JOURNAL = {Transform. Groups},
  FJOURNAL = {Transformation Groups},
      YEAR = {2024},
      ISBN = {1531-586X},
       DOI = {10.1007/s00031-024-09887-8},
       URL = {https://doi.org/10.1007/s00031-024-09887-8},
}

@article{ZS16,
	author = {Zeng, Y. and Shu, B.},
	title = {{O}n {K}ac-{W}eisfeiler modules for general and special linear {L}ie superalgebras},
	JOURNAL = {Israel J. Math.},
        FJOURNAL = {Israel J. of Mathematics},
        VOLUME = {214},
        YEAR = {2016},
        NUMBER = {1},
        PAGES = {471--490},
        DOI = {10.1007/s11856-016-1338-1},
}

@article {PS16,
    AUTHOR = {Poletaeva, E. and Serganova, V.},
     TITLE = {{O}n {K}ostant's theorem for the {L}ie superalgebra {Q}(n)},
   JOURNAL = {Adv. Math.},
  FJOURNAL = {Advances in Mathematics},
    VOLUME = {300},
      YEAR = {2016},
     PAGES = {320-359},
      ISSN = {0001-8708},
       DOI = {https://doi.org/10.1016/j.aim.2016.03.021},
       URL = {https://www.sciencedirect.com/science/article/pii/S0001870816001250},
}

@article {PS17,
    AUTHOR = {Poletaeva, E. and Serganova, V.},
     TITLE = {{O}n the finite {W}-algebra for the {L}ie superalgebra {Q(N)} in the non-regular case},
   JOURNAL = {J. Math. Phys.},
  FJOURNAL = {Journal of Mathematical Physics},
    VOLUME = {58},
    NUMBER = {11},
      YEAR = {2017},
     PAGES = {111701},
      ISSN = {0022-2488},
       DOI = {10.1063/1.4993709},
       URL = {https://doi.org/10.1063/1.4993709},
}

@article{BR03,
	author = {Briot, C. and Ragoucy, E.},
	title = {$\mathcal{W}$-superalgebras as truncations of super-{Y}angians},
	JOURNAL = {J. Phys. A},
        FJOURNAL = {Journal of Physics A: Mathematical and General},
        VOLUME = {36},
        YEAR = {2003},
        NUMBER = {4},
        PAGES = {1057},
        DOI = {10.1088/0305-4470/36/4/314},
}

@article {WZ09,
    AUTHOR = {Wang, W. and Zhao, L.},
     TITLE = {{R}epresentations of {L}ie superalgebras in prime characteristic {I}},
     JOURNAL = {Proc. Lond. Math. Soc.},
     FJOURNAL = {Proceedings of the London Mathematical Society},
     VOLUME = {99},
     YEAR = {2009},
     NUMBER = {1},
     PAGES = {145-167},
     DOI = {https://doi.org/10.1112/plms/pdn057},
}

@article {Wa11,
    AUTHOR = {Wang, W.},
     TITLE = {{N}ilpotent orbits and finite $W$-algebras},
	JOURNAL = {Fields Inst. Commun.},
        FJOURNAL = {Fields Institute Communications},
        VOLUME = {59},
        YEAR = {2011},
        PAGES = {71-105},
        DOI = {https://doi.org/10.1090/fic/059},
}

@article{ZS25,
	author = {Zeng, Y. and Shu, B.},
	title = {{H}ighest {W}eight {T}heory for {M}inimal {F}inite {$W$}-{S}uperalgebras and {R}elated {W}hittaker {C}ategories},
	JOURNAL = {Publ. Res. Inst. Math. Sci.},
        FJOURNAL = {Publications of the Research Institute for Mathematical Sciences},
        VOLUME = {61},
        YEAR = {2025},
        NUMBER = {1},
        PAGES = {53-137},
        DOI = {DOI 10.4171/PRIMS/61-1-2},
}

@article{RS99,
	author = {Ragoucy, E. and Sorba, P.},
	title = {{Y}angian {R}ealisations from {F}inite {W}-{A}lgebras},
	JOURNAL = {Comm. Math. Phys.},
        FJOURNAL = {Communications in Mathematical Physics},
        VOLUME = {203},
        YEAR = {1999},
        NUMBER = {3},
        PAGES = {551-572},
        ISBN = {1432-0916},
        DOI = {10.1007/s002200050034},
}

@article {KTWWY19,
    AUTHOR = {Kamnitzer, J. and Tingley, P. and Webster, B. and
              Weekes, A. and Yacobi, O.},
     TITLE = {Highest weights for truncated shifted {Y}angians and product
              monomial crystals},
   JOURNAL = {J. Comb. Algebra},
  FJOURNAL = {Journal of Combinatorial Algebra},
    VOLUME = {3},
      YEAR = {2019},
    NUMBER = {3},
     PAGES = {237--303},
      ISSN = {2415-6302,2415-6310},
   MRCLASS = {16G20 (05E10 16T99 17B10)},
  MRNUMBER = {4011667},
MRREVIEWER = {Aleksandr\ Panov},
       DOI = {10.4171/JCA/32},
       URL = {https://doi.org/10.4171/JCA/32},
}

@article {MR04,
    AUTHOR = {Molev, A. and Retakh, V.},
     TITLE = {Quasideterminants and {C}asimir elements for the general linear Lie superalgebra},
   JOURNAL = {Int. Math. Res. Notices.},
  FJOURNAL = {International Mathematics Research Notices},
    VOLUME = {2004},
      YEAR = {2004},
    NUMBER = {13},
     PAGES = {611–619},
      ISSN = {1073-7928},
       DOI = {10.1155/S1073792804132935},
       URL = {https://academic.oup.com/imrn/article-pdf/2004/13/611/1824935/2004-13-611.pdf},
}

@article {Lo11,
    AUTHOR = {Losev, I.},
     TITLE = {1-{D}imensional representations and parabolic induction for {W}-algebras},
   JOURNAL = {Adv. Math.},
  FJOURNAL = {Advances in Mathematics},
    VOLUME = {226},
      YEAR = {2011},
    NUMBER = {6},
     PAGES = {4841-4883},
      ISSN = {0001-8708},
       DOI = {https://doi.org/10.1016/j.aim.2010.12.021},
       URL = {https://www.sciencedirect.com/science/article/pii/S0001870810004548},
}

@article{Pr02,
	author = {Premet, A.},
    title = {Special Transverse Slices and Their Enveloping Algebras},
    JOURNAL = {Adv. Math.},
    FJOURNAL = {Advances in Mathematics},
    volume = {170},
    year = {2002},
    number = {1},
    pages = {1-55},
    issn = {0001-8708},
    doi = {https://doi.org/10.1006/aima.2001.2063},
    url ={https://www.sciencedirect.com/science/article/pii/S0001870801920638},
}

@article{Pr95,
	author = {Premet, A.},
	title = {Irreducible representations of Lie algebras of reductive groups and the {K}ac-{W}eisfeiler conjecture},
	journal = {Invent. Math.},
	Fjournal = {Inventiones mathematicae},
	volume = {121},
	year = {1995},
    number = {1},
	pages = {79--117},
	isbn = {1432-1297},
	doi = {10.1007/BF01884291},
	url = {https://doi.org/10.1007/BF01884291},
}

@article{Ko78,
	author = {Kostant, B.},
	title = {On {W}hittaker vectors and representation theory},
	journal = {Invent. Math.},
	Fjournal = {Inventiones mathematicae},
    volume = {48},
	year = {1978},
    number = {2},
	pages = {101--184},
	isbn = {1432-1297},
    doi = {10.1007/BF01390249},
	url = {https://doi.org/10.1007/BF01390249},
}

@article {GGRW05,
    AUTHOR = {Gelfand, I. and Gelfand, S. and Retakh, V. and
              Wilson, R.},
     TITLE = {Quasideterminants},
   JOURNAL = {Adv. Math.},
  FJOURNAL = {Advances in Mathematics},
    VOLUME = {193},
      YEAR = {2005},
    NUMBER = {1},
     PAGES = {56--141},
      ISSN = {0001-8708,1090-2082},
   MRCLASS = {05E05 (15A15 16W30)},
  MRNUMBER = {2132761},
MRREVIEWER = {K.\ Chandrasekhara Rao},
       DOI = {10.1016/j.aim.2004.03.018},
       URL = {https://doi.org/10.1016/j.aim.2004.03.018},
}

@article {Gow05,
    AUTHOR = {Gow, L.},
     TITLE = {On the {Y}angian {$Y(\germ{gl}_{m|n})$} and its quantum
              {B}erezinian},
   JOURNAL = {Czechoslovak J. Phys.},
  FJOURNAL = {Czechoslovak Journal of Physics},
    VOLUME = {55},
      YEAR = {2005},
    NUMBER = {11},
     PAGES = {1415--1420},
      ISSN = {0011-4626,1572-9486},
   MRCLASS = {81R12 (17B81 81R50)},
  MRNUMBER = {2223829},
MRREVIEWER = {Jonathan\ Brundan},
       DOI = {10.1007/s10582-006-0019-4},
       URL = {https://doi.org/10.1007/s10582-006-0019-4},
}

@article{CW26,
  title={{Simple character formulas for finite $W$-superalgebras of type A}},
  author={Cheng, S.-J. and Wang, W.},
  journal={arXiv preprint \arxiv{2603.01373}},
  year={2026}
}

@book {Mus12,
    AUTHOR = {Musson, I.},
     TITLE = {{Lie superalgebras and enveloping algebras}},
    SERIES = {Graduate Studies in Mathematics},
    VOLUME = {131},
 PUBLISHER = {American Mathematical Society, Providence, RI},
      YEAR = {2012},
     PAGES = {xx+488},
      ISBN = {978-0-8218-6867-6},
   MRCLASS = {17-02 (16S30 17B35)},
  MRNUMBER = {2906817},
MRREVIEWER = {Aleksandr\ Nikolaevich\ Sergeev},
       DOI = {10.1090/gsm/131},
       URL = {https://doi.org/10.1090/gsm/131},
}

@article {Zha14,
    AUTHOR = {Zhao, L.},
     TITLE = {Finite {$W$}-superalgebras for queer {L}ie superalgebras},
   JOURNAL = {J. Pure Appl. Algebra},
  FJOURNAL = {Journal of Pure and Applied Algebra},
    VOLUME = {218},
      YEAR = {2014},
    NUMBER = {7},
     PAGES = {1184--1194},
      ISSN = {0022-4049,1873-1376},
   MRCLASS = {17B35 (17B10)},
  MRNUMBER = {3168490},
MRREVIEWER = {Aleksandr\ Nikolaevich\ Sergeev},
       DOI = {10.1016/j.jpaa.2013.11.012},
       URL = {https://doi.org/10.1016/j.jpaa.2013.11.012},
}

@article {GG02,
    AUTHOR = {Gan, W. L. and Ginzburg, V.},
     TITLE = {Quantization of {S}lodowy slices},
   JOURNAL = {Int. Math. Res. Not.},
  FJOURNAL = {International Mathematics Research Notices},
      YEAR = {2002},
    NUMBER = {5},
     PAGES = {243--255},
      ISSN = {1073-7928,1687-0247},
   MRCLASS = {53D20 (17B35 17B63 53D50)},
  MRNUMBER = {1876934},
MRREVIEWER = {William\ M.\ McGovern},
       DOI = {10.1155/S107379280210609X},
       URL = {https://doi.org/10.1155/S107379280210609X},
}

@book {Bha97,
    AUTHOR = {Bhatia, R.},
     TITLE = {Matrix analysis},
    SERIES = {Graduate Texts in Mathematics},
    VOLUME = {169},
 PUBLISHER = {Springer-Verlag, New York},
      YEAR = {1997},
     PAGES = {xii+347},
      ISBN = {0-387-94846-5},
   MRCLASS = {15-02 (47-02)},
  MRNUMBER = {1477662},
MRREVIEWER = {R.\ J.\ Bumcrot},
       DOI = {10.1007/978-1-4612-0653-8},
       URL = {https://doi.org/10.1007/978-1-4612-0653-8},
}

@book {CW12,
    AUTHOR = {Cheng, S.-J. and Wang, W.},
     TITLE = {Dualities and representations of {L}ie superalgebras},
    SERIES = {Graduate Studies in Mathematics},
    VOLUME = {144},
 PUBLISHER = {American Mathematical Society, Providence, RI},
      YEAR = {2012},
     PAGES = {xviii+302},
      ISBN = {978-0-8218-9118-6},
   MRCLASS = {17B10},
  MRNUMBER = {3012224},
MRREVIEWER = {Aleksandr\ Nikolaevich\ Sergeev},
       DOI = {10.1090/gsm/144},
       URL = {https://doi.org/10.1090/gsm/144},
}

@article {CH23,
    AUTHOR = {Chang, H. and Hu, H.},
     TITLE = {A note on the center of the super {Y}angian {$Y_{M|N}(\germ
              s)$}},
   JOURNAL = {J. Algebra},
  FJOURNAL = {Journal of Algebra},
    VOLUME = {633},
      YEAR = {2023},
     PAGES = {648--665},
      ISSN = {0021-8693,1090-266X},
   MRCLASS = {17B37},
  MRNUMBER = {4620131},
       DOI = {10.1016/j.jalgebra.2023.06.025},
       URL = {https://doi.org/10.1016/j.jalgebra.2023.06.025},
}

@article{LPTTW25,
  title={{Shifted twisted Yangians and finite $ W $-algebras of classical type}},
  author={Lu, K. and Peng, Y.-N. and Tappeiner, L. and Topley, L. and Wang, W.},
  journal={arXiv preprint \arxiv{2505.03316}},
  year={2025}
}

@article {Zh96,
    AUTHOR = {Zhang, R. B.},
     TITLE = {The {${\rm gl}(M|N)$} super {Y}angian and its
              finite-dimensional representations},
   JOURNAL = {Lett. Math. Phys.},
  FJOURNAL = {Letters in Mathematical Physics},
    VOLUME = {37},
      YEAR = {1996},
    NUMBER = {4},
     PAGES = {419--434},
      ISSN = {0377-9017,1573-0530},
   MRCLASS = {17B10 (81R50)},
  MRNUMBER = {1401045},
MRREVIEWER = {Hedi\ Benamor},
       DOI = {10.1007/BF00312673},
       URL = {https://doi.org/10.1007/BF00312673},
}

@article {Zh95,
    AUTHOR = {Zhang, R. B.},
     TITLE = {Representations of super {Y}angian},
   JOURNAL = {J. Math. Phys.},
  FJOURNAL = {Journal of Mathematical Physics},
    VOLUME = {36},
      YEAR = {1995},
    NUMBER = {7},
     PAGES = {3854--3865},
      ISSN = {0022-2488,1089-7658},
   MRCLASS = {81R10 (17B70)},
  MRNUMBER = {1339907},
MRREVIEWER = {Peter\ Pre\v snajder},
       DOI = {10.1063/1.530932},
       URL = {https://doi.org/10.1063/1.530932},
}

@article {MY14,
    AUTHOR = {Mukhin, E. and Young, C. A. S.},
     TITLE = {Affinization of category {$\scr{O}$} for quantum groups},
   JOURNAL = {Trans. Amer. Math. Soc.},
  FJOURNAL = {Transactions of the American Mathematical Society},
    VOLUME = {366},
      YEAR = {2014},
    NUMBER = {9},
     PAGES = {4815--4847},
      ISSN = {0002-9947,1088-6850},
   MRCLASS = {17B37},
  MRNUMBER = {3217701},
MRREVIEWER = {Christian\ Ohn},
       DOI = {10.1090/S0002-9947-2014-06039-X},
       URL = {https://doi.org/10.1090/S0002-9947-2014-06039-X},
}

@article {HZ24,
    AUTHOR = {Hernandez, D. and Zhang, H.},
     TITLE = {Shifted {Y}angians and polynomial {$R$}-matrices},
   JOURNAL = {Publ. Res. Inst. Math. Sci.},
  FJOURNAL = {Publications of the Research Institute for Mathematical
              Sciences},
    VOLUME = {60},
      YEAR = {2024},
    NUMBER = {1},
     PAGES = {1--69},
      ISSN = {0034-5318,1663-4926},
   MRCLASS = {20G42 (16T25 81R50)},
  MRNUMBER = {4803333},
MRREVIEWER = {Run-Qiang\ Jian},
       DOI = {10.4171/prims/60-1-1},
       URL = {https://doi.org/10.4171/prims/60-1-1},
}

@article {FPT22,
    AUTHOR = {Frassek, R. and Pestun, V. and Tsymbaliuk, A.},
     TITLE = {Lax matrices from antidominantly shifted {Y}angians and
              quantum affine algebras: {A}-type},
   JOURNAL = {Adv. Math.},
  FJOURNAL = {Advances in Mathematics},
    VOLUME = {401},
      YEAR = {2022},
     PAGES = {Paper No. 108283, 73},
      ISSN = {0001-8708,1090-2082},
   MRCLASS = {17B37 (81R10)},
  MRNUMBER = {4394682},
MRREVIEWER = {Dmitry\ V.\ Artamonov},
       DOI = {10.1016/j.aim.2022.108283},
       URL = {https://doi.org/10.1016/j.aim.2022.108283},
}

@article {BFN18,
    AUTHOR = {Braverman, A. and Finkelberg, M. and Nakajima,
              H.},
     TITLE = {Towards a mathematical definition of {C}oulomb branches of
              3-dimensional {$\mathcal N=4$} gauge theories, {II}},
   JOURNAL = {Adv. Theor. Math. Phys.},
  FJOURNAL = {Advances in Theoretical and Mathematical Physics},
    VOLUME = {22},
      YEAR = {2018},
    NUMBER = {5},
     PAGES = {1071--1147},
      ISSN = {1095-0761,1095-0753},
   MRCLASS = {57R57 (14J33 14N35 16G20 17B67 81T13)},
  MRNUMBER = {3952347},
MRREVIEWER = {Dave\ Auckly},
       DOI = {10.4310/ATMP.2018.v22.n5.a1},
       URL = {https://doi.org/10.4310/ATMP.2018.v22.n5.a1},
}

@incollection {FR99,
    AUTHOR = {Frenkel, E. and Reshetikhin, N.},
     TITLE = {The {$q$}-characters of representations of quantum affine
              algebras and deformations of {$\scr W$}-algebras},
 BOOKTITLE = {Recent developments in quantum affine algebras and related
              topics ({R}aleigh, {NC}, 1998)},
    SERIES = {Contemp. Math.},
    VOLUME = {248},
     PAGES = {163--205},
 PUBLISHER = {Amer. Math. Soc., Providence, RI},
      YEAR = {1999},
      ISBN = {0-8218-1199-1},
   MRCLASS = {17B37 (05E15 17B68)},
  MRNUMBER = {1745260},
       DOI = {10.1090/conm/248/03823},
       URL = {https://doi.org/10.1090/conm/248/03823},
}

@article {Kn95,
    AUTHOR = {Knight, H.},
     TITLE = {Spectra of tensor products of finite-dimensional
              representations of {Y}angians},
   JOURNAL = {J. Algebra},
  FJOURNAL = {Journal of Algebra},
    VOLUME = {174},
      YEAR = {1995},
    NUMBER = {1},
     PAGES = {187--196},
      ISSN = {0021-8693,1090-266X},
   MRCLASS = {17B37 (81R50)},
  MRNUMBER = {1332866},
MRREVIEWER = {Preeti\ Parashar},
       DOI = {10.1006/jabr.1995.1123},
       URL = {https://doi.org/10.1006/jabr.1995.1123},
}

@article {LM21,
    AUTHOR = {Lu, K. and Mukhin, E.},
     TITLE = {Jacobi-{T}rudi identity and {D}rinfeld functor for super
              {Y}angian},
   JOURNAL = {Int. Math. Res. Not. IMRN},
  FJOURNAL = {International Mathematics Research Notices. IMRN},
      YEAR = {2021},
    NUMBER = {21},
     PAGES = {16751--16810},
      ISSN = {1073-7928,1687-0247},
   MRCLASS = {17B10 (16T25 81R12 82B23)},
  MRNUMBER = {4338233},
MRREVIEWER = {Weiqiang\ Wang},
       DOI = {10.1093/imrn/rnab023},
       URL = {https://doi.org/10.1093/imrn/rnab023},
}

@article {Mo07,
    AUTHOR = {Molev, A.},
     TITLE = {Yangians and {C}lassical {L}ie {A}lgebras},
   JOURNAL = {Math. Surveys Monogr.},
  FJOURNAL = {Mathematical Surveys and Monographs},
 PUBLISHER = {Amer. Math. Soc., Providence, RI},
    VOLUME = {143},
      YEAR = {2007},
      ISSN = {0076-5376},
      ISBN = {978-0-8218-4374-1},
   MRCLASS = {17B37},
  MRNUMBER = {2355506},
MRREVIEWER = {Ian M.\ Musson},
       URL = {https://bookstore.ams.org/surv-143},
}

@article {Mo22,
    AUTHOR = {Molev, A.},
     TITLE = {Odd reflections in the {Y}angian associated with
              {$\germ{gl}(m|n)$}},
   JOURNAL = {Lett. Math. Phys.},
  FJOURNAL = {Letters in Mathematical Physics},
    VOLUME = {112},
      YEAR = {2022},
    NUMBER = {1},
     PAGES = {Paper No. 8, 15},
      ISSN = {0377-9017,1573-0530},
   MRCLASS = {17B37},
  MRNUMBER = {4367922},
MRREVIEWER = {Yung-Ning\ Peng},
       DOI = {10.1007/s11005-021-01501-2},
       URL = {https://doi.org/10.1007/s11005-021-01501-2},
}

@article {Pe21,
    AUTHOR = {Peng, Y.-N.},
     TITLE = {Finite {$W$}-superalgebras via super {Y}angians},
   JOURNAL = {Adv. Math.},
  FJOURNAL = {Advances in Mathematics},
    VOLUME = {377},
      YEAR = {2021},
     PAGES = {Paper No. 107459, 60},
      ISSN = {0001-8708,1090-2082},
   MRCLASS = {17B60 (16T05 17B35 17B38)},
  MRNUMBER = {4186002},
MRREVIEWER = {Plamen\ Koshlukov},
       DOI = {10.1016/j.aim.2020.107459},
       URL = {https://doi.org/10.1016/j.aim.2020.107459},
}

@article {Lu22,
    AUTHOR = {Lu, K.},
     TITLE = {A note on odd reflections of super {Y}angian and {B}ethe
              ansatz},
   JOURNAL = {Lett. Math. Phys.},
  FJOURNAL = {Letters in Mathematical Physics},
    VOLUME = {112},
      YEAR = {2022},
    NUMBER = {2},
     PAGES = {Paper No. 29, 26},
      ISSN = {0377-9017,1573-0530},
   MRCLASS = {17B37 (82B23)},
  MRNUMBER = {4400677},
MRREVIEWER = {Yung-Ning\ Peng},
       DOI = {10.1007/s11005-022-01524-3},
       URL = {https://doi.org/10.1007/s11005-022-01524-3},
}

@article {Na91,
    AUTHOR = {Nazarov, M.},
     TITLE = {Quantum {B}erezinian and the classical {C}apelli identity},
   JOURNAL = {Lett. Math. Phys.},
  FJOURNAL = {Letters in Mathematical Physics},
    VOLUME = {21},
      YEAR = {1991},
    NUMBER = {2},
     PAGES = {123--131},
      ISSN = {0377-9017,1573-0530},
   MRCLASS = {17B37 (16W55 17A70 81R50)},
  MRNUMBER = {1093523},
MRREVIEWER = {Tuan\ W.\ Chen},
       DOI = {10.1007/BF00401646},
       URL = {https://doi.org/10.1007/BF00401646},
}

@article {Na20,
    AUTHOR = {Nazarov, M.},
     TITLE = {Yangian of the {G}eneral {L}inear {L}ie {S}uperalgebra,},
   JOURNAL = {SIGMA Symmetry Integrability Geom. Methods Appl.},
  FJOURNAL = {SIGMA. Symmetry, Integrability and Geometry. Methods and Applications},
    VOLUME = {16},
      YEAR = {2020},
    NUMBER = {112},
     PAGES = {Paper No. 112, 24 pages},
      ISSN = {1815-0659},
   MRCLASS = {17B37 (16T20 81R50)},
  MRNUMBER = {4170710},
MRREVIEWER = {Xin\ Tang},
       DOI = {10.3842/SIGMA.2020.112},
       URL = {https://doi.org/10.3842/SIGMA.2020.112},
}

@article {Pe16,
    AUTHOR = {Peng, Y.-N.},
     TITLE = {Parabolic presentations of the super {Y}angian
              {$Y(\germ{gl}_{M|N})$} associated with arbitrary 01-sequences},
   JOURNAL = {Comm. Math. Phys.},
  FJOURNAL = {Communications in Mathematical Physics},
    VOLUME = {346},
      YEAR = {2016},
    NUMBER = {1},
     PAGES = {313--347},
      ISSN = {0010-3616,1432-0916},
   MRCLASS = {17B67},
  MRNUMBER = {3528423},
       DOI = {10.1007/s00220-015-2548-9},
       URL = {https://doi.org/10.1007/s00220-015-2548-9},
}

@article {KU22,
    AUTHOR = {Kodera, R. and Ueda, M.},
     TITLE = {Coproduct for affine {Y}angians and parabolic induction for
              rectangular {$W$}-algebras},
   JOURNAL = {Lett. Math. Phys.},
  FJOURNAL = {Letters in Mathematical Physics},
    VOLUME = {112},
      YEAR = {2022},
    NUMBER = {1},
     PAGES = {Paper No. 3, 37},
      ISSN = {0377-9017,1573-0530},
   MRCLASS = {17B37 (17B69)},
  MRNUMBER = {4362474},
MRREVIEWER = {Dmitry\ V.\ Artamonov},
       DOI = {10.1007/s11005-021-01500-3},
       URL = {https://doi.org/10.1007/s11005-021-01500-3},
}

@article {Ho12,
    AUTHOR = {Hoyt, C.},
     TITLE = {Good gradings of basic {L}ie superalgebras},
   JOURNAL = {Israel J. Math.},
  FJOURNAL = {Israel Journal of Mathematics},
    VOLUME = {192},
      YEAR = {2012},
    NUMBER = {1},
     PAGES = {251--280},
      ISSN = {0021-2172,1565-8511},
   MRCLASS = {17B70},
  MRNUMBER = {3004082},
MRREVIEWER = {Anargyros\ Fellouris},
       DOI = {10.1007/s11856-012-0023-2},
       URL = {https://doi.org/10.1007/s11856-012-0023-2},
}

@article {Ts20,
    AUTHOR = {Tsymbaliuk, A.},
     TITLE = {Shuffle algebra realizations of type {$A$} super {Y}angians
              and quantum affine superalgebras for all {C}artan data},
   JOURNAL = {Lett. Math. Phys.},
  FJOURNAL = {Letters in Mathematical Physics},
    VOLUME = {110},
      YEAR = {2020},
    NUMBER = {8},
     PAGES = {2083--2111},
      ISSN = {0377-9017,1573-0530},
   MRCLASS = {17B37 (81R10)},
  MRNUMBER = {4126874},
MRREVIEWER = {Slaven\ Ko\v zi\'c},
       DOI = {10.1007/s11005-020-01287-9},
       URL = {https://doi.org/10.1007/s11005-020-01287-9},
}

@article {Go07,
    AUTHOR = {Gow, L.},
     TITLE = {Gauss decomposition of the {Y}angian {$Y({\germ{gl}}_{m|n})$}},
   JOURNAL = {Comm. Math. Phys.},
  FJOURNAL = {Communications in Mathematical Physics},
    VOLUME = {276},
      YEAR = {2007},
    NUMBER = {3},
     PAGES = {799--825},
      ISSN = {0010-3616,1432-0916},
   MRCLASS = {17B37 (81R50)},
  MRNUMBER = {2350438},
MRREVIEWER = {Jonathan\ Brundan},
       DOI = {10.1007/s00220-007-0349-5},
       URL = {https://doi.org/10.1007/s00220-007-0349-5},
}

@article {Ge20,
    AUTHOR = {Genra, N.},
     TITLE = {Screening operators and parabolic inductions for affine
              {${\mathcal W}$}-algebras},
   JOURNAL = {Adv. Math.},
  FJOURNAL = {Advances in Mathematics},
    VOLUME = {369},
      YEAR = {2020},
     PAGES = {107179, 62},
      ISSN = {0001-8708,1090-2082},
   MRCLASS = {17B69},
  MRNUMBER = {4091897},
MRREVIEWER = {Philsang\ Yoo},
       DOI = {10.1016/j.aim.2020.107179},
       URL = {https://doi.org/10.1016/j.aim.2020.107179},
}

@article {BK08,
    AUTHOR = {Brundan, J. and Kleshchev, A.},
     TITLE = {Representations of shifted {Y}angians and finite
              {$W$}-algebras},
   JOURNAL = {Mem. Amer. Math. Soc.},
  FJOURNAL = {Memoirs of the American Mathematical Society},
    VOLUME = {196},
      YEAR = {2008},
    NUMBER = {918},
     PAGES = {viii+107},
      ISSN = {0065-9266,1947-6221},
      ISBN = {978-0-8218-4216-4},
   MRCLASS = {17B37 (17B10)},
  MRNUMBER = {2456464},
MRREVIEWER = {Serge\ M.\ Skryabin},
       DOI = {10.1090/memo/0918},
       URL = {https://doi.org/10.1090/memo/0918},
}

@article {BK06,
    AUTHOR = {Brundan, J. and Kleshchev, A.},
     TITLE = {Shifted {Y}angians and finite {$W$}-algebras},
   JOURNAL = {Adv. Math.},
  FJOURNAL = {Advances in Mathematics},
    VOLUME = {200},
      YEAR = {2006},
    NUMBER = {1},
     PAGES = {136--195},
      ISSN = {0001-8708,1090-2082},
   MRCLASS = {17B37},
  MRNUMBER = {2199632},
MRREVIEWER = {Chengming\ Bai},
       DOI = {10.1016/j.aim.2004.11.004},
       URL = {https://doi.org/10.1016/j.aim.2004.11.004},
}

@article {BBG13,
    AUTHOR = {Brown, J. and Brundan, J. and Goodwin, S. M.},
     TITLE = {Principal {$W$}-algebras for {${\rm GL}(m|n)$}},
   JOURNAL = {Algebra Number Theory},
  FJOURNAL = {Algebra \& Number Theory},
    VOLUME = {7},
      YEAR = {2013},
    NUMBER = {8},
     PAGES = {1849--1882},
      ISSN = {1937-0652,1944-7833},
   MRCLASS = {17B10 (17B35)},
  MRNUMBER = {3134037},
MRREVIEWER = {David\ Edward\ Hill},
       DOI = {10.2140/ant.2013.7.1849},
       URL = {https://doi.org/10.2140/ant.2013.7.1849},
}

@article {BG19,
    AUTHOR = {Brundan, J. and Goodwin, S. M.},
     TITLE = {Whittaker coinvariants for {${\rm GL}(m|n)$}},
   JOURNAL = {Adv. Math.},
  FJOURNAL = {Advances in Mathematics},
    VOLUME = {347},
      YEAR = {2019},
     PAGES = {273--339},
      ISSN = {0001-8708,1090-2082},
   MRCLASS = {17B10 (17B37)},
  MRNUMBER = {3916872},
MRREVIEWER = {Aleksandr\ Nikolaevich\ Sergeev},
       DOI = {10.1016/j.aim.2019.02.025},
       URL = {https://doi.org/10.1016/j.aim.2019.02.025},
}

@article {BK05,
    AUTHOR = {Brundan, J. and Kleshchev, A.},
     TITLE = {Parabolic presentations of the {Y}angian {$Y({\germ{gl}}_n)$}},
   JOURNAL = {Comm. Math. Phys.},
  FJOURNAL = {Communications in Mathematical Physics},
    VOLUME = {254},
      YEAR = {2005},
    NUMBER = {1},
     PAGES = {191--220},
      ISSN = {0010-3616,1432-0916},
   MRCLASS = {17B37},
  MRNUMBER = {2116743},
MRREVIEWER = {Sonia\ Natale},
       DOI = {10.1007/s00220-004-1249-6},
       URL = {https://doi.org/10.1007/s00220-004-1249-6},
}

@article {Pe14,
    AUTHOR = {Peng, Y.-N.},
     TITLE = {Finite {W}-{S}uperalgebras and {T}runcated {S}uper {Y}angians},
   JOURNAL = {Lett. Math. Phys.},
  FJOURNAL = {Letters in Mathematical Physics},
    VOLUME = {104},
      YEAR = {2014},
    NUMBER = {1},
     PAGES = {89--101},
   MRCLASS = {17B37},
  MRNUMBER = {3147661},
MRREVIEWER = {Liu.\ Ming},
       DOI = {10.1007/s11005-013-0656-z},
       URL = {https://doi.org/10.1007/s11005-013-0656-z},
}

\end{document}